\documentclass[11pt]{article}
\usepackage{amsfonts,amssymb,amsmath}
\usepackage{graphicx,graphics,psfrag}

\usepackage[english]{babel}
\makeatletter
\@addtoreset{equation}{section}
\makeatother

\newcounter{enunciato}[section]

\newtheorem{ittheorem}{Theorem}
\newtheorem{itlemma}{Lemma}
\newtheorem{itproposition}{Proposition}
\newtheorem{itdefinition}{Definition} 
\newtheorem{itcorollary}{Corollary} 
\newtheorem{itconjecture}{Conjecture}

\newenvironment{theorem}{\addtocounter{enunciato}{1}
\begin{ittheorem}}{\end{ittheorem}}

\newenvironment{lemma}{\addtocounter{enunciato}{1}
\begin{itlemma}}{\end{itlemma}}

\newenvironment{proposition}{\addtocounter{enunciato}{1}
\begin{itproposition}}{\end{itproposition}}

\newenvironment{corollary}{\addtocounter{enunciato}{1}
\begin{itcorollary}}{\end{itcorollary}}

\newenvironment{conjecture}{\addtocounter{enunciato}{1}
\begin{itconjecture}}{\end{itconjecture}}

\newcommand{\halmos}{\rule{1ex}{1.4ex}}

\newenvironment{proof}{\noindent {\em Proof}.\,\,}
{\hspace*{\fill}$\halmos$\medskip}

\def \ba {\begin{array}}
\def \ea {\end{array}}

\def \Z {{\mathbb Z}}
\def \R {{\mathbb R}}

\def \N {{\mathbb N}}

\def \P {{\mathbb P}}
\def \E {{\mathbb E}}

\def \cL {{\mathcal L}}
\def \cD {{\mathcal D}}
\def \cP {{\mathcal P}}

\def \cR {{\mathcal R}}

\def \cC {{\mathcal C}}
\def \cA {{\mathcal A}}
\def \cN {{\mathcal N}}
\def \cR {{\mathcal R}}

\def \cB {\mathcal B}
\def \cM {\mathcal M}

\def \wZ {\widetilde Z}
\def \wP {\widetilde {\mathcal P}}

\def \tr {{\rm tr}}

\def \Pr {{\mathbb P}}

\def \l {\lambda}

\def \g {\gamma}
\def \o {\omega}

\begin{document}
\title{A copolymer near a selective interface:\\ 
variational characterization of the free energy}

\author{
\renewcommand{\thefootnote}{\arabic{footnote}}
E.\ Bolthausen
\footnotemark[1]
\\
\renewcommand{\thefootnote}{\arabic{footnote}}
F.\ den Hollander
\footnotemark[2]\,\,\,\,\footnotemark[3]
\\
\renewcommand{\thefootnote}{\arabic{footnote}}
A.A.\ Opoku
\footnotemark[2]
}

\footnotetext[1]{
Institut f\"ur Mathematik, Universit\"at Z\"urich,
Winterthurerstrasse 190, CH-8057 Z\"urich, Switzerland.
}

\footnotetext[2]{
Mathematical Institute, Leiden University, P.O.\ Box 9512,
2300 RA Leiden, The Netherlands.
}

\footnotetext[3]{
EURANDOM, P.O.\ Box 513, 5600 MB Eindhoven, The Netherlands.
}

\maketitle

\begin{abstract}
In this paper we consider a random copolymer near a selective interface separating 
two solvents. The configurations of the copolymer are directed paths that can make 
i.i.d.\ excursions of finite length above and below the interface. The excursion 
length distribution is assumed to have a tail that is logarithmically equivalent 
to a power law with exponent $\alpha \geq 1$. The monomers carry i.i.d.\ real-valued 
types whose distribution is assumed to have zero mean, unit variance, and a finite 
moment generating function. The interaction Hamiltonian rewards matches and penalizes 
mismatches of the monomer types and the solvents, and depends on two parameters: the 
interaction strength $\beta\geq 0$ and the interaction bias $h \geq 0$. We are 
interested in the behavior of the copolymer in the limit as its length tends to 
infinity. 

The quenched free energy per monomer $(\beta,h) \mapsto g^\mathrm{que}(\beta,h)$ has a 
phase transition along a quenched critical curve $\beta \mapsto h^\mathrm{que}_c(\beta)$ 
separating a localized phase, where the copolymer stays close to the interface, from 
a delocalized phase, where the copolymer wanders away from the interface. We derive 
\emph{variational formulas} for both these quantities. We compare these variational 
formulas with their analogues for the annealed free energy per monomer $(\beta,h) 
\mapsto g^\mathrm{ann}(\beta,h)$ and the annealed critical curve $\beta \mapsto 
h^\mathrm{ann}_c(\beta)$, both of which are explicitly computable. This comparison 
leads to:
\begin{itemize}
\item[(1)] 
A proof that $g^\mathrm{que}(\beta,h)<g^\mathrm{ann}(\beta,h)$ for all $\alpha\geq 1$ 
and $(\beta,h)$ in the annealed localized phase.
\item[(2)] 
A proof that $h_c^\mathrm{ann}(\beta/\alpha)<h_c^\mathrm{que}(\beta)<h_c^\mathrm{ann}
(\beta)$ for all $\alpha>1$ and $\beta>0$.
\item[(3)]
A proof that $\liminf_{\beta \downarrow 0} h_c^\mathrm{que}(\beta)/\beta \geq K_c^*$ 
with $K_c^* = (1+\alpha)/2\alpha$ for $\alpha\geq 2$ and $K_c^* = B(\alpha)/\alpha$ 
for $1<\alpha<2$ with $B(\alpha)>1$.
\item[(4)] 
An estimate of the total number of times the copolymer visits the interface in the 
interior of the quenched delocalized phase.
\item[(5)] 
An identification of the asymptotic frequency at which the copolymer visits the 
interface in the quenched localized phase.
\end{itemize}
The copolymer model has been studied extensively in the literature. The goal of the 
present paper is to open up a window with a variational view and to address a number 
of open problems. 

\medskip\noindent
{\it AMS} 2000 {\it subject classifications.} 60F10, 60K37, 82B27.

\medskip\noindent
{\it Acknowledgment.} 
FdH thanks M.\ Birkner and F.\ Redig for fruitful discussions. EB was supported 
by SNSF-grant 20-100536/1, FdH by ERC Advanced Grant VARIS 267356, and AO by 
NWO-grant 613.000.913.

\end{abstract}


\section{Introduction and main results}
\label{S1}

In Section~\ref{S1.1} we define the model. In Sections~\ref{S1.2} and \ref{S1.3} we define 
the quenched and the annealed free energy and critical curve. In Section~\ref{S1.4} we state 
our main results, while in Section~\ref{S1.5} we place these results in the context of earlier 
work. For more background and key results in the literature, we refer the reader to 
Giacomin~\cite{Gi07}, Chapters 6--8, and den Hollander~\cite{dHo09}, Chapter 9.   


\subsection{A copolymer near a selective interface}
\label{S1.1}

Let $\omega = (\omega_k)_{k\in\N}$ be i.i.d.\ random variables with a probability distribution 
$\nu$ on $\R$ having zero mean and unit variance:
\begin{equation}
\label{nuzmuv}
\int_\R x\,\nu(dx) =0, \qquad \int_\R x^2\,\nu(dx) =1,
\end{equation}
and a finite cumulant generating function: 
\begin{equation}
\label{mgffin}
M(\lambda)=\log \int_\R e^{-\lambda x}\,\nu(dx) < \infty 
\qquad \forall\,\lambda\in\R.
\end{equation} 
Write $\P=\nu^{\otimes\N}$ to denote the distribution of $\omega$. Let
\begin{equation}
\label{Wndef}
\Pi = \big\{\pi=(k,\pi_k)_{k\in\N_0}\colon\,\pi_0=0,\,\pi_k \in \Z\,\,
\forall\,k\in\N\big\}.
\end{equation} 
denote the set of infinite directed paths on $\N_0\times\Z$ (with $\N_0=\N\cup\{0\}$).
Fix $n\in\N_0$ and $\beta,h \geq 0$. For given $\omega$, let
\begin{equation}
\label{copoldef}
H_n^{\beta,h,\omega}(\pi) 
= -\beta \sum_{k=1}^n (\omega_k+h)\,{\rm sign}(\pi_{k-1},\pi_k),
\qquad \pi \in \Pi,
\end{equation}
be the $n$-step Hamiltonian on $\Pi$, and let
\begin{equation}
\label{Pnomdef}
P_n^{\beta,h,\omega}(\pi) = \frac{1}{Z_n^{\beta,h,\omega}}\,
e^{-H_n^{\beta,h,\omega}(\pi)}\,P(\pi), 
\qquad \pi \in \Pi,
\end{equation}
be the $n$-step path measure on $\Pi$, where $P$ is any probability distribution on 
$\Pi$ under which the excursions away from the interface are i.i.d., lie with equal
probability above and below the interface, and have a length whose probability 
distribution $\rho$ on $\N$ has infinite support and a \emph{polynomial tail}
\begin{equation}
\label{rhocond}
\lim_{ {m\to\infty} \atop {\rho(m)>0} } \frac{\log\rho(m)}{\log m} = -\alpha
\mbox{ for some } \alpha \geq 1.
\end{equation}
Note that the Hamiltonian in (\ref{copoldef}) only depends on the signs of the excursions
and on their starting and ending points in $\omega$, not on their shape.

\begin{figure}[htbp] 
\vspace{-.5cm}
\begin{center}
\includegraphics[scale = 0.45]{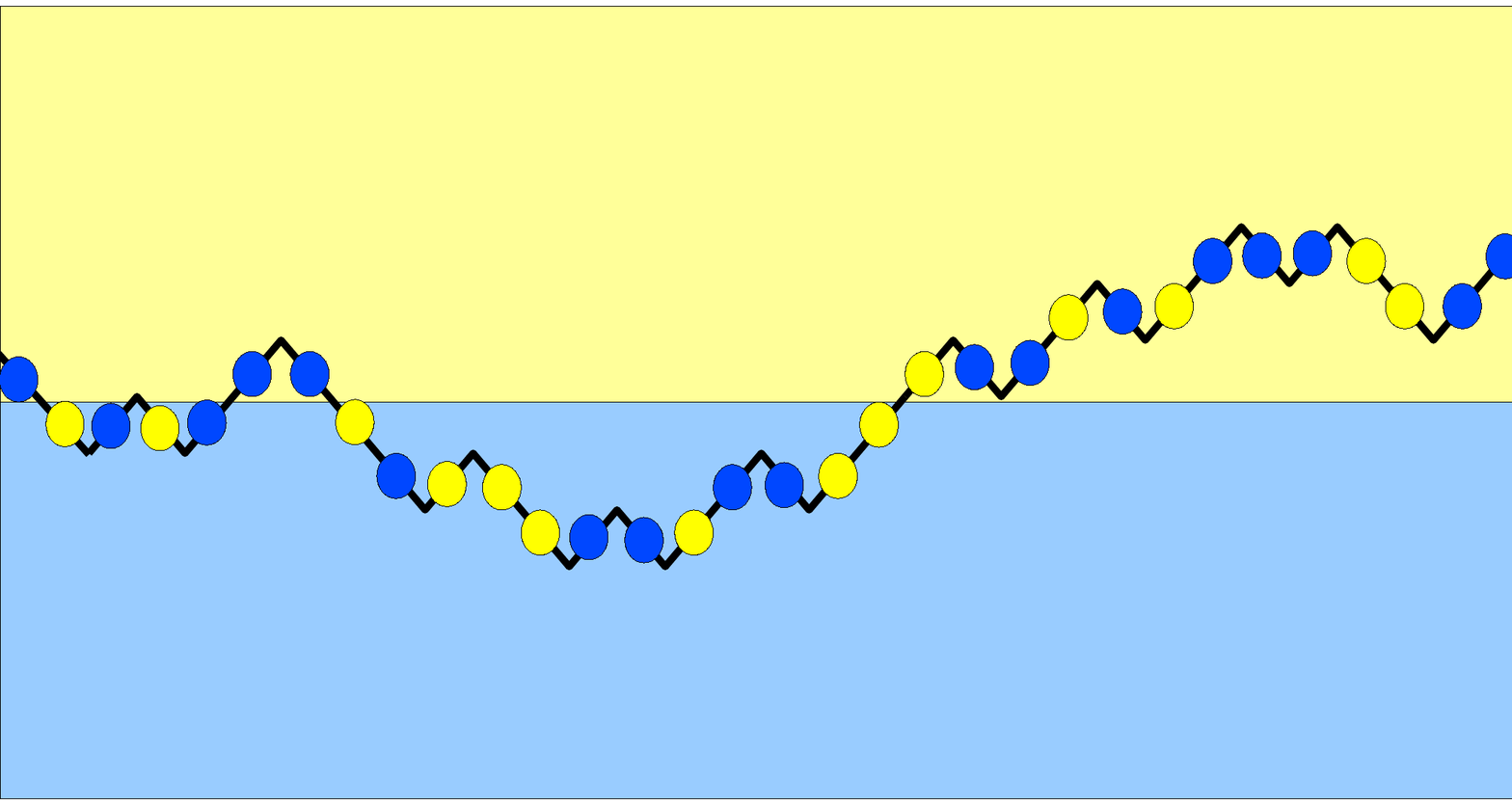}
\end{center}
\vspace{-8.2cm}
\caption{\small A directed copolymer near a linear interface. Oil in the upper
half plane and hydrophobic monomers in the polymer chain are shaded light, water 
in the lower half plane and hydrophilic monomers in the polymer chain are shaded 
dark. (Courtesy of N.\ P\'etr\'elis.)} 
\label{fig-copolex}
\end{figure}

\medskip\noindent
{\bf Example.}
For the special case where $\nu$ is the binary distribution $\nu(-1)=\nu(+1)=\tfrac12$
and $P$ is simple random walk on $\Z$, the above definitions have the following 
interpretation (see Fig.~\ref{fig-copolex}). Think of $\pi\in\Pi$ in (\ref{Wndef}) 
as the path of a directed copolymer on $\N_0\times\Z$, consisting of monomers 
represented by the edges $(\pi_{k-1},\pi_k)$, $k\in\N$, pointing  either north-east 
of south-east. Think of the lower half-plane as water and the upper half-plane as oil. 
The monomers are labeled by $\omega$, with $\omega_k=-1$ indicating that monomer $k$ 
is hydrophilic and $\omega_k=+1$ that it is hydrophobic. Both types occur with density 
$\tfrac12$. The factor ${\rm sign}(\pi_{k-1},\pi_k)$ in (\ref{copoldef}) equals $-1$ 
or $+1$ depending on whether monomer $k$ lies in the water or in the oil. The interaction 
Hamiltonian in (\ref{copoldef}) therefore rewards matches and penalizes mismatches of
the monomer types and the solvents. The parameter $\beta$ is the \emph{interaction strength} 
(or inverse temperature), the parameter $h$ plays the role of the \emph{interaction bias}: 
$h=0$ corresponds to the hydrophobic and hydrophilic monomers interacting equally strongly, 
while $h=1$ corresponds to the hydrophilic monomers not interacting at all. The probability 
distribution of the copolymer given $\omega$ is the quenched Gibbs distribution in 
(\ref{Pnomdef}). For simple random walk the support of $\rho$ is $2\N$ and the exponent 
is $\alpha=\tfrac32$: $\rho(2m) \sim 1/2\pi^{1/2}m^{3/2}$ as $m\to\infty$ (Feller~\cite{Fe68},
Chapter III).


\subsection{Quenched free energy and critical curve}
\label{S1.2}

The model in Section~\ref{S1.1} was introduced in Garel, Huse, Leibler and 
Orland~\cite{GaHuLeOr89}. It was shown in Bolthausen and den Hollander~\cite{BodHo97} 
that for every $\beta,h \geq 0$ the \emph{quenched free energy} per monomer
\begin{equation}
\label{freeenegdef}
f^\mathrm{que}(\beta,h) = \lim_{n\to\infty} \frac{1}{n} \log Z_n^{\beta,h,\omega}
\qquad \mbox{exists $\omega$-a.s.\ and in $L^1(\P)$, and is $\omega$-a.s.\ constant.}
\end{equation} 
It was further noted that 
\begin{equation}
\label{freeeneglb}
f^\mathrm{que}(\beta,h) \geq \beta h.
\end{equation}
This lower bound comes from the strategy where the path spends all of its time above 
the interface, i.e., $\pi_k>0$ for $1 \leq k \leq n$. Indeed, in that case ${\rm sign}
(\pi_{k-1},\pi_k)=+1$ for $1 \leq k \leq n$, resulting in $H_n^{\beta,h,\omega}(\pi) 
= -\beta h n[1+o(1)]$ $\omega$-a.s.\ as $n \to \infty$ by the strong law of large 
numbers for $\omega$ (recall (\ref{nuzmuv})). Since $P(\{\pi\in\Pi\colon\,\pi_k>0 
\mbox{ for } 1 \leq k \leq n\}) = \sum_{k>n} \rho(n) = n^{1-\alpha+o(1)}$ as $n\to\infty$ 
by (\ref{rhocond}), the cost of this strategy under $P$ is negligible on an exponential 
scale.

In view of (\ref{freeeneglb}), it is natural to introduce the \emph{quenched excess free 
energy}
\begin{equation}
\label{exfreeenegdef}
g^\mathrm{que}(\beta,h) = f^\mathrm{que}(\beta,h)-\beta h,
\end{equation}
to define the two phases
\begin{equation}
\label{queLDdef}
\begin{aligned}
\cD^\mathrm{que} &= \{(\beta,h)\colon\,g^\mathrm{que}(\beta,h) = 0\},\\
\cL^\mathrm{que} &= \{(\beta,h)\colon\,g^\mathrm{que}(\beta,h) > 0\},
\end{aligned}
\end{equation}
and to refer to $\cD^\mathrm{que}$ as the \emph{quenched delocalized phase}, where the 
strategy of staying above the interface is optimal, and to $\cL^\mathrm{que}$ as the 
\emph{quenched localized phase}, where this strategy is not optimal. The presence of 
these two phases is the result of a competition between entropy and energy: by staying 
close to the interface the copolymer looses entropy, but it gains energy because it 
can more easily switch between the two sides of the interface in an attempt to place 
as many monomers as possible in their preferred solvent.

General convexity arguments show that $\cD^\mathrm{que}$ and $\cL^\mathrm{que}$ are 
separated by a \emph{quenched critical curve} $\beta \mapsto h^\mathrm{que}_c(\beta)$ 
given by
\begin{equation}
\label{hcquedef}
h^\mathrm{que}_c(\beta) = \sup\{h \geq 0\colon\,g^\mathrm{que}(\beta,h)>0\}
= \inf\{h \geq 0\colon\,g^\mathrm{que}(\beta,h)=0\}, \qquad \beta \geq 0,
\end{equation}
with the property that $h_c^\mathrm{que}(0) = 0$, $\beta\mapsto h_c^\mathrm{que}(\beta)$ 
is strictly increasing and finite on $[0,\infty)$, and $\beta\mapsto\beta h_c^\mathrm{que}
(\beta)$ is strictly convex on $[0,\infty)$. Moreover, it is easy to check that 
$\lim_{\beta\to\infty} h_c^\mathrm{que}(\beta)=\sup[\mathrm{supp}(\nu)]$, the supremum 
of the support of $\nu$ (see Fig.~\ref{fig-critcurve}).

\begin{figure}[htbp]
\setlength{\unitlength}{0.35cm}
\begin{center}
\begin{picture}(12,12)(9,-1)
\put(0,0){\line(12,0){12}}
\put(0,0){\line(0,8){8}}
{\thicklines
\qbezier(6,0)(1,.3)(0,7)
\qbezier(6,0)(8,0)(11,0)
}
\put(-.8,-.8){$0$}  
\put(12.5,-0.2){$h$}
\put(-1,8.7){$g^\mathrm{que}(\beta,h)$}
\put(6,0){\circle*{.4}}
\put(5.5,-1.5){$h^\mathrm{que}_c(\beta)$}
\put(20,0){\line(12,0){12}}
\put(20,0){\line(0,8){8}}
{\thicklines
\qbezier(20,0)(21,5)(28,6)
}
\put(19.2,-.8){$0$}  
\put(32.5,-.2){$\beta$}
\put(19,8.7){$h_c^\mathrm{que}(\beta)$}
\put(20,0){\circle*{.4}}
\put(24.5,3){$\cL^\mathrm{que}$}
\put(20.7,5.5){$\cD^\mathrm{que}$}
\end{picture}
\end{center}
\caption{\small Qualitative pictures of $h\mapsto g^{\rm que}(\beta,h)$ for 
fixed $\beta>0$, respectively, $\beta \mapsto h^\mathrm{que}_c(\beta)$. The 
quenched critical curve is part of $\cD^\mathrm{que}$.}  
\label{fig-critcurve}
\end{figure}
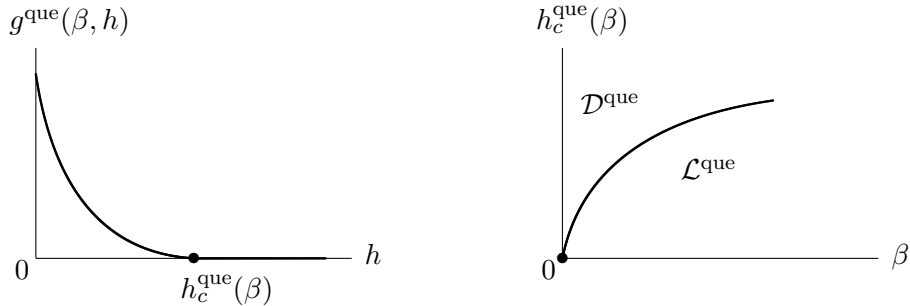

The following bounds are known for the quenched critical curve:
\begin{equation}
\label{hcbds}
\left(\tfrac{2\beta}{\alpha}\right)^{-1} M\left(\tfrac{2\beta}{\alpha}\right)
\leq h^\mathrm{que}_c(\beta) \leq (2\beta)^{-1} M(2\beta) \qquad \forall\,\beta > 0.
\end{equation}
The upper bound was proved in Bolthausen and den Hollander~\cite{BodHo97}, and comes 
from an annealed estimate on $\omega$. The lower bound was proved in Bodineau and 
Giacomin~\cite{BoGi04}, and comes from strategies where the copolymer dips below 
the interface during rare stretches in $\omega$ where the empirical density is 
sufficiently biased downwards. 

\medskip\noindent
{\bf Remark:}
In the literature $\rho$ is typically assumed to be \emph{regularly varying at 
infinity}, i.e., 
\begin{equation}
\label{rhocondalt}
\rho(m) = m^{-\alpha}L(m) \mbox{ for some } \alpha \geq 1
\mbox{ with } L \mbox{ slowly varying at infinity}.
\end{equation}
However, the proof of (\ref{hcbds}) in \cite{BodHo97} and \cite{BoGi04} can be extended 
to $\rho$ satisfying the \emph{much weaker} assumption in (\ref{rhocond}). In the 
literature $\nu$ is sometimes assumed to have Gaussian or sub-Gaussian tails, which 
is stronger than (\ref{mgffin}). Also this is not necessary for (\ref{hcbds}). Throughout 
our paper, \eqref{mgffin} and \eqref{rhocond} \emph{are the only conditions in force} 
(with a sole exception indicated later on).


\subsection{Annealed free energy and critical curve}
\label{S1.3}

Recalling (\ref{Wndef}--\ref{Pnomdef}), (\ref{freeenegdef}) and (\ref{exfreeenegdef}), 
and using that $\beta \sum_{k=1}^n (\omega_k+h) = \beta hn[1+o(1)]$ $\omega$-a.s.\ 
as $n\to\infty$, we see that the quenched excess free energy is given by  
\begin{equation}
\label{exfreeenegdefalt}
g^\mathrm{que}(\beta,h) = \lim_{n\to\infty} \frac{1}{n} 
\log \widetilde Z_n^{\beta,h,\omega}
\qquad \mbox{$\omega$-a.s.}
\end{equation}
with
\begin{equation}
\label{queexpartsum}
\widetilde Z_n^{\beta,h,\omega} = \sum_{\pi\in\Pi} P(\pi)\,
\exp\left[\beta \sum_{k=1}^n (\omega_k+h)\,[{\rm sign}(\pi_{k-1},\pi_k)-1]\right].
\end{equation}
In this partition sum only the excursions of the copolymer below the interface 
contribute. The annealed version of the model has partition sum
\begin{equation}
\label{annexpartsum}
\E(\widetilde Z_n^{\beta,h,\omega}) = \sum_{\pi\in\Pi} P(\pi)\,\prod_{k=1}^n
\left[1_{\{{\rm sign}(\pi_{k-1},\pi_k)=1\}} + e^{M(2\beta)-2\beta h}\,
1_{\{{\rm sign}(\pi_{k-1},\pi_k)=-1\}}\right],
\end{equation}
where $\E$ is expectation w.r.t.\ $\P$. The \emph{annealed excess free energy} 
is therefore given by
\begin{equation}
\label{annexfreeeneg}
g^\mathrm{ann}(\beta,h) = \lim_{n\to\infty} \frac{1}{n} 
\log \E(\widetilde Z_n^{\beta,h,\omega}).
\end{equation}
(\emph{Note:} In the annealed model the average w.r.t.\ $\P$ is taken on the partition 
sum $\widetilde Z_n^{\beta,h,\omega}$ in (\ref{queexpartsum}) rather than on the original 
partition sum $Z_n^{\beta,h,\omega}$ in (\ref{Pnomdef}).) The two corresponding phases 
are
\begin{equation}
\label{annLDdef}
\begin{aligned}
\cD^\mathrm{ann} &= \{(\beta,h)\colon\,g^\mathrm{ann}(\beta,h) = 0\},\\
\cL^\mathrm{ann} &= \{(\beta,h)\colon\,g^\mathrm{ann}(\beta,h) > 0\},
\end{aligned}
\end{equation}
which are referred to as the \emph{annealed delocalized phase}, respectively, the 
\emph{annealed localized phase}, and are separated by an \emph{annealed critical 
curve} $\beta \mapsto h^\mathrm{ann}_c(\beta)$ given by
\begin{equation}
\label{hcanndef}
h^\mathrm{ann}_c(\beta) = \sup\{h \geq 0\colon\,g^\mathrm{ann}(\beta,h)>0\}
= \inf\{h \geq 0\colon\,g^\mathrm{ann}(\beta,h)=0\}, \qquad \beta \geq 0.
\end{equation}

\begin{figure}[htbp]
\vspace{-.5cm}
\setlength{\unitlength}{0.35cm}
\begin{center}
\begin{picture}(12,12)(9,-1)
\put(0,0){\line(12,0){12}}
\put(0,0){\line(0,8){8}}
{\thicklines
\qbezier(0,6)(3,3)(6,0)
\qbezier(6,0)(8,0)(11,0)
}
\put(-.8,-.8){$0$}
\put(-3.8,5.8){$M(2\beta)$}  
\put(12.5,-0.2){$h$}
\put(-1,8.7){$g^\mathrm{ann}(\beta,h)$}
\put(6,0){\circle*{.4}}
\put(5.5,-1.5){$h^\mathrm{ann}_c(\beta)$}
\put(20,0){\line(12,0){12}}
\put(20,0){\line(0,8){8}}
{\thicklines
\qbezier(20,0)(21,5)(28,6)
}
\put(19.1,-.8){$0$}  
\put(32.5,-.2){$\beta$}
\put(19,8.7){$h_c^\mathrm{ann}(\beta)$}
\put(20,0){\circle*{.4}}
\put(24.5,3){$\cL^\mathrm{ann}$}
\put(20.7,5.5){$\cD^\mathrm{ann}$}
\end{picture}
\end{center}
\caption{\small Qualitative picture of $h\mapsto g^{\rm ann}(\beta,h)$ for 
fixed $\beta>0$, respectively, $\beta\mapsto h^\mathrm{ann}_c(\beta)$. The 
annealed critical curve is part of $\cD^\mathrm{ann}$.} 
\label{fig-gann}
\end{figure}
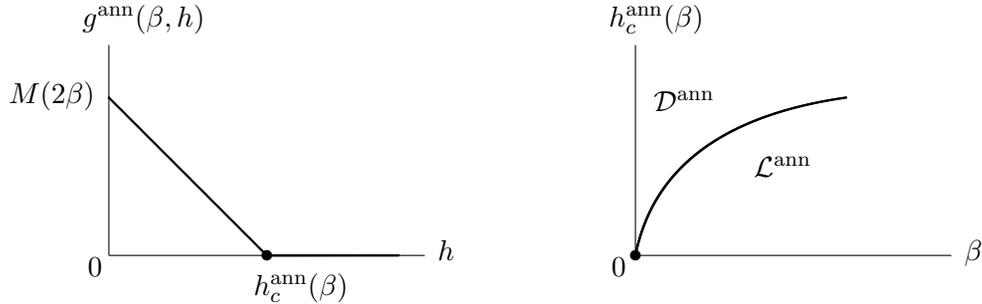

An easy computation based on (\ref{annexpartsum}) gives that (see Fig.~\ref{fig-gann})
\begin{equation}
\label{gannid}
g^\mathrm{ann}(\beta,h)= 0 \vee [M(2\beta)-2\beta h], \qquad \beta,h \geq 0,
\end{equation}
and
\begin{equation}
\label{hcannid}
h^\mathrm{ann}_c(\beta) = (2\beta)^{-1} M(2\beta), \qquad \beta>0. 
\end{equation}
Thus, the upper bound in (\ref{hcbds}) equals $h^\mathrm{ann}_c(\beta)$, while the 
lower bound equals $h^\mathrm{ann}_c(\beta/\alpha)$. 


\subsection{Main results}
\label{S1.4}

Our variational characterization of the excess free energies and the critical curves
are contained in the following theorem. 

\begin{theorem}
\label{freeenegvar} 
Assume {\rm (\ref{mgffin})} and {\rm (\ref{rhocond})}.\\ 
(i) For every $\beta,h>0$, there are lower semi-continuous, convex and non-increasing 
functions
\begin{equation}
\label{Sannounce}
\begin{aligned}
&g\mapsto S^\mathrm{que}(\beta,h;g),\\
&g\mapsto S^\mathrm{ann}(\beta,h;g),
\end{aligned}
\end{equation}
given by explicit variational formulas, such that
\begin{equation}
\label{gquegannvarfor}
\begin{aligned}
g^\mathrm{que}(\beta,h) &= \inf\{g\in\R\colon\,S^\mathrm{que}(\beta,h;g)<0\},\\
g^\mathrm{ann}(\beta,h) &= \inf\{g\in\R\colon\,S^\mathrm{ann}(\beta,h;g)<0\}.
\end{aligned}	
\end{equation} 
(ii) For every $\beta>0$, $g^\mathrm{que}(\beta,h)$ and $g^\mathrm{ann}(\beta,h)$ are the 
unique solutions of the equations
\begin{equation}
\label{Ssolg}
\begin{array}{lll}
&S^\mathrm{que}(\beta,h;g)=0  &\mbox{ for } 0<h\leq h_c^\mathrm{que}(\beta),\\
&S^\mathrm{ann}(\beta,h;g)=0  &\mbox{ for } h=h_c^{\rm ann}(\beta).
\end{array}
\end{equation}
(iii) For every $\beta>0$, $h_c^\mathrm{que}(\beta)$ and $h_c^\mathrm{ann}(\beta)$ are the 
unique solutions of the equations 
\begin{equation}
\label{Ssolh}
\begin{array}{ll}
&S^\mathrm{que}(\beta,h;0)=0,\\ 
&S^\mathrm{ann}(\beta,h;0)=0.
\end{array}
\end{equation} 
\end{theorem} 
 
\noindent
The variational formulas for $S^\mathrm{que}(\beta,h;g)$ and $S^\mathrm{ann}(\beta,h;g)$ 
are given in Theorem~\ref{varfloc}, respectively, Theorem~\ref{varfloc1} in Section~\ref{S3}. 
Figs.~\ref{fig-varfe}--\ref{fig-varhcbar} in Section~\ref{S3} show how these functions depend 
on $\beta,h$ and $g$, which is crucial for our analysis.

Next we state six corollaries that are consequences of the variational formulas. The first
three corollaries are strict inequalities for the excess free energies and the critical 
curves.

\begin{corollary}
\label{freeeneggap} 
$g^\mathrm{que}(\beta,h) < g^\mathrm{ann}(\beta,h)$ for all $(\beta,h) \in \cL^\mathrm{ann}$.
\end{corollary} 

\begin{corollary}
\label{hcubstrict} 
If $\alpha>1$, then $h^\mathrm{que}_c(\beta) < h^\mathrm{ann}_c(\beta)$ for all 
$\beta>0$.
\end{corollary}

\begin{corollary}
\label{hclbstrict}
If $\alpha>1$, then $h^\mathrm{que}_c(\beta) >  h^\mathrm{ann}_c(\beta/\alpha)$ 
for all $\beta>0$.
\end{corollary}

The fourth corollary concerns the slope of the quenched critical curve at $\beta=0$. 
For $1<\alpha<2$, let
\begin{equation}
\label{IaBdef}
I_\alpha(B) = \int_0^\infty dy\,y^{-\alpha}\,[E_\alpha(y,B)-1], \qquad B \geq 1,
\end{equation}
where
\begin{equation}
\label{EalyBdef}
E_\alpha(y,B) = \int_\R dx\,\frac{1}{\sqrt{2\pi}}\,e^{-\tfrac12 x^2}\,
f_\alpha\left(e^{-2By-2\sqrt{y}x}\right)
\end{equation}
with
\begin{equation}
\label{faldef}
f_\alpha(z)=\big\{\tfrac12(1+z^\alpha)\big\}^{1/\alpha},
\end{equation}
and let $1<B(\alpha)<\infty$ be the unique solution of the equation $I_\alpha(B)=0$.
We say that $\rho$ is asymptotically periodic when there exists a $p\in\N$ such that
$\rho(m)>0$ if and only if $m \in p\,\N$ for $m$ large enough.

\begin{corollary}
\label{Kcslope}
Suppose that $\rho$ is asymptotically periodic. Suppose further that either 
$m_\rho<\infty$, or $m_\rho=\infty$ and $\rho$ is regularly varying at 
infinity (i.e., {\rm \eqref{rhocondalt}} holds along the support of $\rho$). 
Then $\liminf_{\beta\downarrow 0} h^\mathrm{que}_c(\beta)/\beta \geq K_c^*$ 
with (see Fig.~{\rm \ref{fig-slopeplot}})
\begin{equation}
\label{Kcvalue}
K_c^* = K_c^*(\alpha) = \left\{\begin{array}{ll}
\frac{1}{\alpha}B(\alpha), 
&\text{ for } 1<\alpha<2,\\[0.2cm]
\frac{1+\alpha}{2\alpha}, 
&\text{ for } \alpha \geq 2.
\end{array}
\right.
\end{equation}
\end{corollary}

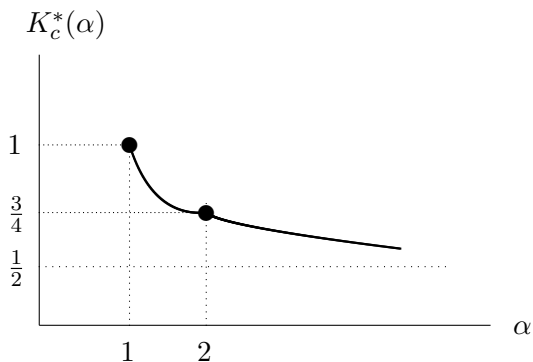
\begin{figure}[htbp]
\vspace{-.5cm}
\begin{center}
\setlength{\unitlength}{0.6cm}
\begin{picture}(8,8)(0,0)
\put(0,0){\line(10,0){10}}
\put(0,0){\line(0,6){6}}
{\thicklines
\qbezier(3.7,2.5)(4,2.2)(8,1.7)
\qbezier(2,4)(2.5,2.4)(3.7,2.5)
}
\qbezier[60](0,1.3)(4.5,1.3)(9,1.3)
\qbezier[20](0,4)(1,4)(2,4)
\qbezier[30](2,4)(2,2)(2,0)
\qbezier[40](0,2.5)(2,2.5)(3.7,2.5)
\qbezier[20](3.7,2.7)(3.7,1.5)(3.7,0)
\put(-.3,6.5){$K_c^*(\alpha)$}
\put(10.5,-.2){$\alpha$}
\put(1.8,-.8){$1$}
\put(3.5,-.8){$2$}
\put(-.7,1.1){$\tfrac12$}
\put(-.7,2.3){$\tfrac34$}
\put(-.7,3.8){$1$}
\put(2,4){\circle*{.35}}
\put(3.7,2.5){\circle*{.35}}
\end{picture}
\end{center}
\caption{Qualitative picture of $\alpha \mapsto K_c^*(\alpha)$.}
\label{fig-slopeplot}
\end{figure}

The last two corollaries concern the typical path behavior. Let $\wP^{\beta,h,\o}_n$ denote 
the path measure associated with the constrained partition sum $\wZ^{\beta,h,\o}_n$ defined 
in \eqref{queexpartsum}. Write $\cM_n = |\{1\leq i\leq n\colon\,\pi_i=0\}|$ to denote the 
number of times $\pi$ returns to the interface up to time $n$.

\begin{corollary}
\label{delocpathprop}
For every $(\beta,h)\in\mathrm{int}(\cD^\mathrm{que})$ and $c>\alpha/[-S^\mathrm{que}
(\beta,h;0)]\in (0,\infty)$,
\begin{equation}
\label{ubdint}
\lim_{n\to\infty} \wP^{\beta,h,\o}_n\left(\cM_n \geq c\log n\right) = 0 \qquad \o-a.s.
\end{equation}
\end{corollary}

\begin{corollary}
\label{locpathprop}
For every $(\beta,h)\in\cL^\mathrm{que}$, 
\begin{equation}
\label{LLNpath}
\lim_{n\to\infty} 
\wP^{\beta,h,\o}_n\left(|\tfrac1n\cM_n-C| \leq \varepsilon \right) = 1
\qquad \o-a.s.\quad \forall\,\varepsilon>0,
\end{equation}
where 
\begin{equation}
\label{densid}
-\frac{1}{C} = \frac{\partial}{\partial g}\,
S^\mathrm{que}\big(\beta,h;g^\mathrm{que}(\beta,h)\big)
\in (-\infty,0),
\end{equation}
provided this derivative exists. (By convexity, at least the left-derivative and the 
right-derivative exist.)
\end{corollary}


\subsection{Discussion}
\label{S1.5}

\medskip\noindent
{\bf 1.}
The main importance of our results in Section~\ref{S1.4} is that they open up a window 
on the copolymer model with a variational view. Whereas the results in the literature 
were obtained with the help of a variety of \emph{estimation techniques}, 
Theorem~\ref{freeenegvar} provides \emph{variational formulas} that are new and 
explicit. As we will see in Section~\ref{S3}, these variational formulas are not 
easy to manipulate. However, they provide a natural setting, and are robust in the 
sense that the large deviation principles on which they are based (see Section~\ref{S2}) 
can be applied to other polymer models as well, e.g.\ the pinning model with disorder 
(Cheliotis and den Hollander~\cite{ChdHo10}). Still other applications involve certain 
classes of interacting stochastic systems (Birkner, Greven and den 
Hollander~\cite{BiGrdHo11}). For an overview, see den Hollander~\cite{dHo10}.

\medskip\noindent
{\bf 2.} 
The gap between the excess free energies stated in Corollary~\ref{freeeneggap} has never 
been claimed in the literature, but follows from known results. Fix $\beta>0$. We know 
that $h\mapsto g^\mathrm{ann}(\beta,h)$ is strictly positive, strictly decreasing and 
linear on $(0,h_c^\mathrm{ann}(\beta)]$, and zero on $[h_c^\mathrm{ann}(\beta),\infty)$ 
(see Fig.~\ref{fig-gann}). We also know that $h\mapsto g^\mathrm{que}(\beta,h)$ is 
strictly positive, strictly decreasing and convex on $(0,h_c^\mathrm{que}(\beta)]$, 
and zero on $[h_c^\mathrm{que}(\beta),\infty)$. It was shown in Giacomin and 
Toninelli~\cite{GiTo06a,GiTo06b} that $h\mapsto g^\mathrm{que}(\beta,h)$ drops below 
a quadratic as $h \uparrow h_c^\mathrm{que}(\beta)$, i.e., the phase transition is 
``at least of second order'' (see Fig.~\ref{fig-critcurve}). Hence, the gap is present 
in a left-neighborhood of $h_c^\mathrm{que}(\beta)$. Combining this observation with 
the fact that $g^\mathrm{que}(\beta,h) \leq g^\mathrm{ann}(\beta,h)$ and 
$h_c^\mathrm{que}(\beta) \leq h_c^\mathrm{ann}(\beta)$, it follows that the gap is 
present for all $h \in (0,h_c^\mathrm{ann}(\beta))$. \emph{Note:} The above argument 
crucially relies on the linearity of $h \mapsto g^\mathrm{ann}(\beta,h)$ on 
$(0,h_c^\mathrm{ann}(\beta)]$. However, we will see in Section~\ref{S3} that our 
proof of Corollary~\ref{freeeneggap} is robust and does not depend on this linearity.

\medskip\noindent
{\bf 3.}
For a number of years, all attempts in the literature to improve (\ref{hcbds}) had failed. 
As explained in Orlandini, Rechnitzer and Whittington~\cite{OrReWh02} and Caravenna 
and Giacomin~\cite{CaGi05}, the reason behind this failure is that any improvement of 
(\ref{hcbds}) necessarily requires a deep understanding of the global behavior of the 
copolymer when the parameters are close to the quenched critical curve. Toninelli~\cite{To08} 
proved the strict upper bound in Corollary~\ref{hcubstrict} with the help of \emph{fractional
moment estimates} for unbounded disorder and large $\beta$ subject to (\ref{mgffin}) and 
(\ref{rhocondalt}), and this result was later extended by Bodineau, Giacomin, Lacoin and 
Toninelli~\cite{BoGiLaTo08} to arbitrary disorder and arbitrary $\beta$, again subject to 
(\ref{mgffin}) and (\ref{rhocondalt}). The latter paper also proved the strict lower bound 
in Corollary~\ref{hclbstrict} with the help of \emph{appropriate localization strategies} 
for small $\beta$ and $\alpha\geq\alpha_0$, where $\alpha_0 \approx 1.801$ (theoretical 
bound) and $\alpha_0 \approx 1.65$ (numerical bound), which unfortunately excludes the 
simple random walk example in Section~\ref{S1.1} for which $\alpha=\tfrac32$. 
Corollaries~\ref{hcubstrict} and \ref{hclbstrict} settle the strict inequalities in 
full generality subject to (\ref{mgffin}) and (\ref{rhocond}).

\medskip\noindent
{\bf 4.} 
A point of heated debate has been the value of 
\begin{equation}
\label{Kcdef}
K_c = \lim_{\beta \downarrow 0} h_c^\mathrm{que}(\beta)/\beta,
\end{equation}
which is believed to be \emph{universal}, i.e., to depend on $\alpha$ alone and 
to be robust under changes of the fine details of the interaction Hamiltonian. 
The existence of $K_c$ was proved in Bolthausen and den Hollander~\cite{BodHo97} 
for $\rho$ associated with simple random walk ($\alpha=\tfrac32$) and binary 
disorder. The proof uses a Brownian approximation of the copolymer model. This 
result was extended in Caravenna and Giacomin~\cite{CaGi10} to $\rho$ satisfying 
(\ref{rhocondalt}) with $1<\alpha<2$ and disorder with a moment generating function 
that is finite in a neighborhood of the origin. The proof uses a L\'evy approximation 
of the copolymer model. The L\'evy copolymer serves as the \emph{attractor of a 
universality class}, indexed by the exponent $1<\alpha<2$. For $\alpha \geq 2$, 
the existence of the limit has remained open. The bounds in (\ref{hcbds}) imply 
that $K_c \in [1/\alpha,1]$, and various claims were made in the literature 
arguing in favor of $K_c=1/\alpha$, respectively, $K_c=1$. However, in Bodineau, 
Giacomin, Lacoin and Toninelli~\cite{BoGiLaTo08} it was shown that $\liminf_{\beta
\downarrow 0} h_c^\mathrm{que}(\beta)/\beta > 1/\alpha$ for $\alpha \geq \alpha_0$
and $\liminf_{\beta \downarrow 0} h_c^\mathrm{que}(\beta)/\beta \geq \tfrac12 \vee 
(1/\sqrt{\alpha})$ for $\alpha>2$. Corollary~\ref{Kcslope} improves these two 
lower bounds. We do not have an upper bound. In \cite{BoGiLaTo08} it was shown that
$\limsup_{\beta \downarrow 0} h_c^\mathrm{que}(\beta)/\beta < 1$ for $\alpha>2$, 
which was later extended to $\alpha>1$ in Toninelli~\cite{To09}. For an overview, 
see Caravenna, Giacomin and Toninelli~\cite{CaGiTo11}. 
 
\medskip\noindent
{\bf 5.} A numerical analysis for simple random walk ($\alpha=\tfrac32$) and 
binary disorder carried out in Caravenna, Giacomin and Gubinelli~\cite{CaGiGu06} 
(see also Giacomin~\cite{Gi07}, Chapter 9) showed that $K_c \in [0.82,0.84]$. 
Since $\tfrac56=0.833\dots$, the following conjecture is natural.
 
\begin{conjecture}
\label{kcconjecture}
$K_c = \frac{1+\alpha}{2\alpha}$ for all $\alpha>1$.
\end{conjecture}

\noindent
In \cite{CaGiGu06} is was also shown that 
\begin{equation}
h^\mathrm{que}_c(\beta) \approx (2K_c\beta)^{-1}\log\cosh(2K_c\beta)
\mbox{ for moderate } \beta.
\end{equation} 
Thus, the quenched critical curve lies ``somewhere halfway'' between the two bounds 
in (\ref{hcbds}), and so it remains a challenge to quantify the strict inequalities 
in Corollaries~\ref{hcubstrict} and \ref{hclbstrict}. Some quantification for the 
upper bound was offered in Bodineau, Giacomin, Lacoin and Toninelli~\cite{BoGiLaTo08}, 
and for the lower bound in Toninelli~\cite{To09}. Our proofs of Corollaries~\ref{hcubstrict} 
and \ref{hclbstrict} sharpen these quantifications.

\medskip\noindent
{\bf 6.} Because of (\ref{hcbds}), it was suggested that the quenched critical curve 
possibly depends on the exponent $\alpha$ of $\rho$ alone and not on the fine details 
of $\rho$. However, it was shown in Bodineau, Giacomin, Lacoin and Toninelli~\cite{BoGiLaTo08} 
that, subject to \eqref{mgffin}, for every $\alpha>1$, $\beta>0$ and $\epsilon>0$ there 
exists a $\rho$ satisfying (\ref{rhocondalt}) such that $h^\mathrm{que}_c(\beta)$ is 
$\epsilon$-close to the upper bound, which rules out such a scenario. Our variational
characterization in Section~\ref{S3} confirms this observation, and makes it quite evident 
that the fine details of $\rho$ do indeed matter.

\medskip\noindent
{\bf 7.}
Special cases of Corollaries~\ref{delocpathprop} and \ref{locpathprop} were proved
in Biskup and den Hollander~\cite{BidHo99} (for simple random walk and binary disorder) 
and in Giacomin and Toninelli~\cite{GiTo05,GiTo06c} (subject to (\ref{rhocondalt}) and 
for disorder satisfying a Gaussian concentration of measure bound). However, no formulas 
were obtained for the relevant constants.The latter two papers prove the bound under 
the average quenched measure, i.e., under $\E(P_n^{\beta,h,\omega})$. For the pinning 
model with disorder, the same result as in Corollary~\ref{delocpathprop} was derived 
in Mourrat~\cite{Mopr} with the help of the variational characterization obtained 
in Cheliotis and den Hollander~\cite{ChdHo10}.


\subsection{Outline}
\label{S1.7}

In Section~\ref{S2} we recall two large deviation principles (LDP's) derived in 
Birkner~\cite{Bi08} and Birkner, Greven and den Hollander~\cite{BiGrdHo10}, which 
describe the large deviation behavior of the empirical process of words cut out 
from a random letter sequence according to a random renewal process with exponentially 
bounded, respectively, polynomial tails. In Section~\ref{S3} we use these LDP's 
to prove Theorem~\ref{freeenegvar}. In Sections~\ref{S4}--\ref{S8} we prove 
Corollaries~\ref{freeeneggap}--\ref{locpathprop}. Appendices~\ref{appA}--\ref{appD} 
contain a number of technical estimates that are needed in Section~\ref{S3}.  

In Cheliotis and den Hollander~\cite{ChdHo10}, the LDP's in \cite{BiGrdHo10} were 
applied to the pinning model with disorder, and variational formulas were derived 
for the critical curves (not the free energies). The Hamiltonian is similar in spirit 
to (\ref{copoldef}), except that the disorder is felt only \emph{at} the interface, 
which makes the pinning model easier than the copolymer model. The present paper borrows 
ideas from \cite{ChdHo10}. However, the new challenges that come up are considerable.


\section{Large deviation principles: intermezzo}
\label{S2}

In this section we recall the LDP's from Birkner~\cite{Bi08} and  Birkner, Greven 
and den Hollander~\cite{BiGrdHo10}, which are the key tools in the present paper. 
Section~\ref{S2.1} introduces the relevant notation, while Sections~\ref{S2.2} and 
\ref{S2.3} state the annealed, respectively, quenched version of the LDP. Apart 
from minor modifications, this section is copied from \cite{BiGrdHo10}. We repeat 
it here in order to set the notation and to keep the paper self-contained. 


\subsection{Notation}
\label{S2.1}

Let $E$ be a Polish space, playing the role of an alphabet, i.e., a set of \emph{letters}. 
Let $\widetilde{E} = \cup_{k\in\N} E^k$ be the set of \emph{finite words} drawn from
$E$, which can be metrized to become a Polish space. Write $\cP(E)$ and $\cP(\widetilde E)$
to denote the set of probability measures on $E$ and $\widetilde E$.

\begin{figure}[htbp]
\vspace{-4cm}
\begin{center}
\includegraphics[scale = .8]{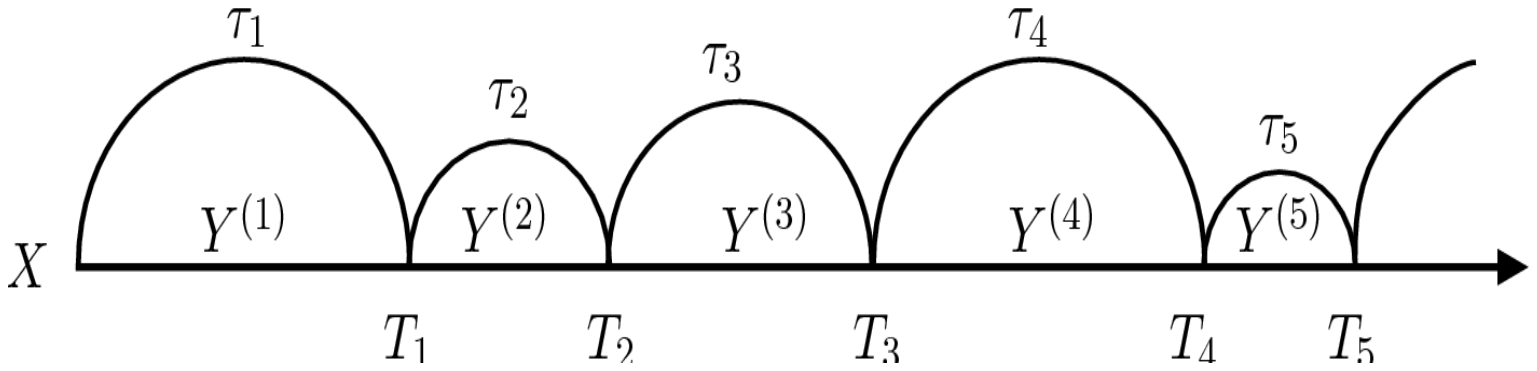}
\vspace{-20cm}
\caption{\small Cutting words out from a sequence of letters according to renewal times.}
\label{fig-cutting}
\end{center}
\end{figure}
\vspace{-0cm}

Fix $\nu \in \cP(E)$, and $\rho\in\cP(\N)$ satisfying (\ref{rhocond}). Let $X=(X_k)_{k\in\N}$
be i.i.d.\ $E$-valued random variables with marginal law $\nu$, and $\tau=(\tau_i)_{i\in\N}$
i.i.d.\ $\N$-valued random variables with marginal law $\rho$. Assume that $X$ and
$\tau$ are independent, and write $\Pr^\ast=\Pr\otimes P^\ast$ to denote their joint law. Cut 
words out of the letter sequence $X$ according to $\tau$ (see Fig.~\ref{fig-cutting}),
i.e., put
\begin{equation}
\label{Tdefs}
T_0=0 \quad \mbox{ and } \quad T_i=T_{i-1}+\tau_i,\quad i\in\N,
\end{equation}
and let
\begin{equation}
\label{eqndefYi}
Y^{(i)} = \bigl( X_{T_{i-1}+1}, X_{T_{i-1}+2},\dots, X_{T_{i}}\bigr),
\quad i \in \N.
\end{equation}
Under the law $\Pr^\ast$, $Y = (Y^{(i)})_{i\in\N}$ is an i.i.d.\ sequence of words
with marginal law $q_{\rho,\nu}$ on $\widetilde{E}$ given by
\begin{equation}
\label{q0def}
\begin{aligned}
q_{\rho,\nu}\big(dx_1,\dots,dx_m\big) 
&= \Pr^\ast\big(Y^{(1)}\in(dx_1,\dots,dx_m)\big)\\[0.2cm]
& = \rho(m) \,\nu(dx_1) \times \dots \times \nu(dx_m), \qquad
m\in\N,\,x_1,\dots,x_m\in E.
\end{aligned}
\end{equation}

We define $\rho_g$ as the tilted version of $\rho$ given by
\begin{equation}
\label{rhoz}
\rho_g(m) = \frac{e^{-gm}\rho(m)}{\cN(g)}, \quad m\in\N, \qquad
\cN(g)=\sum_{m\in\N} e^{-gm}\rho(m), \quad g \in [0,\infty).
\end{equation}
Note that if $g>0$, then $\rho_g$ has an \emph{exponentially bounded tail}. For $g=0$ 
we write $\rho$ instead of $\rho_0$. We write $P^\ast_g$ and $q_{\rho_g,\nu}$ for 
the analogues of $P^\ast$ and $q_{\rho,\nu}$ when $\rho$ is replaced by $\rho_g$ defined 
in \eqref{rhoz}.

The reverse operation of \emph{cutting} words out of a sequence of letters is 
\emph{glueing} words together into a sequence of letters. Formally, this is done 
by defining a \emph{concatenation} map $\kappa$ from $\widetilde{E}^\N$ to $E^{\N}$. 
This map induces in a natural way a map fsrom $\cP(\widetilde{E}^\N)$ to $\cP(E^{\N})$,
the sets of probability measures on $\widetilde{E}^\N$ and $E^{\N}$ (endowed with
the topology of weak convergence). The concatenation $q_{\rho,\nu}^{\otimes\N}\circ
\kappa^{-1}$ of $q_{\rho,\nu}^{\otimes\N}$ equals $\nu^{\N}$, as is evident from 
(\ref{q0def}). 

Let $\cP^{\mathrm{inv}}(\widetilde{E}^\N)$ be the set of probability measures
on $\widetilde{E}^\N$ that are invariant under the left-shift $\widetilde{\theta}$
acting on $\widetilde{E}^\N$. For $Q\in\cP^\mathrm{inv}(\widetilde{E}^\N)$, let 
$H(Q \mid q_{\rho,\nu}^{\otimes\N})$ be the \emph{specific relative entropy of $Q$ 
w.r.t.\ $q_{\rho,\nu}^{\otimes\N}$} defined by
\begin{equation}
\label{spentrdef}
H(Q \mid q_{\rho,\nu}^{\otimes\N}) = \lim_{N\to\infty} \frac{1}{N}\,
h(\widetilde\pi_N Q \mid q_{\rho,\nu}^N),
\end{equation}
where $\widetilde\pi_N Q \in \cP(\widetilde{E}^N)$ denotes the projection of $Q$ 
onto the first $N$ words, $h(\,\cdot\mid\cdot\,)$ denotes relative entropy, and 
the limit is non-decreasing. The following lemma relates the specific relative 
entropies of $Q$ w.r.t.\ $q_{\rho,\nu}^{\otimes\N}$ and $q_{\rho_g,\nu}^{\otimes\N}$.

\begin{lemma}
\label{rhozre}
For $Q\in\cP^\mathrm{inv}(\widetilde{E}^\N)$ and $g \in [0,\infty)$,
\begin{equation}
\label{rhozrho}
H(Q\mid q_{\rho_g,\nu}^{\otimes\N}) = H(Q\mid q_{\rho,\nu}^{\otimes\N})\
+ \log\cN(g)+gm_Q
\end{equation}
with $\cN(g)\in (0,1]$ defined in {\rm \eqref{rhoz}} and $m_Q=E_Q(\tau_1) \in [1,\infty]$ the 
average word length under $Q$ ($E_Q$ denotes expectation under the law $Q$ and $\tau_1$ 
is the length of the first word).
\end{lemma}

\begin{proof}
Observe from \eqref{rhoz} that 
\begin{equation}
\begin{split}
h(\widetilde\pi_N Q\mid q_{\rho_g,\nu}^N)
&= \int_{\widetilde{E}^N} (\widetilde\pi_NQ)(dy)\,
\log\left(\frac{d\widetilde\pi_N Q}{d q_{\rho_g,\nu}^N}(y)\right)\\
&= \int_{\widetilde{E}^N} (\widetilde\pi_NQ)(dy)\,
\log\left(\frac{\cN(g)^N}{e^{-g\sum_{i=1}^N |y^{(i)}|}}\,
\frac{d\widetilde\pi_N Q}{d q_{\rho,\nu}^N}(y)\right)\\
&= h(\widetilde\pi_N Q\mid q_{\rho,\nu}^N)+N\log\cN(g)+Ngm_{Q},
\end{split}
\end{equation}
where $|y^{(i)}|$ is the length of the $i$-th word and the second equality uses 
that $Q\in\cP^\mathrm{inv}(\widetilde{E}^\N)$. Let $N\to\infty$ and use 
\eqref{spentrdef}, to get the claim.
\end{proof}

Lemma~\ref{rhozre} implies that if $g>0$, then $m_Q<\infty$ whenever 
$H(Q\mid q_{\rho_g,\nu}^{\otimes\N})<\infty$. This is a special case of
\cite{Bi08}, Lemma 7.


\subsection{Annealed LDP}
\label{S2.2}

For $N\in\N$, let $(Y^{(1)},\dots,Y^{(N)})^\mathrm{per}$ be the periodic extension of 
the $N$-tuple $(Y^{(1)},\dots,Y^{(N)})\in\widetilde{E}^N$ to an element of 
$\widetilde{E}^\N$, and define 
\begin{equation}
\label{eqndefRN}
R_N^X = \frac{1}{N} \sum_{i=0}^{N-1}
\delta_{\widetilde{\theta}^i (Y^{(1)},\dots,Y^{(N)})^\mathrm{per}}
\in \mathcal{P}^{\mathrm{inv}}(\widetilde{E}^\N).
\end{equation}
This is the \emph{empirical process of $N$-tuples of words}. The superscript $X$ indicates 
that the words $Y^{(1)},\dots,Y^{(N)}$ are cut from the letter sequence $X$. The following 
\emph{annealed LDP} is standard (see e.g.\ Dembo and Zeitouni~\cite{DeZe98}, Section 6.5). 

\begin{theorem}
\label{aLDP}
For every $g \in [0,\infty)$, the family $(\Pr \times P^\ast_g)(R_N^X \in \cdot\,)$, $N\in\N$, satisfies the LDP on $\cP^{\mathrm{inv}}(\widetilde{E}^\N)$ with rate $N$ and with rate 
function $I_g^{\mathrm{ann}}$ given by
\begin{equation}
\label{Ianndef}
I_g^{\mathrm{ann}}(Q)= H\big(Q \mid q_{\rho_g,\nu}^{\otimes\N}\big), 
\qquad Q \in \cP^{\mathrm{inv}}(\widetilde{E}^\N).
\end{equation}
This rate function is lower semi-continuous, has compact level sets, has a unique
zero at $q_{\rho_g,\nu}^{\otimes\N}$, and is affine.
\end{theorem}

It follows from Lemma~\ref{rhozre} that 
\begin{equation}\label{Iannz}
I_g^{\rm ann}(Q)=I^{\rm ann}(Q)+\log\cN(g)+gm_Q,
\end{equation}
where $I^{\rm ann}(Q)=H(Q\mid q_{\rho,\nu}^{\otimes \N})$, the annealed rate function 
for $g=0$.


\subsection{Quenched LDP}
\label{S2.3}

To formulate the quenched analogue of Theorem~\ref{aLDP}, we need some more notation. 
Let $\cP^{\mathrm{inv}}(E^{\N})$ be the set of probability measures on $E^{\N}$ 
that are invariant under the left-shift $\theta$ acting on $E^{\N}$. For $Q\in
\cP^{\mathrm{inv}}(\widetilde{E}^\N)$ such that $m_Q < \infty$, define
\begin{equation}
\label{PsiQdef}
\Psi_Q = \frac{1}{m_Q} E_Q\left(\sum_{k=0}^{\tau_1-1}
\delta_{\theta^k\kappa(Y)}\right) \in \cP^{\mathrm{inv}}(E^{\N}).
\end{equation}
Think of $\Psi_Q$ as the shift-invariant version of $Q\circ\kappa^{-1}$ obtained
after \emph{randomizing} the location of the origin. This randomization is necessary
because a shift-invariant $Q$ in general does not give rise to a shift-invariant
$Q\circ\kappa^{-1}$.

For $\tr \in \N$, let $[\cdot]_\tr \colon\,\widetilde{E} \to [\widetilde{E}]_\tr
= \cup_{k=1}^\tr E^k$ denote the \emph{truncation map} on words defined by
\begin{equation}
\label{trunword}
y = (x_1,\dots,x_m) \mapsto [y]_\tr = (x_1,\dots,x_{m \wedge \tr}),
\qquad  m\in\N,\,x_1,\dots,x_m\in E,
\end{equation}
i.e., $[y]_\tr$ is the word of length $\leq\tr$ obtained from the word $y$ by dropping
all the letters with label $>\tr$. This map induces in a natural way a map from 
$\widetilde E^\N$ to $[\widetilde{E}]_\tr^\N$, and from $\cP^{\mathrm{inv}}
(\widetilde{E}^\N)$ to $\cP^{\mathrm{inv}}([\widetilde{E}]_\tr^\N)$. Note that 
if $Q\in\cP^{\mathrm{inv}}(\widetilde{E}^\N)$, then $[Q]_\tr$ is an element of the 
set
\begin{equation}
\label{Pfin} 
\cP^{\mathrm{inv,fin}}(\widetilde{E}^\N) = \{Q\in\cP^{\mathrm{inv}}(\widetilde{E}^\N)
\colon\,m_Q<\infty\}.  
\end{equation}

Define (w-lim means weak limit)
\begin{equation}
\label{Rdef}
\cR = \left\{Q\in\cP^{\mathrm{inv}}(\widetilde{E}^\N)\colon\,
{\rm w}-\lim_{N\to\infty} \frac{1}{N}
\sum_{k=0}^{N-1} \delta_{\theta^k\kappa(Y)}=\nu^{\otimes\N} \quad Q-a.s.\right\},
\end{equation}
i.e., the set of probability measures in $\cP^{\mathrm{inv}}(\widetilde{E}^\N)$ under
which the concatenation of words almost surely has the same asymptotic statistics as 
a typical realization of $X$.

\begin{theorem}
\label{qLDP}
{\rm (Birkner~\cite{Bi08}; Birkner, Greven and den Hollander~\cite{BiGrdHo10})}
Assume {\rm (\ref{mgffin})} and {\rm (\ref{rhocond})}. Then, for $\nu^{\otimes \N}$--a.s.\ 
all $X$ and all $g \in [0,\infty)$, the family of (regular) conditional probability 
distributions $P^\ast_g(R_N^X \in \cdot\,)$, $N\in\N$, satisfies the LDP on $\cP^{\mathrm{inv}}
(\widetilde{E}^\N)$ with rate $N$ and with deterministic rate function $I^{\mathrm{que}}_g$ 
given by
\begin{equation}
\label{eqgndefinitionIalgz}
I^\mathrm{que}_g(Q) = \left\{\begin{array}{ll}
I^{\rm ann}_g(Q),
&\mbox{if } Q\in\cR,\\
\infty,
&\mbox{otherwise},
\end{array}
\right.
\text{ when } g>0,
\end{equation}
and
\begin{equation}
\label{eqgndefinitionIalg}
I^\mathrm{que}(Q) = \left\{\begin{array}{ll}
I^\mathrm{fin}(Q),
&\mbox{if } Q\in\cP^{\mathrm{inv,fin}}(\widetilde{E}^\N),\\
\lim_{\tr \to \infty} I^\mathrm{fin}\big([Q]_\tr\big),
&\mbox{otherwise},
\end{array}
\right.
\text{ when } g=0,
\end{equation}
where
\begin{equation}
\label{eqnratefctexplicitalg}
I^\mathrm{fin}(Q) = H(Q \mid q_{\rho,\nu}^{\otimes\N}) 
+ (\alpha-1) \, m_Q \,H\big(\Psi_{Q} \mid \nu^{\otimes\N}\big).
\end{equation}
This rate function is lower semi-continuous, has compact level sets, has a unique zero
at $q_{\rho_g,\nu}^{\otimes\N}$, and is affine.
\end{theorem}

The difference between \eqref{eqgndefinitionIalgz} for $g>0$ and
(\ref{eqgndefinitionIalg}--\ref{eqnratefctexplicitalg}) for $g=0$ 
can be explained as follows. For $g=0$, the word length distribution 
$\rho$ has a polynomial tail. It therefore is only exponentially 
costly to cut out a few words of an exponentially large length in 
order to move to stretches in $X$ that are suitable to build a 
large deviation $\{R_N^X \approx Q\}$ with words whose length is 
of order 1. This is precisely where the second term in 
\eqref{eqnratefctexplicitalg} comes from: this term is the extra cost 
to find these stretches under the quenched law rather than to create 
them ``on the spot'' under the annealed law. For $g>0$, on the other hand, 
the word length distribution $\rho_g$ has an exponentially bounded 
tail, and hence exponentially long words are too costly, so that
suitable stretches far away cannot be reached. Phrased differently, $g>0$ 
and $\alpha \in [1,\infty)$ is qualitatively similar to $g=0$ 
and $\alpha=\infty$, for which we see that the expression in 
\eqref{eqnratefctexplicitalg} is finite if and only $\Psi_{Q} 
=\nu^{\otimes\N}$. It was shown in \cite{Bi08}, Lemma 2, 
that
\begin{equation}
\label{Requiv}
\Psi_{Q} = \nu^{\otimes\N} \quad \Longleftrightarrow \quad  
Q\in\cR \qquad \mbox{ on } \cP^{\mathrm{inv,fin}}(\widetilde{E}^\N),
\end{equation}
and so this explains why the restriction $Q\in\cR$ appears in 
\eqref{eqgndefinitionIalgz}. For more background, see \cite{BiGrdHo10}. 

Note that $I^\mathrm{que}(Q)$ requires a truncation approximation when 
$m_Q=\infty$, for which case there is no closed form expression like in
\eqref{eqnratefctexplicitalg}. As we will see later on, the cases 
$m_Q<\infty$ and $m_Q=\infty$ need to be separated. For later reference 
we remark that, for all $Q\in\cP^{\mathrm{inv}}(\widetilde{E}^\N)$,
\begin{equation}
\label{truncapproxcont}
\begin{aligned}
I^\mathrm{ann}(Q)
&= \lim_{\tr \to \infty} I^\mathrm{ann}([Q]_\tr)
= \sup_{\tr \in \N} I^\mathrm{ann}([Q]_\tr),\\
I^\mathrm{que}(Q)
&= \lim_{\tr \to \infty} I^\mathrm{que}([Q]_\tr)
= \sup_{\tr \in \N} I^\mathrm{que}([Q]_\tr),
\end{aligned}
\end{equation}
as shown in \cite{BiGrdHo10}, Lemma A.1.


\section{Proof of Theorem~\ref{freeenegvar}}
\label{S3}

We are now ready to return to the copolymer and start our variational analysis. 

In Sections~\ref{S3.1} and \ref{S3.2} we derive the variational formulas for the 
quenched and the annealed excess free energies and critical curves that were 
announced in Theorem~\ref{freeenegvar}. These variational formulas are stated 
in Theorems~\ref{varfloc} and \ref{varfloc1} below and imply part (i) of 
Theorem~\ref{freeenegvar}. In Section~\ref{S3.3} we state additional properties 
that imply parts (ii) and (iii). 


\subsection{Quenched excess free energy and critical curve}
\label{S3.1}

Let
\begin{equation}
\label{Zaltdef}
\widetilde Z_{n,0}^{\beta,h,\omega}
= E\left(\exp\left[\beta\sum_{k=1}^n 
(\omega_k+h)\,[{\rm sign}(\pi_{k-1},\pi_k)-1]\right]\,1_{\{\pi_n=0\}}\right),
\end{equation}
which differs from $\widetilde Z_n^{\beta,h,\omega}$ in (\ref{queexpartsum}) because of 
the extra indicator $1_{\{\pi_n=0\}}$. This indicator is harmless in the limit as 
$n\to\infty$ (see Bolthausen and den Hollander~\cite{BodHo97}, Lemma 2) and is added 
for convenience. To derive a variational expression for $g^\mathrm{que}(\beta,h)
=\lim_{n\to\infty} \frac{1}{n}\log \widetilde Z_{n,0}^{\beta,h,\omega}$ $\o-a.s.$, we 
use Theorem~\ref{qLDP} with
\begin{equation}
\label{convert}
X= \omega, \quad E = \R, \quad \widetilde{E} = \cup_{k\in\N}\,\R^k, 
\quad \nu\in\cP(\R), \quad \rho\in\cP(\N),
\end{equation}
where $\nu$ satisfies \eqref{mgffin} and $\rho$ satisfies (\ref{rhocond}), with
$\rho(n) = P(\{\pi\in\Pi\colon\,\pi_k \neq 0\,\,\forall\,1 \leq k < n,\pi_n=0\})$,
$n\in\N$, the excursion length distribution.

Abbreviate
\begin{equation}
\label{Csetdef}
\cC = \{Q\in\cP^\mathrm{inv}(\widetilde{E}^\N)\colon\,I^\mathrm{ann}(Q)<\infty\},
\qquad
\cC^\mathrm{fin} = \{Q\in\cC\colon\,m_Q<\infty\}. 
\end{equation}

\begin{theorem}
\label{varfloc}
Assume {\rm (\ref{mgffin})} and {\rm (\ref{rhocond})}. Fix $\beta,h>0$.\\
(i) The quenched excess free energy is given by
\begin{equation}
\label{gvarexp}
g^\mathrm{que}(\beta,h) = \inf\{g\in\R\colon\,S^\mathrm{que}(\beta,h;g)<0\},
\end{equation}
where
\begin{equation}
\label{Sdef}
S^\mathrm{que}(\beta,h;g) 
= \sup_{Q\in\cC^\mathrm{fin}\cap\cR}
\left[\Phi_{\beta,h}(Q)-g m_Q-I^\mathrm{ann}(Q)\right]
\end{equation}
with
\begin{eqnarray}
\label{Phidef}
\Phi_{\beta,h}(Q) &=& \int_{\widetilde{E}} (\widetilde\pi_1Q)(dy)\,
\log \phi_{\beta,h}(y),\\
\label{phidef}
\phi_{\beta,h}(y) &=& \tfrac12\left(1+e^{-2\beta h\,\tau(y)-2\beta\,\sigma(y)}\right),
\end{eqnarray}
where $\widetilde\pi_1\colon\,\widetilde{E}^\N\to\widetilde{E}$ is the projection 
onto the first word, i.e., $\widetilde\pi_1Q=Q\circ\widetilde\pi_1^{-1}$, and 
$\tau(y),\sigma(y)$ are the length, respectively, the sum of the letters in the 
word $y$.\\ 
(ii) An alternative variational formula at $g=0$ is $S^\mathrm{que}(\beta,h;0)
=S_{*}^\mathrm{que}(\beta,h)$ with
\begin{equation}
\label{Sdefalt}
S_*^\mathrm{que}(\beta,h) = \sup_{Q\in\cC^\mathrm{fin}} 
\left[\Phi_{\beta,h}(Q)-I^\mathrm{que}(Q)\right].
\end{equation}
(iii) The function $g \mapsto S^\mathrm{que}(\beta,h;g)$ is lower semi-continuous, convex 
and non-increasing on $\R$, is infinite on $(-\infty,0)$, and is finite, continuous and 
strictly decreasing on $(0,\infty)$. 
\end{theorem}

\begin{proof}
The proof comes in 7 steps. Throughout the proof $\beta,h>0$ are fixed.

\medskip\noindent
{\bf 1.}
Let $t_n=t_n(\pi)$ denote the number of excursions in $\pi$ away from the interface
(recall that $\pi_n=0$ in (\ref{Zaltdef})). For $i=1,\ldots,t_n$, let $I_i=I_i(\pi)$ 
denote the $i$-th excursion interval in $\pi$. Then
\begin{equation}
\label{Hamexcsplit}
\beta\sum_{k=1}^n (\omega_k+h)[{\rm sign}(\pi_{k-1},\pi_k)-1] 
= \beta\sum_{i=1}^{t_n} \sum_{k \in I_i} (\omega_k+h)[{\rm sign}(\pi_{k-1},\pi_k)-1].  
\end{equation}
During the $i$-th excursion, $\pi$ cuts out the word $\omega_{I_i}=(\omega_k)_{k \in I_i}$ 
from $\omega$. Each excursion can be either above or below the interface, with probability 
$\tfrac12$ each, and so the contribution to $\widetilde Z_{n,0}^{\beta,h,\omega}$ in 
\eqref{Zaltdef} coming from the $i$-th excursion is
\begin{equation}
\label{jexcont}
\psi_{\beta,h}^\omega(I_i) 
= \tfrac12\left(1 + \exp\left[-2\beta \sum_{k \in I_i} (\omega_k+h)\right]\right).
\end{equation}  
Hence, putting $I_i=(k_{i-1},k_i]\cap\N$, we have
\begin{equation}
\label{Zniddef}
\widetilde Z_{n,0}^{\beta,h,\omega} = \sum_{N\in\N}\,\, \sum_{0=k_0<k_1<\cdots<k_N=n}\,\, 
\prod_{i=1}^N \rho(k_i-k_{i-1})\,\,\psi_{\beta,h}^\omega\big((k_{i-1},k_i]\big).
\end{equation} 
Summing on $n$, we get
\begin{equation}
\label{ZFNrel}
\sum_{n\in\N} e^{-gn}\,\widetilde Z_{n,0}^{\beta,h,\omega} = \sum_{N\in\N} F_N^{\beta,h,\omega}(g),
\qquad g \in [0,\infty),
\end{equation}
with (recall \eqref{rhoz})
\begin{equation}
\label{FNdef}
F_N^{\beta,h,\omega}(g) = \cN(g)^N\sum_{0=k_0<k_1<\cdots<k_N<\infty}\,\, 
\left(\prod_{i=1}^N \rho_g(k_i-k_{i-1})\right)\,\exp\left[\sum_{i=1}^N 
 \log \psi_{\beta,h}^\omega\big((k_{i-1},k_i]\big)\right].
\end{equation}

\medskip\noindent
{\bf 2.} 
Let
\begin{equation}
\label{empprocomega}
R_N^\omega = \frac{1}{N} \sum_{i=1}^N 
\delta_{\widetilde{\theta}^i(\omega_{I_1},\ldots,\omega_{I_N})^{\mathrm{per}}}
\end{equation}
denote the \emph{empirical process of $N$-tuples of words} in $\omega$ cut out by the
successive excursions. Then (\ref{FNdef}) gives
\begin{equation}
\label{FNexpr}
\begin{aligned}
F_N^{\beta,h,\omega}(g) 
&= \cN(g)^N\,E^\ast_g\left(\exp\left[N\int_{\widetilde{E}}
(\widetilde\pi_1R_N^\omega)(dy)\,\log \phi_{\beta,h}(y)\right]\right)\\
&= \cN(g)^N\,E^\ast_g\big(\exp\big[N\Phi_{\beta,h}(R_N^\omega)\big]\big)
\end{aligned}
\end{equation}
with $\Phi_{\beta,h}$ and $\phi_{\beta,h}$ defined in (\ref{Phidef}--\ref{phidef}). 
Next, let
\begin{equation}
\label{Slim}
\bar S^\mathrm{que}(\beta,h;g) = \limsup_{N\to\infty} 
\frac{1}{N} \log F_N^{\beta,h,\omega}(g), \qquad g \in [0,\infty),
\end{equation}
and note that the limsup exists and is constant (possibly infinity) $\omega$-a.s.\ because it 
is measurable w.r.t.\ the tail sigma-algebra of $\omega$ (which is trivial). By 
\eqref{exfreeenegdefalt}, the left-hand side of (\ref{ZFNrel}) is a power series that converges 
for $g>g^\mathrm{que}(\beta,h)$ and diverges for $g<g^\mathrm{que}(\beta,h)$. Hence we have
\begin{equation}
\label{roc}
g^\mathrm{que}(\beta,h) = \inf\{g\in\R\colon\,\bar{S}^\mathrm{que}(\beta,h;g)<0\}.
\end{equation}
Below we will see that $g \mapsto \bar{S}^\mathrm{que}(\beta,h;g)$ is strictly 
decreasing when finite, so that $\bar{S}^\mathrm{que}(\beta,h;g)$ changes sign
precisely at $g=g^\mathrm{que}(\beta,h)$.

\medskip\noindent
{\bf 3.} A \emph{naive} application of Varadhan's lemma to (\ref{FNexpr}--\ref{Slim}) based 
on the quenched LDP in Theorem~\ref{qLDP} yields that 
\begin{equation}
\label{Seqs}
\bar{S}^\mathrm{que}(\beta,h;g) = \log\cN(g) 
+ \sup_{Q\in\cP^\mathrm{inv}(\widetilde{E}^\N)} 
\left[\Phi_{\beta,h}(Q)-I^\mathrm{que}_g(Q)\right].
\end{equation}
This variational formula brings us close to where we want, because Lemma~\ref{rhozre} 
and the formulas for $I^\mathrm{que}_g(Q)$ given in Theorem~\ref{qLDP} tell us that
\begin{equation}
\label{Sext}
\mbox{r.h.s. }\eqref{Seqs} 
= \left\{\begin{array}{ll}
\sup\limits_{Q\in\cR}
\left[\Phi_{\beta,h}(Q)-gm_Q-I^\mathrm{ann}(Q)\right], 
&\mbox{if } g \in (0,\infty),\\
\sup\limits_{Q\in\cP^\mathrm{inv}(\widetilde{E}^\N)} 
\left[\Phi_{\beta,h}(Q)-I^\mathrm{que}(Q)\right], &\mbox{if } g=0,
\end{array}
\right.
\end{equation}
which is the same as the variational formulas in \eqref{Sdef} and \eqref{Sdefalt}, 
except that the suprema in \eqref{Sext} are not restricted to $\cC^\mathrm{fin}$. 
Unfortunately, the application of Varadhan's lemma is \emph{problematic}, because 
$Q\mapsto m_Q$ and $Q\mapsto\Phi_{\beta,h}(Q)$ are neither bounded nor continuous 
in the weak topology. The proof of (\ref{Seqs}--\ref{Sext}) therefore requires an 
\emph{approximation argument}, which is written out in Appendix~\ref{appB} and is
valid for $g \in (0,\infty)$. This approximation argument also shows how the restriction 
to $\cC^\mathrm{fin}$ comes in. This restriction is needed to make the variational 
formulas proper, namely, it is shown in Appendix~\ref{appA} that if $I^\mathrm{ann}$ 
is finite, then also $\Phi_{\beta,h}$ is finite.   
Thus, we have
\begin{equation}
\label{Sfinalpre}
\bar{S}^\mathrm{que}(\beta,h;g) = S^\mathrm{que}(\beta,h;g), 
\qquad g \in (0,\infty).
\end{equation}

\medskip\noindent
{\bf 4.} 
To include $g \in (-\infty,0)$ in (\ref{Sfinalpre}) we argue as follows. We see from 
(\ref{Phidef}--\ref{phidef}) and (\ref{FNexpr}) that $F_N^{\beta,h,\omega}(g) \geq 
\left[\tfrac12\cN(g)\right]^N$. Since $\cN(g)=\infty$ for $g\in (-\infty,0)$, it follows 
from (\ref{Slim}) that $\bar{S}^\mathrm{que}(\beta,h;g)=\infty$ for $g\in (-\infty,0)$. 
Moreover, we have
\begin{equation}
\label{Slb}
S^\mathrm{que}(\beta,h;g) \geq \log(\tfrac12) + \sup_{\rho'\in\cP(\N)}
\left[-g m_{\rho'}-h(\rho'\mid \rho)\right],
\end{equation}
which is obtained from (\ref{Sdef}--\ref{phidef}) by picking $Q=q'^{\otimes\N}$ with 
$q'(dx_1,\ldots,dx_m)=\rho'(m)\nu(dx_1)\times\cdots\times\nu(dx_m)$, $m\in\N$, $x_1,
\ldots,x_m \in \R$ (compare with \eqref{q0def}). By picking $\rho'(m)=\delta_{mL}$, 
$m\in\N$, with $L\in\N$ arbitrary, we get from \eqref{Slb} that $S^\mathrm{que}
(\beta,h;g) \geq \log(\tfrac12) - gL + \log\rho(L)$. Letting $L\to\infty$ and using
(\ref{rhocond}), we obtain that $S^\mathrm{que}(\beta,h;g)=\infty$ for $g\in (-\infty,0)$. 
Thus, \eqref{Sfinalpre} extends to
\begin{equation}
\label{Sfinal}
\bar{S}^\mathrm{que}(\beta,h;g) = S^\mathrm{que}(\beta,h;g), 
\qquad g\in\R\backslash \{0\}.
\end{equation}

\begin{figure}[htbp]
\vspace{1.5cm}
\begin{minipage}[hbt]{4cm}
\centering
\setlength{\unitlength}{0.3cm}
\begin{picture}(8,8)(-7,-2.2)
\put(-2,0){\line(10,0){10}}
\put(0,-6){\line(0,10){14}}
{\thicklines
\qbezier(2,-2.5)(0,1.5)(0,4)
\qbezier(-3,7)(-1.5,7)(-.2,7)
}
\put(8.5,-0.2){$g$}
\put(-2,9.5){$S^\mathrm{que}(\beta,h;g)$}
\put(-2,7.5){$\infty$}
\put(0,4){\circle*{.5}}
\put(0,7){\circle{.5}}
\end{picture}
\vspace{1.2cm}
\begin{center}
\qquad (1) $h < h^\mathrm{que}_c(\beta)$
\end{center}
\end{minipage}
\begin{minipage}[hbtb]{4cm}
\centering
\setlength{\unitlength}{0.3cm}
\begin{picture}(8,8)(-7,-2.2)
\put(-2,0){\line(8,0){10.5}}
\put(0,-6){\line(0,12){14}}
{\thicklines
\qbezier(2.5,-4)(1,-2.5)(0,0)
\qbezier(-3,7)(-1.5,7)(-.2,7)
}
\put(9,-0.2){$g$}
\put(-2,9.5){$S^\mathrm{que}(\beta,h;g)$}
\put(-2,7.5){$\infty$}
\put(0,0){\circle*{.5}}
\put(0,7){\circle{.5}}
\end{picture}
\vspace{1.2cm}
\begin{center}
\qquad (2) $h = h^\mathrm{que}_c(\beta)$
\end{center}
\end{minipage}
\begin{minipage}[hbt]{4cm}
\centering
\setlength{\unitlength}{0.3cm}
\begin{picture}(8,8)(-7,-2.2)
\put(-2,0){\line(8,0){10}}
\put(0,-6){\line(0,14){14}}
{\thicklines
\qbezier(2.5,-6.5)(1,-5.5)(0,-2)
\qbezier(-3,7)(-1.5,7)(-.2,7)
}
\put(8.5,-0.2){$g$}
\put(-2,9.5){$S^\mathrm{que}(\beta,h;g)$}
\put(-2,7.5){$\infty$}
\put(0,-2){\circle*{.5}}
\put(0,7){\circle{.5}}
\end{picture}
\vspace{1.2cm}
\begin{center}
\qquad (3) $h > h^\mathrm{que}_c(\beta)$
\end{center}
\end{minipage}
\vspace{-.4cm}
\begin{center}
\caption{\small Qualitative picture of $g \mapsto S^\mathrm{que}(\beta,h;g)$ 
for $\beta,h>0$.}
\label{fig-varfe}
\end{center}
\vspace{-.5cm}
\end{figure}
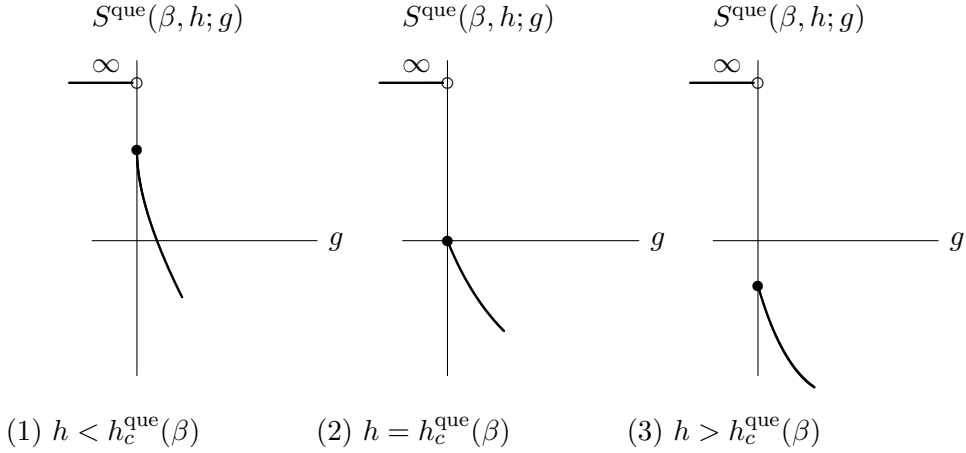

\medskip\noindent
{\bf 5.}
To complete the proof of (i) and (ii), we need to include $g=0$ in \eqref{Sfinal} and derive 
the alternative variational formula for $S^\mathrm{que}(\beta,h;0)$ given in \eqref{Sdefalt}.
In Appendix~\ref{appC} we will show that
\begin{equation}
\label{Sin1}
\bar{S}^\mathrm{que}(\beta,h;0+) \geq S^\mathrm{que}(\beta,h;0), \qquad
\bar{S}^\mathrm{que}(\beta,h;0+) \geq S_*^\mathrm{que}(\beta,h),
\end{equation}
where $\bar{S}^\mathrm{que}(\beta,h;0+)=\lim_{g \downarrow 0} \bar{S}^\mathrm{que}(\beta,h;g)$. 
Moreover, by \eqref{Sdef} and \eqref{Sfinalpre}, we have 
\begin{equation}
\label{Sin2}
\bar{S}^\mathrm{que}(\beta,h;0+) = S^\mathrm{que}(\beta,h;0+) \leq S^\mathrm{que}(\beta,h;0).
\end{equation}
Furthermore, from \eqref{Sdef} and \eqref{Sdefalt} it follows that  
\begin{equation}
\label{Sin3}
\begin{split}
S_*^\mathrm{que}(\beta,h) = \sup_{Q\in\cC^\mathrm{fin}} 
\left[\Phi_{\beta,h}(Q)-I^\mathrm{que}(Q)\right]
\geq \sup_{Q\in\cC^\mathrm{fin}\cap \cR}\left[\Phi_{\beta,h}(Q)-I^\mathrm{que}(Q)\right]
= S^\mathrm{que}(\beta,h;0),
\end{split}
\end{equation}
where the last equality uses that $I^\mathrm{que}=I^\mathrm{ann}$ on $\cC^\mathrm{fin}\cap\cR$ 
(recall \eqref{Requiv}). Combining (\ref{Sin1}--\ref{Sin3}), we obtain
\begin{equation}
\label{Seqsvar}
\bar{S}^\mathrm{que}(\beta,h;0+) = S^\mathrm{que}(\beta,h;0) =  S_*^\mathrm{que}(\beta,h).
\end{equation}
Hence \eqref{Sfinal} indeed extends to
\begin{equation}
\label{Sfinalext}
\bar{S}^\mathrm{que}(\beta,h;g) = S^\mathrm{que}(\beta,h;g), 
\qquad g\in\R.
\end{equation}
Combine \eqref{roc}, \eqref{Sdefalt} and (\ref{Seqsvar}--\ref{Sfinalext}) to get
parts (i) and (ii).

\medskip\noindent
{\bf 6.}
In Appendix~\ref{appA} we will prove that, for every $g \in (0,\infty)$, $\omega$-a.s.\ 
there exists a $K(\omega,g)<\infty$ such that
\begin{equation}
\label{Fub}
-gm_{R_N^\omega} + \Phi_{\beta,h}(R_N^\omega) \leq K(\omega,g) 
\qquad \forall\,N\in\N.
\end{equation}
Via (\ref{FNexpr}--\ref{Slim}) this implies that $\bar S^{\rm que}(\beta,h;g)<\infty$ 
for $g\in(0,\infty)$.

\medskip\noindent
{\bf 7.}
By (\ref{Sdef}), $g \mapsto S^\mathrm{que}(\beta,h;g)$ is a supremum of functions 
that are finite and linear on $\R$. Hence, $g \mapsto S^\mathrm{que}(\beta,h;g)$ is 
lower semi-continuous and convex on $\R$ and, being finite on $(0,\infty)$, is 
continuous on $(0,\infty)$. Moreover, since $m_Q\geq 1$, it is strictly decreasing 
on $(0,\infty)$ as well. This completes the proof of part (iii). 
\end{proof}

Fig.~\ref{fig-varfe} provides a sketch of $g \mapsto S^\mathrm{que}(\beta,h;g)$ for
$(\beta,h)$ drawn from $\cL^\mathrm{que}$, $\partial\cD^\mathrm{que}$ and $\mathrm{int}
(\cD^\mathrm{que})$, respectively, and completes the variational characterization in 
Theorem~\ref{varfloc}. In Section~\ref{S3.3} we look at $h \mapsto S^\mathrm{que}
(\beta,h;0)$ and obtain the picture drawn in Fig.~\ref{fig-varhc}, which is crucial
for our analysis.

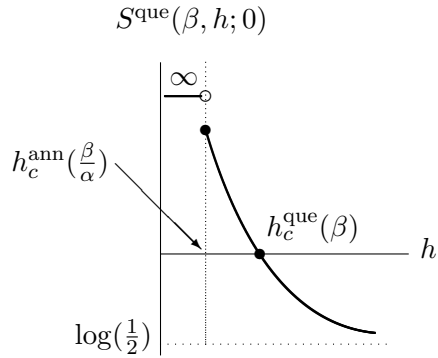
\begin{figure}[htbp]
\vspace{3.5cm}
\begin{center}
\setlength{\unitlength}{0.3cm}
\begin{picture}(8,4)(0,-4)
\put(0,0){\line(11,0){11}}
\put(0,-4.5){\line(0,13){13}}
{\thicklines
\qbezier(2,5.5)(4.5,-3)(9.5,-3.5)
\qbezier(.2,7)(1,7)(1.8,7)
}
\qbezier[30](0,-4)(5,-4)(10,-4)
\qbezier[80](2,-4)(2,0)(2,8.5)
\put(11.5,-0.2){$h$}
\put(-2,10){$S^\mathrm{que}(\beta,h;0)$}
\put(4.7,.8){$h^\mathrm{que}_c(\beta)$}
\put(4.4,0){\circle*{.5}}
\put(2,5.5){\circle*{.5}}
\put(2,7){\circle{.5}}
\put(-3.8,-4){$\log(\tfrac12)$}
\put(-2,4){\vector(1,-1){3.8}}
\put(-6.6,3.7){$h_c^\mathrm{ann}(\tfrac{\beta}{\alpha})$}
\put(0.4,7.5){$\infty$}
\end{picture}
\end{center}
\vspace{-0.5cm}
\caption{\small Qualitative picture of $h \mapsto S^\mathrm{que}(\beta,h;0)$ for $\beta>0$.}
\label{fig-varhc}
\end{figure}

\medskip\noindent
{\bf Remark:} A major advantage of the variational formula in (\ref{Sdefalt}) over 
the one in (\ref{Sdef}) at $g=0$ is that the supremum runs over $\cC^\mathrm{fin}$ 
rather than $\cC^\mathrm{fin} \cap \cR$. This will be crucial for the proof of 
Corollaries~\ref{hcubstrict} and \ref{hclbstrict} in Sections~\ref{S5} and \ref{S6}, 
respectively. 

\medskip\noindent
{\bf Remark:} In Section~\ref{S6} we will show that 
\begin{equation}
\label{BGpos}
S^\mathrm{que}\big(\beta,h_c^\mathrm{ann}(\tfrac{\beta}{\alpha});0\big)>0. 
\end{equation}
It will turn out that $S^\mathrm{que}(\beta,h_c^\mathrm{ann}(\tfrac{\beta}{\alpha});0)
<\infty$ for some choices of $\rho$, but we do not know whether it is finite in general.  


\subsection{Annealed excess free energy and critical curve}
\label{S3.2}

In order to exploit Theorem~\ref{varfloc}, we need an analogous variational expression
for the annealed excess free energy defined in (\ref{annexpartsum}--\ref{annexfreeeneg}).
This variational expression will serve as a \emph{comparison object} and will be crucial
for the proof of Corollaries~\ref{freeeneggap}--\ref{hclbstrict}.
 
\begin{theorem}
\label{varfloc1}
Assume {\rm (\ref{mgffin})} and {\rm (\ref{rhocond})}. Fix $\beta,h>0$.\\
(i) The annealed excess free energy is given by
\begin{equation}
\label{annf}
g^\mathrm{ann}(\beta,h) = \inf\{g\in\R\colon\,S^\mathrm{ann}(\beta,h;g)<0\},
\end{equation} 
where
\begin{equation}
\label{Sandef}
S^\mathrm{ann}(\beta,h;g) =
\sup_{Q\in\cC^\mathrm{fin}}\left[\Phi_{\beta,h}(Q)-gm_Q-I^\mathrm{ann}(Q)\right]. 
\end{equation}
(ii) The function $g\mapsto S^\mathrm{ann}(\beta,h;g) $ is lower semi-continuous, convex 
and non-increasing on $\R$, infinite on $(-\infty,g^{\rm ann}(\beta,h))$, and finite, 
continuous and strictly decreasing on $[g^{\rm ann}(\beta,h),\infty)$.
\end{theorem}
 
\begin{proof} 
Throughout the proof $\beta,h>0$ are fixed.

\medskip\noindent
(i) Replacing $\widetilde Z_n^{\beta,h,\omega}$ by $\E(\widetilde Z_n^{\beta,h,\omega})$ 
in (\ref{ZFNrel}--\ref{FNdef}), we obtain from (\ref{Slim}) that 
\begin{equation}
\label{Slimann}
\bar{S}^\mathrm{ann}(\beta,h;g) = \limsup_{N\to\infty} 
\frac{1}{N} \log \E\left(F_N^{\beta,h,\omega}(g)\right).  
\end{equation} 
Using (\ref{q0def}--\ref{rhoz}), (\ref{phidef}), (\ref{jexcont}) and (\ref{FNdef})
we compute
\begin{equation}
\label{Subproj}
\bar{S}^\mathrm{ann}(\beta,h;g) = \log \cN(\beta,h;g)
\end{equation}
with
\begin{equation}  
\label{Nid}
\begin{aligned}
\cN(\beta,h;g) 
&= \int_{\widetilde{E}} q_{\rho,\nu}(dy)\,e^{-g\tau(y)} \phi_{\beta,h}(y)\\ 
&= \sum_{m\in\N} \int_{x_1,\dots,x_m\in\R} \rho(m)\,
\nu(dx_1)\times\cdots\times\nu(dx_m)
\,e^{-gk}\,\tfrac12\Big(1+e^{-2\beta h m-2\beta[x_1+\cdots+x_m]}\Big)\\
&= \tfrac12 \sum_{m\in\N} \rho(m)\,e^{-gm}\, 
+ \tfrac12 \sum_{m\in\N} \rho(m)\,e^{-gm}\,\left[e^{-2\beta h + M(2\beta)}\right]^m\\
&= \tfrac12\,\cN(g) + \tfrac12\,\cN\big(g-[M(2\beta)-2\beta h]\big),
\end{aligned}
\end{equation}
where $\cN(g)$ is the normalization constant in \eqref{rhoz}. The right-hand side 
of \eqref{Subproj} has the behavior as sketched in Fig.~\ref{fig-varfe1}. It is 
therefore immediate that (\ref{annf}--\ref{Sandef}) is consistent with (\ref{gannid}), 
provided we have 
\begin{equation}
\label{S=Sbar}
S^\mathrm{ann}(\beta,h;g) = \bar{S}^\mathrm{ann}(\beta,h;g).
\end{equation}
To prove this equality we must distinguish three cases.

\medskip\noindent
(I) $g(\beta,h) \geq g^\mathrm{ann}(\beta,h) = 0 \vee [M(2\beta)-2\beta h]$. The proof 
comes in 2 steps. Note that the right-hand side of (\ref{Nid}) is finite.

\medskip\noindent
{\bf 1.} Note that $\Phi_{\beta,h}(Q)$ defined in (\ref{Phidef}) is a functional of 
$\widetilde\pi_1Q$. Moreover, by (\ref{spentrdef}),
\begin{equation}
\label{Hiidred}
\inf_{ {Q\in\cP^{\mathrm{inv}}(\widetilde{E}^\N)} \atop {\widetilde\pi_1 Q=q} } 
H(Q \mid q_{\rho,\nu}^{\otimes\N}) = h(q \mid q_{\rho,\nu}) \qquad \forall\,q \in 
\cP(\widetilde{E})
\end{equation}  
with the infimum \emph{uniquely} attained at $Q = q^{\otimes\N}$, where the right-hand 
side denotes the relative entropy of $q$ w.r.t.\ $q_{\rho,\nu}$. (The uniqueness of
the minimum is easily deduced from the strict convexity of relative entropy on finite
cylinders.) Consequently, the variational formula in \eqref{Sandef} reduces to  
\begin{equation}
\label{varred2}
\begin{split}
S^\mathrm{ann}(\beta,h;g)
&= \sup_{ {q \in \cP(\widetilde{E})} \atop {m_q<\infty,\,h(q \mid q_{\rho,\nu})<\infty} } 
\Big\{\int_{\widetilde{E}} q(dy)\,[-g\tau(y)+\log\phi_{\beta,h}(y)] 
- h(q \mid q_{\rho,\nu})\Big\}\cr
&=\sup_{ {q \in \cP(\N\times\R)} \atop {m_q<\infty,\,h(q \mid q_{\rho,\nu})<\infty} } 
\Big\{ \int_{\widetilde{E}} q(dy)\,[-g\tau(y)+\log\phi_{\beta,h}(y)]\\ 
&\qquad\qquad\qquad\qquad\qquad - \int_{\widetilde{E}} q(dy) 
\log \left(\frac{q(dy)}{q_{\rho,\nu}(dy)}\right)\Big\} 
\end{split}
\end{equation}
with $\phi_{\beta,h}(y)$ defined in (\ref{phidef}) and $m_q=\int_{\widetilde{E}} 
q(dy)\tau(y)$. 

\medskip\noindent
{\bf 2.}
Define
\begin{equation}
\label{qsol}
q_{\beta,h;g}(dy) = \frac{1}{\cN(\beta,h;g)}\,q_{\rho,\nu}(dy)\,
e^{-g\tau(y)}\,\phi_{\beta,h}(y), \qquad y\in \widetilde{E},
\end{equation}
with $\cN(\beta,h;g)$ the normalizing constant in (\ref{Nid}) (which is finite because 
$g\geq [M(2\beta)-2\beta h]$). Then the term between braces in the second equality of (\ref{varred2}) 
can be rewritten as
\begin{equation}
\label{barcrew}
\log \cN(\beta,h;g) - h(q \mid q_{\beta,h;g}),
\end{equation}
and so we have two cases:
\begin{itemize}
\item[(1)]
if both $m_{q_{\beta,h;g}}<\infty$ and $h(q_{\beta,h;g} \mid q_{\rho,\nu})<\infty$, then 
the supremum in (\ref{varred2}) has a unique maximizer at $q=q_{\beta,h;g}$;
\item[(2)] 
if $m_{q_{\beta,h;g}}=\infty$ and/or $h(q_{\beta,h;g} \mid q_{\rho,\nu})=\infty$,
then any maximizing sequence $(q_l)_{l\in\N}$ with $m_{q_l}<\infty$ and $h(q_l \mid 
q_{\rho,\nu})<\infty$ for all $l\in\N$ satisfies ${\rm w}-\lim_{l\to\infty} q_l = 
q_{\beta,h;g}$ (weak limit).
\end{itemize}
In both cases
\begin{equation}
\label{Subprojeq}
S^\mathrm{ann}(\beta,h;g) = \log \cN(\beta,h;g),
\end{equation} 
which settles (\ref{S=Sbar}) in view of (\ref{Subproj}). 

\medskip\noindent
(II) $g<[M(2\beta)-2\beta h]$.
It follows from (\ref{Subproj}--\ref{Nid}) that $\bar{S}^\mathrm{ann}(\beta,h,g)=\infty$. 
We therefore need to show that $S^\mathrm{ann}(\beta,h;g)=\infty$ as well. For $L\in\N$, 
let $q^L_\beta\in\cP(\widetilde{E})$ be defined by
\begin{equation}
\label{bqbetarels}
q^L_\beta(dx_1,\dots,dx_m) 
= \delta_{mL}\,\nu_\beta(dx_1)\times\cdots\times\nu_\beta(dx_m), 
\qquad m\in\N,\,x_1,\dots,x_m\in\R,
\end{equation}
where $\nu_\beta\in\cP(\R)$ is defined by
\begin{equation}
\label{nubetadefalt}
\nu_\beta(dx) = e^{-2\beta x-M(2\beta)}\,\nu(dx), \qquad x \in\R.
\end{equation}

\begin{figure}[htbp]
\vspace{1cm}
\begin{minipage}[hbt]{5cm}
\centering
\setlength{\unitlength}{0.3cm}
\begin{picture}(8,8)(-5,-2.2)
\put(0,0){\line(8,0){8}}
\put(0,-4){\line(0,12){12}}
{\thicklines
\qbezier(3,-5)(2,-3.5)(1,-1)
\qbezier(-3,7)(-1.5,7)(.8,7)
}
\qbezier[40](1,-1)(1,-3)(1,7)
\put(8.3,-0.2){$g$}
\put(0,9){$S^\mathrm{ann}(\beta,h;g)$}
\put(-2,7.5){$\infty$}
\put(1,-1){\circle*{.5}}
\put(1,7){\circle{.5}}
\end{picture}
\vspace{1.2cm}
\begin{center}
{\qquad (1) $h < h^\mathrm{ann}_c(\beta)$}
\end{center}
\end{minipage}
\begin{minipage}[hbtb]{5cm}
\centering
\setlength{\unitlength}{0.3cm}
\begin{picture}(8,8)(-5,-2.2)
\put(0,0){\line(8,0){8.5}}
\put(0,-4){\line(0,12){12}}
{\thicklines
\qbezier(3,-4.5)(1,-2.5)(0,0)
\qbezier(-3,7)(-1.5,7)(-.2,7)
}
\qbezier[40](4,0)(4,0)(0,0)
\qbezier[40](4,0)(4,0)(0,0)  
\put(8.7,-0.2){$g$}
\put(0,9){$S^\mathrm{ann}(\beta,h;g)$}
\put(-2,7.5){$\infty$}
\put(0,0){\circle*{.5}}
\put(0,7){\circle{.5}}
\end{picture}
\vspace{1.3cm}
\begin{center}
\qquad (2) $h = h^\mathrm{ann}_c(\beta)$
\end{center}
\end{minipage}
\begin{minipage}[hbt]{5cm}
\centering
\setlength{\unitlength}{0.3cm}
\begin{picture}(8,8)(-5,-2.2)
\put(0,0){\line(8,0){8}}
\put(0,-4){\line(0,12){12}}
{\thicklines
\qbezier(4,-5)(1,-3.5)(0,-1)
\qbezier(-3,7)(-1.5,7)(-.2,7)
}
\put(8.3,-0.2){$g$}
\put(0,9){$S^\mathrm{ann}(\beta,h;g)$}
\put(-2,7.5){$\infty$}
\put(0,-1){\circle*{.5}}
\put(0,7){\circle{.5}}
\end{picture}
\vspace{1.2cm}
\begin{center}
\qquad (3) $h > h^\mathrm{ann}_c(\beta)$
\end{center}
\end{minipage}
\begin{center}
\caption{\small Qualitative picture of $g \mapsto S^\mathrm{ann}(\beta,h;g)$ for $\beta,h>0$. 
Compare with Fig.~\ref{fig-varfe}.}
\label{fig-varfe1}
\end{center}
\vspace{-.4cm}
\end{figure}
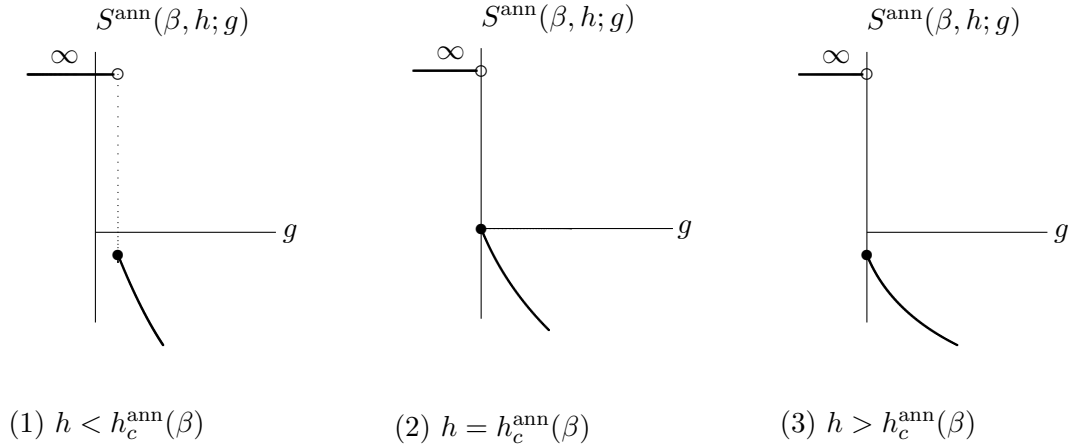

Put $Q^L_\beta = (q^L_\beta)^{\otimes \N}$. Then $m_{Q^L_\beta}=L$, while
\begin{equation}
\label{imp1}
\begin{split}
I^\mathrm{ann}(Q_\beta^L)
&= H(Q^L_\beta \mid q_{\rho,\nu}^{\otimes\N})\\ 
&= h(q^L_\beta \mid q_{\rho,\nu})\\
&= \int_{\widetilde{E}} q^L_\beta(dy)\,\frac{dq^L_\beta}{dq_{\rho,\nu}}(y)\\
&= -\log \rho(L) + L h(\nu_\beta \mid \nu)\\
&= -\log \rho(L) + L \int_{\R} \nu_\beta(dx)\,\log\left(e^{-2\beta x-M(2\beta)}\right)\\
&= -\log \rho(L)-L\left[2\beta\,\E_{\nu_\beta}(\omega_1)+M(2\beta)\right]
\end{split}
\end{equation}
and
\begin{equation}
\label{imp2}
\begin{split}
\Phi_{\beta,h}(Q^L_\beta)
&= \int_{\widetilde{E}} q^L_\beta(dy)\,\log\phi_{\beta,h}(y)\\
&= \int_{\R^L} \nu_\beta(dx_1)\times\cdots\times\nu_\beta(dx_L)\,
\log\left(\tfrac12\left[1+e^{-2\beta hL-2\beta[x_1+\cdots+x_L]}\right]\right)\\
&\geq \log(\tfrac12)- L\left[2\beta \E_{\nu_\beta}(\o_1)+2\beta h\right].
\end{split}
\end{equation}
It follows that
\begin{equation}
\label{imp3}
\Phi_{\beta,h}(Q^L_\beta)-g\,m_{Q^L_\beta}-I^\mathrm{ann}(Q_\beta^L)
\geq \log(\tfrac12) + \log \rho(L) + L\left[M(2\beta)-2\beta h-g\right],
\end{equation}
which tends to infinity as $L\to\infty$ (use (\ref{rhocond}) and let $L\to\infty$
along the support of $\rho$).

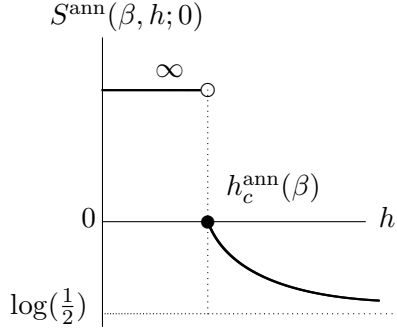
\begin{figure}[htbp]
\vspace{2.7cm}
\begin{center}
\setlength{\unitlength}{0.35cm}
\begin{picture}(8,4)(0,-3.5)
\put(0,0){\line(10,0){10}}
\put(0,-4){\line(0,11){11}}
{\thicklines
\qbezier(4,0)(5,-2.7)(10.5,-3)
\qbezier(0,5)(2,5)(3.7,5)
}
\qbezier[70](0,-3.5)(3.5,-3.5)(11,-3.5)
\qbezier[40](4,-3.5)(4,1)(4,5)
\put(10.5,-0.2){$h$}
\put(-2,7.5){$S^\mathrm{ann}(\beta,h;0)$}
\put(4.7,1){$h^\mathrm{ann}_c(\beta)$}
\put(2,5.5){$\infty$}
\put(-.8,-.2){$0$}
\put(-3.5,-3.5){$\log(\tfrac12)$}
\put(4,0){\circle*{.5}}
\put(4,5){\circle{.5}}
\end{picture}
\end{center}
\vspace{-.2cm}
\caption{\small Qualitative picture of $h \mapsto S^\mathrm{ann}(\beta,h;0)$ for $\beta>0$.
Compare with Fig.~\ref{fig-varhc}.}
\label{fig-varhcbar}
\end{figure}

\medskip\noindent
(III) $M(2\beta)-2\beta h<0$ and $g \in [M(2\beta)-2\beta h,0)$.
Repeat the argument in (\ref{imp1}--\ref{imp3}) with $Q^L_\beta$ replaced
by $Q^L_0$ and keep only the first term in the right-hand side of (\ref{imp3}). 
This gives
\begin{equation}
\label{imp3alt}
\Phi_{\beta,h}(Q^L_0)-g\,m_{Q^L_0}-I^\mathrm{ann}(Q^L_0)
\geq \log(\tfrac12) + \log \rho(L) - Lg,
\end{equation}
which tends to infinity as $L\to\infty$ for $g<0$.
\end{proof}

Fig.~\ref{fig-varfe1} provides a sketch of $g \mapsto S^\mathrm{ann}(\beta,h;g)$ for
$(\beta,h)$ drawn from $\cL^\mathrm{ann}$, $\partial\cD^\mathrm{ann}$ and $\mathrm{int}
(\cD^\mathrm{ann})$, respectively, and completes the variational characterization in 
Theorem~\ref{varfloc1}. Fig.~\ref{fig-varhcbar} provides a sketch of $h \mapsto 
S^\mathrm{ann}(\beta,h;0)$.


\subsection{Proof of Theorem~\ref{freeenegvar}}
\label{S3.3}

Theorems~\ref{varfloc} and \ref{varfloc1} complete the proof of part (i) of 
Theorem~\ref{freeenegvar}. From the computations carried out in Section~\ref{S3.2} 
we also get parts (ii) and (iii) for the annealed model, but to get parts (ii) 
and (iii) for the quenched model we need some further information. 

Theorem~\ref{varfloc} provides no information on $S^\mathrm{que}(\beta,h;0)$. We know 
that, for every $\beta>0$, $h\mapsto S^\mathrm{que}(\beta,h;0)$ is lower semi-continuous, 
convex and non-increasing on $(0,\infty)$. Indeed, $h\mapsto \phi_{\beta,h}(k,l)$ is 
continuous, convex and non-increasing for all $k\in\N$ and $l\in\R$, hence $h\mapsto 
\Phi_{\beta,h}(Q)$ is lower semi-continuous, convex and non-increasing for every $Q\in
\cC^\mathrm{fin}$, and these properties are preserved under taking suprema. We know 
that $h\mapsto S^\mathrm{que}(\beta,h;0)$ is strictly negative on $(h_c^\mathrm{que}
(\beta),\infty)$. In Section~\ref{S6} we prove the following theorem, which corroborates 
the picture drawn in Fig.~\ref{fig-varhc} and completes the proof of parts (ii) and 
(iii) of Theorem~\ref{freeenegvar} for the quenched model.

\begin{theorem}
\label{asymptote}
For every $\beta>0$, 
\begin{equation}
S^\mathrm{que}(\beta,h;0)=S^\mathrm{que}_*(\beta,h)\left\{\begin{array}{ll}
= \infty &\mbox{ for } h < h_c^\mathrm{ann}(\beta/\alpha),\\
> 0 &\mbox{ for } h = h_c^\mathrm{ann}(\beta/\alpha),\\
< \infty &\mbox{ for } h > h_c^\mathrm{ann}(\beta/\alpha).
\end{array}
\right. 
\end{equation}
\end{theorem} 

We close this section with the following remark. The difference between the variational 
formulas in \eqref{Sdef} (quenched model) and \eqref{Sandef} (annealed model) is that 
the supremum in the former runs over $\cC^\mathrm{fin} \cap \cR$ while the supremum in 
the latter runs over $\cC^\mathrm{fin}$. Both involve the annealed rate function 
$I^\mathrm{ann}$. However, the restriction to $\cR$ for the quenched model allows us to 
replace $I^\mathrm{ann}$ by $I^\mathrm{que}$ (recall \eqref{Requiv}). After passing 
to the limit $g \downarrow 0$, we can remove the restriction to $\cR$ to obtain the 
alternative variational formula for the quenched model given in \eqref{Sdefalt}. The 
latter turns out to be crucial in Sections~\ref{S5} and \ref{S6}. 

Note that the two variational formulas for $g\neq 0$ are different even when $\alpha=1$, 
although in that case $I^\mathrm{ann}=I^\mathrm{que}$ (compare Theorems~\ref{aLDP} 
and \ref{qLDP}). \emph{For $\alpha=1$ the quenched and the annealed critical curves coincide,
but the free energies do not}.


\section{Proof of Corollary~\ref{freeeneggap}}
\label{S4}

\begin{proof}
The claim is trivial for $h_c^\mathrm{que}(\beta)\leq h<h_c^\mathrm{ann}(\beta)$ because
$g^\mathrm{que}(\beta,h)=0<g^\mathrm{ann}(\beta,h)$. Therefore we may assume that 
$0<h<h_c^\mathrm{que}(\beta)$. Since $I^\mathrm{que}(Q) \geq I^\mathrm{ann}(Q)$, 
(\ref{Sdef}) and (\ref{Sandef}) yield
\begin{equation}
\label{Sineqs}
S^\mathrm{que}(\beta,h;0) \leq S^\mathrm{ann}(\beta,h;0)
\end{equation}
which, via (\ref{gvarexp}) and (\ref{annf}), implies that $g^\mathrm{que}(\beta,h) \leq 
g^\mathrm{ann}(\beta,h)$, a property that is also evident from (\ref{exfreeenegdef}) and 
(\ref{annexfreeeneg}). To prove that $g^\mathrm{que}(\beta,h)<g^\mathrm{ann}(\beta,h)$ 
for $0<h<h_c^\mathrm{que}(\beta)$, we combine (\ref{Sineqs}) with Figs.~\ref{fig-varfe} 
and \ref{fig-varfe1}. First note that
\begin{equation}
S^\mathrm{que}(\beta,h;g^\mathrm{ann}(\beta,h)) \leq S^\mathrm{ann}(\beta,h;g^\mathrm{ann}
(\beta,h)) < 0, \qquad 0<h<h_c^\mathrm{ann}(\beta). 
\end{equation}
Next, for $0<h<h^\mathrm{ann}_c(\beta)$, $g \mapsto S^\mathrm{ann}(\beta,h;g)$ blows up at 
$g=g^\mathrm{ann}(\beta,h)>0$ by jumping from a strictly negative value to infinity (see 
Fig.~\ref{fig-varfe1}). Since $S^\mathrm{que}(\beta,h;g^\mathrm{ann}(\beta,h))<0$, 
and $g \mapsto S^\mathrm{que}(\beta,h;g)$ is strictly decreasing and continuous when 
finite, the claim is immediate from Theorem~\ref{freeenegvar}(ii), which says that
$S^\mathrm{que}(\beta,h;g^\mathrm{que}(\beta,h))=0$. 
\end{proof}


\section{Proof of Corollary~\ref{hcubstrict}}
\label{S5}

\begin{proof}
Throughout the proof, $\alpha>1$ and $\beta>0$ are fixed. It follows from \eqref{spentrdef} 
and the remark made below it that
\begin{equation}
\label{red1}
H(Q \mid q_{\rho,\nu}^{\otimes\N}) \geq h(\widetilde\pi_1Q \mid q_{\rho,\nu}),
\qquad H(\Psi_Q \mid \nu^{\otimes\N}) \geq h(\pi_1 \Psi_Q \mid \nu),
\end{equation}
where $\widetilde\pi_1$ is the projection onto the first \emph{word} and $\pi_1$ is 
the projection onto the first \emph{letter}. Moreover, it follows from 
\eqref{PsiQdef} that
\begin{equation}
\label{red2}
\pi_1\Psi_Q = \pi_1\Psi_{(\widetilde\pi_1 Q)^{\otimes\N}}.
\end{equation} 
Since $m_Q = m_{(\widetilde\pi_1 Q)^{\otimes\N}}= m_{(\widetilde\pi_1 Q)}$, 
(\ref{red1}--\ref{red2}) combine with \eqref{Sdefalt} to give
\begin{equation}
\label{upbound}
S_*^\mathrm{que}(\beta,h)
\leq \sup_{ {q\in\cP(\widetilde{E})} \atop {m_q<\infty} } 
\left[\int_{\widetilde{E}} q(dy)\log\phi_{\beta,h}(y) - h(q \mid q_{\rho,\nu})
- (\alpha-1) m_q h(\pi_1\psi_q \mid \nu) \right],
\end{equation}
where
\begin{equation}
\begin{aligned}
\phi_{\beta,h}(y) &= \tfrac12 \left(1 + e^{-2\beta h m-2\beta [x_1+\dots+x_m]}\right),\\
q_{\rho,\nu}(dy) &= \rho(m)\nu(dx_1)\times\cdots\times\nu(dx_m),
\end{aligned}
\end{equation}
and
\begin{equation}
(\pi_1\psi_q)(dx) = \frac{1}{m_q} \sum_{m\in\N} r(m) 
\sum_{k=1}^m q_m(E^{k-1},dx,E^{m-k}) 
\end{equation}
with the notation
\begin{equation}
q(dy) = r(m)q_m(dx_1,\dots,dx_m), \qquad y=(x_1,\dots,x_m).
\end{equation}
Let
\begin{equation}
q_{\beta,h}^*(dy) = \frac{1}{\cN(\beta,h)} q_{\rho,\nu}(dy) \phi_{\beta,h}(y)
\end{equation}
with $\cN(\beta,h)$ the normalizing constant (which equals $\cN(\beta,h;0)$ in
\eqref{Nid} and is finite for $h \geq h_c^\mathrm{ann}(\beta)=M(2\beta)/2\beta$).
Therefore, combining the first two terms in the supremum in \eqref{upbound}, we 
obtain
\begin{equation}
S_*^\mathrm{que}(\beta,h) \leq \log\cN(\beta,h) 
- \inf_{ {q\in\cP(\widetilde{E})} \atop {m_q<\infty} }
\left[h(q \mid q_{\beta,h}^*)
+ (\alpha-1) m_q h(\pi_1\psi_q \mid \nu) \right].
\end{equation}
Since $\cN(\beta,h_c^\mathrm{ann}(\beta))=1$, we have
\begin{equation}
S_*^\mathrm{que}(\beta,h_c^\mathrm{ann}(\beta)) 
\leq - \inf_{ {q\in\cP(\widetilde{E})} \atop {m_q<\infty} }
\left[h(q \mid q_{\beta,h_c^\mathrm{ann}(\beta)}^*) 
+ (\alpha-1) m_q h(\pi_1\psi_q \mid \nu) \right].
\end{equation}
The first term achieves its minimal value zero at $q=q_{\beta,h_c^\mathrm{ann}
(\beta)}^*$ (or along a minimizing sequence converging to $q_{\beta,h_c^\mathrm{ann}
(\beta)}^*$). However, $\pi_1\psi_{q_{\beta,h_c^\mathrm{ann}(\beta)}^*} 
= \tfrac12\nu+\tfrac12\nu_\beta \neq \nu$, and so we have
\begin{equation}
S_*^\mathrm{que}(\beta,h_c^\mathrm{ann}(\beta)) < 0.
\end{equation} 
Since $S_*^\mathrm{que}(\beta,h_c^\mathrm{que}(\beta)) = 0$ and $h\mapsto
S_*^\mathrm{que}(\beta,h)$ is strictly decreasing on $(h_c^\mathrm{ann}
(\beta/\alpha),\infty)$, it follows that $h_c^\mathrm{que}(\beta)<h_c^\mathrm{ann}
(\beta)$.
\end{proof}

We close this section with the following remark. As (\ref{eqnratefctexplicitalg})
shows, $I^{\mathrm{fin}}(Q)$ depends on $q_{\rho,\nu}$, the reference law defined 
in (\ref{q0def}). Since the latter depends on the full law $\rho\in\cP(\N)$ of the 
excursion lengths, it is evident from Theorem~\ref{freeenegvar} (iii) and \eqref{Sdefalt}
that the quenched critical curve is \emph{not} a function of the exponent $\alpha$ 
in (\ref{rhocond}) alone. This supports the statement made in Section~\ref{S1.5}, 
item 6.


\section{Proof of Corollary~\ref{hclbstrict}}
\label{S6}

The proof is immediate from Theorem~\ref{asymptote} (recall Fig.~\ref{fig-varhc}), which 
is proved in Sections~\ref{S6.1}--\ref{S6.3}.


\subsection{Proof for $h>h_c^\mathrm{ann}(\beta/\alpha)$}
\label{S6.1}

\begin{proof}
Recall from (\ref{FNexpr}--\ref{Slim}) and (\ref{Sfinalext}) that
\begin{equation}
\label{fracmest}
\begin{split}
S^\mathrm{que}(\beta,h;g) 
&= \limsup_{N\to\infty} \frac{1}{N} \log F_N^{\beta,h,\omega}(g)\\
&= \log \cN(g) + \limsup_{N\to\infty} \frac{1}{N} \log 
E^\ast_g\big(\exp\big[N\Phi_{\beta,h}(R_N^\omega)\big]\big).
\end{split}
\end{equation}
Abbreviate
\begin{equation}
\label{SNdef}
S_N^\o(g) = E^\ast_g\big(\exp\big[N\Phi_{\beta,h}(R_N^\omega)\big]\big)
\end{equation}
and pick 
\begin{equation}
\label{thchoice}
t = [0,1], \qquad h=h_c^\mathrm{ann}(\beta t).
\end{equation}
Then the $t$-th moment of $S_N^\o(g)$ can be estimated as (recall (\ref{jexcont}--\ref{Zniddef}))
\begin{equation}
\label{SNfmest}
\begin{split}
\E\left([S_N^\o(g)]^t\right) 
&= \E\left(\left[E^\ast_g\left(\exp\left[\sum_{i=1}^N \log \left(\tfrac12
\left[1+e^{-2\beta\sum_{k\in I_i}(\o_k+h)}\right]\right)\right]\right)\right]^t\right)\\
&= \E\left(\left[E^\ast_g\left(\prod_{i=1}^N 
\tfrac12\left[1+e^{-2\beta\sum_{k\in I_i}(\o_k+h)}\right]\right)\right]^t\right)\\
&= \E\left(\left[\sum_{0<k_1<\cdots<k_N<\infty} \left\{\prod_{i=1}^N 
\rho_g(k_i-k_{i-1})\right\}\left\{\prod_{i=1}^N \tfrac12 \left[1+e^{-2\beta 
\sum_{k\in (k_{i-1},k_i]}(\o_k+h)}\right]\right\}\right]^t\right)\\
&\leq \E\left(\sum_{0<k_1<\cdots<k_N<\infty} \left\{\prod_{i=1}^N 
\rho_g(k_i-k_{i-1})^t\right\}
\left\{\prod_{i=1}^N 2^{-t}\left[1+e^{-2\beta t\sum_{k\in (k_{i-1},k_i]}(\o_k+h)}\right]\right\}\right)\\
&= \sum_{0<k_1<\cdots<k_N<\infty} \left\{\prod_{i=1}^N \rho_g(k_i-k_{i-1})^t\right\}
\left\{\prod_{i=1}^N 2^{-t} \left[1+e^{(k_i-k_{i-1})[M(2\beta t)-2\beta th]}\right]\right\}\\
&= 2^{(1-t)N} \sum_{0<k_1<\cdots<k_N<\infty} 
\left\{\prod_{i=1}^N \rho_g(k_i-k_{i-1})^t\right\}\\
&= \left(2^{1-t} \sum_{m\in\N} \rho_g(m)^t\right)^N.
\end{split}
\end{equation}
The inequality uses that $(u+v)^t \leq u^t+v^t$ for $u,v\geq 0$ and $t\in[0,1]$, while 
the fifth equality uses that $M(2\beta t)-2\beta th=0$ for the choice of $t$ and $h$ 
in \eqref{thchoice} (recall \eqref{hcannid}).

Let $K(g)$ denote the term between round brackets in the last line of \eqref{SNfmest}.
Then, for every $\epsilon>0$, we have
\begin{equation}
\label{momest} 
\begin{split}
\Pr\left(\frac1N \log S_N^\o(g)\geq \frac{1}{t}
\left[\log K(g)+\epsilon\right]\right)
&= \Pr\left([S_N^\o(g)]^t \geq K(g)^N\,e^{N\epsilon}\right)\\
&\leq \E\left([S_N^\o(g)]^t\right)\, 
K(g)^{-N}\,e^{-N\epsilon}
\leq e^{-N\epsilon}.
\end{split}
\end{equation}
Since this bound is summable it follows from the Borel-Cantelli lemma that 
\begin{equation}
\label{BCappl}
\limsup_{N\to\infty} \frac1N \log S_N^\o(g)
\leq \frac{1}{t} \log K(g) \qquad \o-a.s.
\end{equation}
Combine (\ref{fracmest}--\ref{SNdef}) and \eqref{BCappl} to obtain
\begin{equation}
\label{Sbdbd}
\begin{aligned}
S^\mathrm{que}(\beta,h;g) &\leq \log\cN(g) + \frac{1-t}{t}\log 2 + 
\frac{1}{t} \log \left(\sum_{m\in\N} \rho_g(m)^t\right)\\
&= \frac{1-t}{t}\log 2 + \frac{1}{t} \log \left(\sum_{m\in\N} e^{-gtm} \rho(m)^t\right).
\end{aligned}
\end{equation}
We see from (\ref{Sbdbd}) that $S^\mathrm{que}(\beta,h_c^\mathrm{ann}(\beta t);g)<\infty$ 
for $g>0$ and $t \in (0,1]$, and also for $g=0$ and $t \in (1/\alpha,1]$, i.e., $S^\mathrm{que}
(\beta,h;0)<\infty$ for $h \in (h_c^\mathrm{ann}(\beta/\alpha),h_c^\mathrm{ann}(\beta)]$.
This completes the proof because we already know that $S^\mathrm{que}(\beta,h;0) < 0$ for 
$h \in (h_c^\mathrm{ann}(\beta),\infty)$.
\end{proof}

Note that if $\sum_{m\in\N} \rho(m)^{1/\alpha}<\infty$, then $S^\mathrm{que}(\beta,
h_c^\mathrm{ann}(\beta/\alpha);0)<\infty$. This explains the remark made below 
(\ref{BGpos}). The above argument also shows that $S^\mathrm{que}(\beta,h;g)<\infty$ 
for all $\beta,h,g>0$, since for $\beta,g>0$ and any $h>h_c^\mathrm{ann}(\beta)
=M(2\beta)/2\beta$ the fifth equality in \eqref{SNfmest} becomes an inequality 
for any $t \in (0,1]$, while any $0<h\leq h_c^\mathrm{ann}(\beta)$ equals 
$h=h^\mathrm{ann}_c(\beta t)$ for some $t \in (0,1]$.


\subsection{Proof for $h<h_c^\mathrm{ann}(\beta/\alpha)$}
\label{S6.2}

\begin{proof}
For $L\in\N$, define (recall (\ref{nubetadefalt}))
\begin{equation}
\label{qLbetadef}
q^L_\beta(dx_1,\dots,x_m) = \delta_{mL}\,\nu_{\beta/\alpha}(dx_1)
\times\cdots\times\nu_{\beta/\alpha}(dx_m), 
\qquad m\in\N,\,x_1,\dots,x_m\in\R,
\end{equation}
and
\begin{equation}
\label{QLbetadef}
Q^L_\beta = (q^L_\beta)^{\otimes\N} \in \cP^{\mathrm{inv}}(\widetilde{E}^\N).
\end{equation}
We will show that
\begin{equation}
\label{hclbest}
h < h_c^\mathrm{ann}(\beta/\alpha) \quad \Longrightarrow \quad
\liminf_{L\to\infty} \frac{1}{L}\,\big[\Phi_{\beta,h}(Q^L_\beta) 
- I^{\mathrm{que}}(Q^L_\beta)\big] > 0, 
\end{equation}
which will imply the claim because $Q^L_\beta \in \cC^\mathrm{fin}$. (Recall (\ref{Csetdef}) 
and note that both $m_{Q^L_\beta}=L$ and $I^\mathrm{ann}(Q^L_\beta) = h(q^L_\beta \mid
q_{\rho,\nu}) = -\log\rho(L)+h(\nu_{\beta/\alpha} \mid \nu_\beta)$ are finite.)    

We have (recall (\ref{Phidef}) and (\ref{phidef}))
\begin{equation}
\begin{aligned}
\Phi_{\beta,h}(Q^L_\beta) &= \int_{\widetilde{E}} 
q^L_\beta(dy) \log \phi_{\beta,h}(y),\\
H(Q^L_\beta \mid q_{\rho,\nu}^{\otimes\N}) &= h(q^L_\beta \mid q_{\rho,\nu})
= \int_{\widetilde{E}} q^L_\beta(dy) \log
\Big(\frac{q^L_\beta(dy)}{q_{\rho,\nu}(dy)}\Big).
\end{aligned}
\end{equation}
Dropping the $1$ in front of the exponential in (\ref{phidef}), we obtain (similarly
as in (\ref{imp1}--\ref{imp3}))
\begin{equation}
\label{PhiHdif}
\begin{aligned}
&\Phi_{\beta,h}(Q^L_\beta) - H(Q^L_\beta \mid q_{\rho,\nu}^{\otimes\N})\\
&\qquad \geq \log(\tfrac12) + \int_{\widetilde{E}} q^L_\beta(dy)\,
\log\left[\frac{e^{-2\beta h\tau(y)-2\beta\sigma(y)}\,
q_{\rho,\nu}(dy)}{q^L_\beta(dy)}\right]\\
&\qquad = \log(\tfrac12) + \int_{\R^L} \nu_{\beta/\alpha}^{\otimes L}(dx_1,\dots,dx_L)
\,\log\left[e^{-2\beta h L}\,e^{-2\beta[x_1+\cdots+x_L]}\,
\frac{d\nu^{\otimes L}}{d\nu_{\beta/\alpha}^{\otimes L}}(x_1,\dots,x_L)\,\rho(L)\right]\\
&\qquad = \log(\tfrac12) + \int_{\R^L} \nu_{\beta/\alpha}^{\otimes L}(dx_1,\dots,dx_L)
\,\log\left[e^{[M(2\beta)-2\beta h]L}\,\frac{d\nu_\beta^{\otimes L}}
{d\nu_{\beta/\alpha}^{\otimes L}}(x_1,\dots,x_L)\,\rho(L)\right]\\
&\qquad = \log(\tfrac12) + [M(2\beta)-2\beta h]\,L
-h\big(\nu_{\beta/\alpha} \mid \nu_\beta\big)\,L
+\log\rho(L). 
\end{aligned}
\end{equation}
Furthermore, from (\ref{qLbetadef}) we have (recall (\ref{PsiQdef}))
\begin{equation}
\label{mQPsiQid}
m_{Q^L_\beta}=L, \qquad \Psi_{Q^L_\beta} = \nu_{\beta/\alpha}^{\otimes\N},
\end{equation}
which gives
\begin{equation}
\label{Idiffexpl}
(\alpha-1)\,m_{Q^L_\beta}\,H\big(\Psi_{Q^L_\beta} \mid \nu^{\otimes\N}\big)
= (\alpha-1)\,L\,h(\nu_{\beta/\alpha} \mid \nu).
\end{equation}

Combining (\ref{PhiHdif}--\ref{Idiffexpl}), recalling 
(\ref{eqgndefinitionIalg}--\ref{eqnratefctexplicitalg}) and using that
$\lim_{L\to\infty} L^{-1}\log \rho(L)=0$ by (\ref{rhocond}) when $L\to\infty$ 
along the support of $\rho$, we arrive at
\begin{equation}
\label{alldif}
\begin{aligned}
\liminf_{L\to\infty} \frac{1}{L}\,\big[\Phi_{\beta,h}(Q^L_\beta) 
- I^{\mathrm{que}}(Q^L_\beta)\big] 
&\geq [M(2\beta)-2\beta h]-h\big(\nu_{\beta/\alpha} \mid \nu_\beta\big)
-(\alpha-1)\,h(\nu_{\beta/\alpha} \mid \nu)\\
&= \alpha M(\tfrac{2\beta}{\alpha}) - 2\beta h
= 2\beta\,[h_c^\mathrm{ann}(\beta/\alpha)-h],
\end{aligned}
\end{equation}
where the first equality uses the relation (recall \eqref{hcannid} and \eqref{nubetadefalt})
\begin{equation}
\begin{aligned}
&h\big(\nu_{\beta/\alpha} \mid \nu_\beta\big)
+(\alpha-1)\,h(\nu_{\beta/\alpha} \mid \nu)\\
&\qquad = \int_{l\in\R} \nu_{\beta/\alpha}(dl)\,
\left(\left[-\tfrac{2\beta}{\alpha}\,l-M\big(\tfrac{2\beta}{\alpha}\big)\right]
+\left[2\beta\,l+M(2\beta)\right]
+(\alpha-1)\left[-\tfrac{2\beta}{\alpha}l-M\big(\tfrac{2\beta}{\alpha}\big)\right]\right)\\ 
&\qquad = M(2\beta)-\alpha M\big(\tfrac{2\beta}{\alpha}\big).
\end{aligned}
\end{equation} 
Note that (\ref{alldif}) proves (\ref{hclbest}).
\end{proof}


\subsection{Proof for $h=h_c^\mathrm{ann}(\beta/\alpha)$}
\label{S6.3}

\begin{proof}
Our starting point is (\ref{Sdefalt}), where (recall Theorem~\ref{qLDP})
\begin{equation}
\label{arg1}
I^\mathrm{que}(Q) = I^\mathrm{fin}(Q) 
= H(Q \mid q_{\rho,\nu}^{\otimes\N}) + (\alpha-1)\,m_Q\,H(\Psi_Q \mid \nu^{\otimes\N}),
\qquad Q \in \cC^\mathrm{fin}.
\end{equation}
The proof comes in 4 steps.

\medskip\noindent
{\bf 1.} As shown in Birkner, Greven and den Hollander~\cite{BiGrdHo10}, Equation (1.32),
\begin{equation}
\label{arg2}
H(Q \mid q_{\rho,\nu}^{\otimes\N}) = m_Q\,H(\Psi_Q \mid \nu^{\otimes\N}) + R(Q),
\end{equation}  
where $R(Q) \geq 0$ is the ``specific relative entropy w.r.t.\ $\rho^{\otimes\N}$ of 
the word length process under $Q$ conditional on the concatenation''.  Combining 
(\ref{arg1}--\ref{arg2}), we have $I^\mathrm{que}(Q) \leq \alpha\,H(Q \mid 
q_{\rho,\nu}^{\otimes\N})$, which yields
\begin{equation}
\label{arg3}
S_*^\mathrm{que}(\beta,h) \geq \sup_{Q\in\cC^\mathrm{fin}} 
\big[\Phi_{\beta,h}(Q) - \alpha\,H(Q \mid q_{\rho,\nu}^{\otimes\N})\big].
\end{equation}

\medskip\noindent
{\bf 2.} The variational formula in the right-hand side of (\ref{arg3}) can be computed 
similarly as in part (I) of Section~\ref{S3.2}. Indeed,
\begin{equation}
\label{arg4}
\mbox{r.h.s. } (\ref{arg3}) = \sup_{ {q\in \cP(\widetilde{E})} \atop 
{m_q<\infty,\,h(q\mid q_{\rho,\nu})<\infty} }
\left[\int_{\widetilde{E}} q(dy) \log \phi_{\beta,h}(y) 
- \alpha\,h(q \mid q_{\rho,\nu})\right].
\end{equation}
Define
\begin{equation}
\label{arg5}
q_{\beta,h}(dy) = \frac{1}{\cN(\beta,h)}\,[\phi_{\beta,h}(y)]^{1/\alpha}
\,q_{\rho,\nu}(dy),
\end{equation}
where $\cN(\beta,h)$ is the normalizing constant. Then the term between square brackets
in the right-hand side of (\ref{arg4}) equals $\alpha \log \cN(\beta,h)-\alpha h(q\mid
q_{\beta,h})$, and hence
\begin{equation}
\label{arg6}
S_*^\mathrm{que}(\beta,h) \geq \alpha \log \cN(\beta,h),
\end{equation}
provided $\cN(\beta,h)<\infty$ so that $q_{\beta,h}$ is well-defined.

\medskip\noindent
{\bf 3.} 
Abbreviate $\mu=2\beta/\alpha$. Since $h_c^\mathrm{ann}(\beta/\alpha) = M(\mu)/\mu$, 
we have
\begin{equation}
\label{arg7}
\cN\big(\beta,h_c^\mathrm{ann}(\beta/\alpha)\big)
= \sum_{m\in\N} \rho(m) \int_{\R^m} \nu(dx_1)\times\cdots\times\nu(x_m)\,
\left\{\tfrac12\left(1+e^{-\alpha(M(\mu)m+\mu[x_1+\cdots+x_m])}\right)\right\}^{1/\alpha}. 
\end{equation}
Let $Z$ be the random variable on $(0,\infty)$ whose law $P$ is equal to the law of 
$e^{-(M(\mu)m+\mu[x_1+\cdots+x_m])}$ under $\rho(m)\,\nu(dx_1)\times\cdots\times\nu(x_m)$. 
Let 
\begin{equation}
\label{falphadef}
f_\alpha(z) = \{\tfrac12(1+z^\alpha)\}^{1/\alpha}, \qquad z>0.
\end{equation} 
Then
\begin{equation}
\label{arg8}
\mbox{r.h.s. } (\ref{arg7}) = E(f_\alpha(Z)).
\end{equation}
We have $E(Z)=1$. Moreover, an easy computation gives
\begin{equation}
\label{arg9}
\begin{aligned}
f_\alpha'(z) &= (\tfrac12)^{1/\alpha}\,(1+z^\alpha)^{(1/\alpha)-1}\,z^{\alpha-1},\\
f_\alpha''(z) &= (\tfrac12)^{1/\alpha}\,(1+z^\alpha)^{(1/\alpha)-2}\,z^{\alpha-2}\,
(\alpha-1),
\end{aligned}
\end{equation}
so that $f_\alpha$ is strictly convex. Therefore, by Jensen's inquality and the fact that $P$
is not a point mass, we have
\begin{equation}
\label{arg10}
E(f_\alpha(Z)) > f_\alpha(E(Z)) = f_\alpha(1) = 1.
\end{equation}
Combining (\ref{arg6}--\ref{arg8}) and (\ref{arg10}), we arrive at
\begin{equation}
\label{arg11}
S_*^\mathrm{que}\big(\beta,h_c^\mathrm{ann}(\beta/\alpha)\big) > 0,
\end{equation}
which proves the claim. 

\medskip\noindent
{\bf 4.} 
It remains to check that $\cN(\beta,h_c^\mathrm{ann}(\beta/\alpha))<\infty$. But 
$f_\alpha(z) \leq (\tfrac12)^{1/\alpha}(1+z)$, $z>0$, and so we have 
\begin{equation}
\label{arg12}
\cN\big(\beta,h_c^\mathrm{ann}(\beta/\alpha)\big) 
\leq (\tfrac12)^{1/\alpha}(1+E(Z))
\leq 2^{1-(1/\alpha)}<\infty.
\end{equation}
\end{proof}


\section{Proof of Corollary~\ref{Kcslope}}
\label{S7}
  
\begin{proof}
The proof comes in 6 Steps. In Steps 1--3 we give the proof for the case where the 
disorder $\omega$ is standard Gaussian and the excursion length distribution $\rho$ 
satisfies $\rho(k) \sim Ak^{-\alpha}$ as $k\to\infty$ for some $0<A<\infty$ and 
$1<\alpha<2$. In Steps 4--6 we explain how to extend the proof to arbitrary $\omega$ 
and $\rho$ satisfying \eqref{mgffin} and \eqref{rhocond}. 

\medskip\noindent
{\bf 1.} 
Our starting point is \eqref{arg6} with
\begin{equation}
\label{slope1}
\begin{split}
\cN(\beta,h)&= \sum_{m\in\N} \rho(m) \int_{\R^m} \nu(dx_1)\times\cdots\times\nu(dx_m)
\left\{\tfrac12\left(1+e^{-2\beta h m - 2\beta(x_1+\cdots+x_m)}
\right)\right\}^{1/\alpha}\cr
&= \sum_{m\in\N} \rho(m) \int_{l\in\R} \nu^{\circledast m}(dl)
\left\{\tfrac12\left(1+e^{-2\beta h m - 2\beta l}\right)\right\}^{1/\alpha},
\end{split}
\end{equation}
where $\nu^{\circledast m}$ is a $m$-fold convolution of $\nu$.
Pick $h=B\beta/\alpha$ with $B \geq 1$, introduce the variables
\begin{equation}
\label{slope2}
x = l/\sqrt{m}, \qquad y = (\beta/\alpha)^2 m,
\end{equation}
and write out
\begin{equation}
\label{slope3}
\begin{aligned}
\cN(\beta,B\beta/\alpha)
&= \sum_{m\in\N} \rho(m) \int_{l\in\R} N(0,m)(dl)
\left\{\tfrac12\left(1+e^{-2B(\beta^2/\alpha)m - 2\beta l}\right)\right\}^{1/\alpha}\\
&= \sum_{y\in (\beta/\alpha)^2\N} \rho(y(\alpha/\beta)^2)
\int_{x\in\R} N(0,1)(dx)\left\{\tfrac12
\left(1+e^{-\alpha[2By+2\sqrt{y}x]}\right)\right\}^{1/\alpha}\\
&= \sum_{y\in (\beta/\alpha)^2\N} \rho(y(\alpha/\beta)^2)\,
E_{y,B}(f_\alpha(Z)),
\end{aligned}
\end{equation}
where $N(0,k)$ is the Gaussian distribution with mean $0$ and variance $k$, $f_\alpha$
is the function defined in \eqref{falphadef}, and $Z$ is the random variable
\begin{equation}
\label{slope4}
Z= e^{-2By-2\sqrt{y}X} \mbox{ with } X \mbox{ standard Gaussian},
\end{equation}
whose law we denote by $P_{y,B}$. Substract $1$ to obtain
\begin{equation}
\label{slope5}
\cN(\beta,B\beta/\alpha)-1 = \sum_{y\in (\beta/\alpha)^2\N} \rho(y(\alpha/\beta)^2)
\left[E_{y,B}(f_\alpha(Z))-1\right].
\end{equation}

\medskip\noindent
{\bf 2.}
Suppose that $\rho(m) \sim Am^{-\alpha}$ as $m\to\infty$ for some $0<A<\infty$ and 
$1<\alpha<2$. Then, letting $\beta \downarrow 0$ in \eqref{slope5}, we obtain
\begin{equation}
\label{slope6}
\begin{aligned}
\lim_{\beta \downarrow 0} \frac{1}{\beta^{2(\alpha-1)}}[\cN(\beta,B\beta/\alpha)-1]
= \frac{A}{\alpha^{2(\alpha-1)}} \int_0^\infty dy\,y^{-\alpha}
[E_{y,B}(f_\alpha(Z))-1].
\end{aligned}
\end{equation} 
Here, we note that the integral converges near $y=0$ because $\alpha<2$ and 
$E_{y,B}(f_\alpha(Z))-1 = O(y)$ as $y \downarrow 0$, and also converges near 
$y=\infty$ because $\alpha>1$ and $E_{y,B}(f_\alpha(Z)) \leq  2^{1/\alpha}
(1+E_{y,B}(Z))\leq 2^{1/\alpha}(1+E_{y,1}(Z))= 2^{1-(1/\alpha)}<\infty$. Next, 
abbreviate
\begin{equation}
\label{slope7}
I_\alpha(B) = \int_0^\infty dy\,y^{-\alpha} [E_{y,B}(f_\alpha(Z))-1].
\end{equation}
If $B=1$, then the integrand is strictly positive, because $z \mapsto f_\alpha(z)$
is strictly convex and $E_{y,1}(Z)=1$ for all $y$, so that $E_{y,1}(f_\alpha(Z))
> f_\alpha(E_{y,1}(Z)) = f_\alpha(1)=1$ by Jensen's inequality. Thus, we have
$I_\alpha(1)>0$. However, $B \mapsto I_\alpha(B)$ is strictly decreasing and 
continuous on $[1,\infty)$, because $z \mapsto f_\alpha(z)$ is strictly increasing 
and continuous on $[0,\infty)$. Hence there exists a $B(\alpha)>1$ such that 
$I_\alpha(B(\alpha))=0$. 

\medskip\noindent
{\bf 3.}
The estimate in Step 2 implies that $\cN(\beta,B\beta/\alpha)>1$ for any $B \in 
(1,B(\alpha))$ and $\beta$ small enough. Since $h\mapsto S_*^\mathrm{que}(\beta,h)$
is non-increasing and $S_*^\mathrm{que}(\beta,h_c^\mathrm{que}(\beta))=0$, it 
therefore follows from \eqref{arg6} that $h_c^\mathrm{que}(\beta) \geq B\beta/\alpha$ 
for any $B \in (1,B(\alpha))$ and $\beta$ small enough, which yields 
\begin{equation}
\label{slope8}
\liminf_{\beta \downarrow 0}  h_c^\mathrm{que}(\beta)/\beta \geq B(\alpha)/\alpha
\end{equation}
andh proves the first half of the lower bound in Corollary~\ref{Kcslope}.

\medskip\noindent
{\bf 4.}
If the disorder is not standard Gaussian, then the \emph{same scaling} as in \eqref{slope6} 
holds because the disorder satisfies the central limit theorem (recall that we have 
assumed that the disorder has zero mean and unit variance). The finiteness of the moment 
generating function assumed in \eqref{mgffin} suffices to justify this claim. If the 
excursion length distribution is modulated by a slowly varying function $L$, as in
\eqref{rhocondalt}, then we can use the fact that $L(y(\alpha/\beta)^2) \sim 
L(1/\beta^2)$ as $\beta \downarrow 0$ uniformly in $y$ on compact subsets of 
$(0,\infty)$ (Bingham, Golide and Teugels~\cite{BiGoTe87}, Theorem 1.2.1, and all 
that changes is that the left-hand side of \eqref{slope6} must be divided by an 
extra factor $L(1/\beta^2)$. We need that $\rho$ is asymptotically periodic in order 
to get the integral over $y$ w.r.t.\ the Lebesgue measure $dy$ (modulo a factor
1 over the period, which comes in front and therefore is irrelevant).

\medskip\noindent
{\bf 5.} We next turn to the case $\alpha \geq 2$. For $y \downarrow 0$,
\begin{equation}
\label{slope9}
e^{-2By-2\sqrt{y}X} - 1 = \sqrt{y}\,(-2X) + y\,(-2B+2X^2) + O(y^{3/2}), 
\end{equation}
while for $z \to 1$,
\begin{equation}
\label{slope10}
f_\alpha(z) = 1 + \tfrac12(z-1) + \tfrac18(\alpha-1)(z-1)^2 + O((z-1)^3).
\end{equation}
Combining these expansions with the observation that $X$ has zero mean and unit 
variance, we find that for $y \downarrow 0$,
\begin{equation}
\label{slope11}
E_{y,B}(f_\alpha(Z)) = 1 + y\,[\tfrac12(1+\alpha)-B] + O(y^{3/2}).
\end{equation}
Since $E_{y,B}(f_\alpha(Z))$ is bounded from above, it follows from \eqref{slope11}
that if $B<\tfrac12(1+\alpha)$ and
\begin{equation}
\label{slope12}
\lim_{\beta \downarrow 0} 
\frac{
\sum_{y \in (\beta/\alpha)^2\N,\,y>\epsilon} \rho(y(\alpha/\beta)^2)}
{\sum_{y \in (\beta/\alpha)^2\N,\,y\leq\epsilon} y\,\rho(y(\alpha/\beta)^2)
} = 0 \qquad \forall\,\epsilon>0,
\end{equation}
then the behavior of the sum in \eqref{slope5} for $\beta \downarrow 0$ is dominated by 
the small values of $y$, i.e.,
\begin{equation}
\label{slope13}
\begin{aligned}
\cN(\beta,B\beta/\alpha)-1 
&\sim [\tfrac12(1+\alpha)-B] \sum_{y\in (\beta/\alpha)^2\N,\,y\leq\epsilon} 
y\,\rho(y(\alpha/\beta)^2),\\
&= [\tfrac12(1+\alpha)-B]\,(\beta/\alpha)^2 \sum_{m\in\N,\,m\leq\epsilon(\alpha/\beta)^2} 
m\rho(m) \qquad \forall\,\epsilon>0.
\end{aligned}
\end{equation}
The condition in \eqref{slope13} is equivalent to
\begin{equation}
\label{slope14}
\lim_{M\to\infty} \frac{M\sum_{m>M} \rho(m)}
{\sum_{1\leq m\leq M} m\rho(m)} = 0.
\end{equation}
Clearly, the condition in \eqref{slope14} is satisfied when $m_\rho=\sum_{m\in\N} 
m\rho(m)<\infty$ (because the numerator tends to zero and the denominator tends to 
$m_\rho$), in which case \eqref{slope13} yields
\begin{equation}
\label{slope15}
\lim_{\beta \downarrow 0} \frac{1}{\beta^2}\,[\cN(\beta,B\beta/\alpha)-1]
=  [\tfrac12(1+\alpha)-B]\,\frac{1}{\alpha^2}\,m_\rho.
\end{equation}
As in Step 3, it therefore follows that $h_c^\mathrm{que}(\beta) \geq B\beta/\alpha$ 
for any $B<\tfrac12(1+\alpha)$ and $\beta$ small enough, which yields 
\begin{equation}
\label{slope16}
\liminf_{\beta \downarrow 0}  h_c^\mathrm{que}(\beta)/\beta \geq 
\tfrac{1+\alpha}{2\alpha}
\end{equation}
and proves the second half of the lower bound in Corollary~\ref{Kcslope} when 
$m_\rho<\infty$. It remains to check \eqref{slope14} when $m_\rho=\infty$ and
$\rho$ is regularly varying at infinity. Since $\alpha\geq 2$, this corresponds 
to the case where $\rho(m)=m^{-2}L(m)$ along the (asymptotically periodic) support 
of $\rho$ with $L$ slowly varying at infinity and not decaying too fast. Now, 
by \cite{BiGoTe87}, Theorem 1.5.10, we have $\sum_{m>M} \rho(m) = \sum_{m>M} 
m^{-2}L(m) \sim M^{-1}L(M)$, and so the numerator of \eqref{slope14} is $\sim L(M)$. 
On the other hand, for every $0<\delta<1$ the denominator is bounded from below 
by $\sum_{\delta M<m\leq M} m\rho(m) = \sum_{\delta M<m\leq M} m^{-1}L(m) \sim 
L(M) \sum_{\delta M<m\leq M} m^{-1} \sim L(M)\log(1/\delta)$, and so \eqref{slope14} 
follows by first letting $M\to\infty$ and then letting $\delta\downarrow 0$.  

\medskip\noindent
{\bf 6.} As in Step 4, the argument in Step 5 extends to arbitrary disorder subject 
to \eqref{mgffin}.
\end{proof}


\section{Proof of Corollaries~\ref{delocpathprop} and \ref{locpathprop}}
\label{S8}

Corollaries~\ref{delocpathprop} and \ref{locpathprop} are proved in Sections~\ref{S8.1}
and \ref{S8.2}, respectively.


\subsection{Proof of Corollary~\ref{delocpathprop}}
\label{S8.1}
 
\begin{proof}
Fix $(\beta,h)\in\mathrm{int}(\cD^\mathrm{que})$. We know that $S^\mathrm{que}(\beta,h;0)<0$ 
(recall Fig.~\ref{fig-varhc}) and $\sum_{n\in\N} \wZ^{\beta,h,\o}_n<\infty$. It follows from 
\eqref{Slim} and \eqref{Sfinalext} that for every $\epsilon>0$ and $\omega$-a.s.\ there exists 
an $N_0=N_0(\o,\epsilon)<\infty$ such that 
\begin{equation}
\label{FNbd}
F_N^{\beta,h,\o}(0) \leq e^{N[S^\mathrm{que}(\beta,h;0)+\epsilon]},
\qquad N\geq N_0.
\end{equation}
For $E$ an arbitrary event, write $\wZ^{\beta,h,\o}_n(E)$ to denote the constrained partition 
restricted to $E$. Estimate, for $M\in\N$ and $\epsilon$ small enough such that $S^\mathrm{que}
(\beta,h;0)+\epsilon<0$,
\begin{equation}
\begin{split}
\wP^{\beta,h,\omega}_n(\cM_n \geq M) 
&=\frac{\wZ^{\beta,h,\o}_n(\cM_n \geq M)}{\wZ^{\beta,h,\o}_n}
\leq \frac{\sum_{n\in\N} \wZ^{\beta,h,\o}_n(\cM_n \geq M)}{\wZ^{\beta,h,\o}_n}\\
&=\frac{1}{\wZ^{\beta,h,\o}_n} \sum_{N \geq M} F_N^{\beta,h,\o}(0)
\leq\frac{2}{\rho(n)} \frac{e^{M[S^{\rm que}(\beta,h;0)+\epsilon]}}
{1-e^{[S^{\rm que}(\beta,h;0)+\epsilon]}},
\end{split}
\end{equation}
where the second equality follows from (\ref{Zniddef}--\ref{FNdef}). The second inequality 
follows from (\ref{FNbd}) and the bound $\wZ^{\beta,h,\o}_n\geq\tfrac12\rho(n)$, the latter 
being immediate from \eqref{queexpartsum} and the fact that every excursion has probability 
$\tfrac12$ of lying below the interface. Since $\rho(n)=n^{-\alpha+o(1)}$, we get the claim 
by choosing $M=\lceil c\log n\rceil$ with $c$ such that $\alpha+c[S^{\rm que}(\beta,h;0)
+\epsilon]<0$, and letting $n\to\infty$ followed by $\epsilon\downarrow 0$.
\end{proof}
 

\subsection{Proof of Corollary~\ref{locpathprop}}
\label{S8.2}

\begin{proof}
Fix $(\beta,h)\in\cL^\mathrm{que}$. We know that $g^\mathrm{que}(\beta,h)>0$ and 
$S^\mathrm{que}(\beta,h;g^\mathrm{que}(\beta,h))=0$. It follows from \eqref{Slim} 
and \eqref{Sfinalext} that for every $\epsilon,\delta>0$ and $\o$-a.s.\ there exist 
$n_0=n_0(\o,\epsilon)<\infty$ and $M_0=M_0(\o,\delta)<\infty$ such that 
\begin{equation}
\label{ZFasymp}
\begin{split}
\wZ^{\beta,h,\o}_n 
&\geq e^{n[g^\mathrm{que}(\beta,h)-\epsilon]},
\qquad n\geq n_0,\\ 
F^{\beta,h,\o}_M(g^\mathrm{que}(\beta,h)+\delta) 
&\leq e^{M[S^\mathrm{que}(\beta,h;g^\mathrm{que}(\beta,h)+\delta)+\delta^2]}, 
\qquad M\geq M_0,\\
F^{\beta,h,\o}_M(g^\mathrm{que}(\beta,h)-\delta) 
&\leq e^{M[S^\mathrm{que}
(\beta,h;g^\mathrm{que}(\beta,h)-\delta)+\delta^2]}, 
\qquad M\geq M_0.
\end{split}
\end{equation}
For every $M_1,M_2\in\N$ with $M_1<M_2$ we have
\begin{equation}
\label{breakup}
\begin{split}
\wP^{\beta,h,\o}_n(M_1<\cM_n<M_2)
&= 1-\left[\wP^{\beta,h,\o}_n(\cM_n\geq M_2)+\wP^{\beta,h,\o}_n(\cM_n\leq M_1)\right].
\end{split}
\end{equation}
Below we show that the probabilities in the right-hand side of (\ref{breakup}) vanish 
as $n\to\infty$ when $M_1=\lceil c_1n\rceil $ with $c_1<C_-$ and $M_2=\lceil c_2n\rceil$ 
with $c_2>C_+$, respectively, where
\begin{equation}
\label{C+-defs}
\begin{aligned}
-\frac{1}{C_-} &= \left(\frac{\partial}{\partial g}\right)^-\,
S^\mathrm{que}\big(\beta,h;g^\mathrm{que}(\beta,h)\big),\\
-\frac{1}{C_+} &= \left(\frac{\partial}{\partial g}\right)^+\,
S^\mathrm{que}\big(\beta,h;g^\mathrm{que}(\beta,h)\big),
\end{aligned}
\end{equation}
are the left-derivative and right-derivative of $g\mapsto S^\mathrm{que}(\beta,h;g)$ at 
$g=g^\mathrm{que}(\beta,h)$, which exist by convexity, are strictly negative (recall 
Fig.~\ref{fig-varfe}) and satisfy $C_-\leq C_+$. Throughout the proof we assume that 
$M_1\geq M_0$. 

\medskip\noindent
{\bf 1.} Put $M_2=\lceil c_2 n\rceil$, and abbreviate
\begin{equation}
\label{abbr}
a(\beta,h,\delta) = S^\mathrm{que}(\beta,h;g^\mathrm{que}(\beta,h)+\delta)+\delta^2,
\end{equation}
where we choose $\delta$ small enough such that $a(\beta,h,\delta)<0$ (recall Fig.~\ref{fig-varhc}). 
Estimate
\begin{equation}
\label{lbest}
\begin{split}
\wP^{\beta,h,\o}_n(\cM_n\geq M_2)
&= \frac{\wZ^{\beta,h,\o}_n\left(\cM_n \geq M_2\right)}{\wZ^{\beta,h,\o}_n}\\
&\leq e^{n[\epsilon+\delta]}\,\wZ^{\beta,h,\o}_n\left(\cM_n \geq M_2\right)
\,e^{-n[g^\mathrm{que}(\beta,h)+\delta]}\\
&\leq e^{n[\epsilon+\delta]} \sum_{n'\in\N}\wZ^{\beta,h,\o}_{n'}(\cM_{n'}\geq M_2)
\,e^{-n'[g^\mathrm{que}(\beta,h)+\delta]}\\
&= e^{n[\epsilon+\delta]} \sum_{N\geq M_2} F_N^{\beta,h,\o}(g^\mathrm{que}(\beta,h)+\delta)\\
&\leq e^{n[\epsilon+\delta]} \sum_{N\geq M_2} 
e^{Na(\beta,h,\delta)}\\
&= \frac{e^{n[\epsilon+\delta+c_2a(\beta,h,\delta)]}}
{1-e^{a(\beta,h,\delta)}}.
\end{split}
\end{equation}
The first inequality follows from the first line in \eqref{ZFasymp}, the second equality 
from (\ref{Zniddef}--\ref{FNdef}), and the third inequality from \eqref{ZFasymp}. The 
claim follows by picking $c_2$ such that 
\begin{equation}
\epsilon+\delta+c_2a(\beta,h,\delta)<0,
\end{equation}
letting $n\to\infty$ followed by $\epsilon\downarrow 0$ and $\delta\downarrow 0$,
and using that 
\begin{equation}
\lim_{\delta \downarrow 0} \frac{1}{\delta}\,a(\beta,h,\delta) 
= \left(\frac{\partial}{\partial g}\right)^+ 
S^\mathrm{que}(\beta,h;g^\mathrm{que}(\beta,h)) = -\frac{1}{C_+} < -\frac{1}{c_2}.
\end{equation}

\medskip\noindent
{\bf 2.} Put $M_1=\lceil c_1 n\rceil$ and abbreviate
\begin{equation}
b(\beta,h,\delta)
=S^\mathrm{que}(\beta,h;g^\mathrm{que}(\beta,h)-\delta)+\delta^2,
\end{equation}
where we choose $\delta$ small enough such that $b(\beta,h,\delta)>0$. Split
\begin{equation}
\wP^{\beta,h,\o}_n(\cM_n\leq M_1) = I + II
\end{equation}
with
\begin{equation}
I=\frac{\wZ^{\beta,h,\o}_n\left(\cM_n < M_0\right)}{\wZ^{\beta,h,\o}_n},
\qquad
II=  \frac{\wZ^{\beta,h,\o}_n\left(M_0 \leq \cM_n \leq M_1\right)}{\wZ^{\beta,h,\o}_n}.
\end{equation}
Since
\begin{equation}
I \leq e^{-n[g^{\rm que}(\beta,h)-\epsilon]}\, \wZ^{\beta,h,\o}_n\left(\cM_n<M_0\right)
= e^{-n[g^{\rm que}(\beta,h)-\epsilon]}\,\sum_{N<M_0} 
F_N^{\beta,h,\o}(g^\mathrm{que}(\beta,h)-\delta),
\end{equation}
this term is harmless as $n\to\infty$ (recal \eqref{Fub}). Repeat the arguments leading 
to \eqref{lbest}, to estimate
\begin{equation}
\begin{split}
II &\leq
e^{n[\epsilon-\delta]} \sum_{n'\in\N} \wZ^{\beta,h,\o}_{n'}\left(M_0\leq \cM_{n'} \leq M_1\right)
\,e^{-n'[g^\mathrm{que}(\beta,h)-\delta]}\\
&= e^{n[\epsilon-\delta]} \sum_{M_0 \leq N \leq M_1} 
F_N^{\beta,h,\o}(g^\mathrm{que}(\beta,h)-\delta)\\
&\leq e^{n[\epsilon-\delta]} \sum_{M_0 \leq N \leq M_1} e^{Nb(\beta,h,\delta)}\\
&\leq e^{n[\epsilon-\delta+c_1b(\beta,h,\delta)]} 
\sum_{N\leq M_1} e^{[N-M_1]b(\beta,h,\delta)}\\
&\leq \frac{e^{n\left[\epsilon-\delta+c_1b(\beta,h,\delta)\right]}}
{1-e^{-b(\beta,h,\delta)}}.
\end{split}
\end{equation}
Therefore the assertion follows by choosing $c_1$ such that 
\begin{equation}
\epsilon-\delta+c_1b(\beta,h,\delta)<0,
\end{equation}
letting $n\to\infty$ followed by $\epsilon\downarrow 0$ and $\delta\downarrow 0$, and using that
\begin{equation}
\lim_{\delta \downarrow 0} \frac{1}{\delta}\,b(\beta,h,\delta) 
= -\left(\frac{\partial}{\partial g}\right)^- S^\mathrm{que}(\beta,h;g^\mathrm{que}(\beta,h))
= \frac{1}{C_-} < \frac{1}{c_1}.
\end{equation}

\medskip
Recalling \eqref{breakup}, we have now proved that
\begin{equation}
\lim_{n\to\infty} \wP^{\beta,h,\o}_n(\lceil c_1n\rceil <\cM_n<\lceil c_2n\rceil) = 1
\qquad \forall\,c_1<C_-,\,c_2>C_+. 
\end{equation}
Finally, if \eqref{densid} holds, then $C_-=C_+$, and we get the law of large numbers in
\eqref{LLNpath}.  
\end{proof}


\appendix


\section{Control of $\Phi_{\beta,h}$}
\label{appA}

In  Appendix~\ref{appA.1} we prove the bound in \eqref{Fub} (Lemma~\ref{mainlemma} below).
In Appendix~\ref{appA.2}  we prove that $h(\pi_1Q|q_{\rho,\nu})<\infty$ implies that 
$\Phi_{\beta,h}(Q)<\infty$ for all $\beta,h>0$ (Lemma~\ref{finitephi} below).  In both
proofs we make use of a concentration of measure estimate for the disorder $\omega$ whose 
proof is given in Appendix~\ref{appD}. 


\subsection{Proof of $\omega$-a.s.\ boundedness of 
$-gm_{R_N^\omega}+\Phi_{\beta,h}(R_N^\omega)$ for $g>0$}
\label{appA.1}

Recall the definition of $R_N^\omega$ in (\ref{empprocomega}).

\begin{lemma}
\label{mainlemma}
Fix $\beta,h,g>0$. Then $\omega$-a.s.\ there exists a $K(\omega,g)<\infty$ such 
that, for all $N\in\N$ and for all sequences $0=k_0 < k_1 < \cdots < k_N < \infty$,
\begin{equation}
\label{sumbd}
-g k_N + \sum_{i=1}^N \log\psi^\omega_{\beta,h}\big((k_{i-1},k_i]\big)
\leq K(\omega,g)N.
\end{equation}
\end{lemma}

\begin{proof}
The proof comes in 3 steps. Throughout the proof $\beta,h,g>0$ are fixed.

\medskip\noindent
{\bf 1.} For $l\in\N$ and $m\in\Z$, define
\begin{equation}
\begin{aligned}
J^\omega(l,m) &= \left\{I\subset\N\colon\,|I| = l,
\,m \leq -2\beta\sum_{k \in I}(\omega_k+h) < m+1\right\},\\
J^\omega(m) &= \bigcup_{l\in\N} J^\omega(l,m),
\end{aligned}
\end{equation}
and
\begin{equation}
\label{Tdef}
\begin{aligned}
&T^\omega_0(l,m)=0,\\ 
&T^\omega_j(l,m) 
= \inf\big\{n \geq T^\omega_{j-1}(l,m)+l\,1_{\{j>1\}}\colon\,
(n,n+l]\cap\N \in J^\omega(l,m)\big\},
\quad j \in \N,
\end{aligned}
\end{equation}
and
\begin{equation}
\label{Adef}
A(m) = \left\{\omega\colon\,T^\omega_j(l,m) \leq jm^4 \mbox{ for some } 
j,l\in\N\right\}.
\end{equation}
Below we will show that
\begin{equation}
\label{sumfin}
\sum_{m\in\N_0} \P(A(m)) < \infty.
\end{equation}
By Borel-Cantelli, this implies that $\omega$-a.s.\ there exists an $M(\omega)<\infty$
such that $\omega \notin A(m)$ for all $m > M(\omega)$. 

\medskip\noindent
{\bf 2.} Abbreviate $I_i=(k_i,k_{i+1}]$ and split
\begin{equation}
\label{ineq1} 
\sum_{i=1}^N \log \psi^\omega_{\beta,h}({I_i}) \leq I + II
\end{equation}
with
\begin{equation}
I = \sum_{i=1}^N 1_{B_i}(\omega)\,\log \psi^\omega_{\beta,h}(I_i),
\qquad 
II = \sum_{i=1}^N 1_{C_i}(\omega)\,\log \psi^\omega_{\beta,h}(I_i),
\end{equation}
where 
\begin{equation}
\begin{aligned}
B_i &=\{\omega\colon\,I_i\in J^\omega(m) \mbox{ for some } 0 \leq m \leq M(\omega)\},\\
C_i &=\{\omega\colon\,I_i\in J^\omega(m) \mbox{ for some } m > M(\omega)\}.
\end{aligned}
\end{equation}
The inequality in (\ref{ineq1}) comes from dropping the contribution of the $I_i$'s 
in $\cup_{m\in\Z\backslash\N_0} J(m)$ (for which $\log\psi^\omega_{\beta,h}(I_i) 
\leq 0 \wedge [-2\beta\sum_{k \in I_i}(\omega_k+h)] = 0$) and retaining only the 
$I_i$'s in $\cup_{m\in\N_0} J^\omega(m)$. Clearly, we have
\begin{equation}
\label{ineq2}
I \leq N[M(\omega)+1].
\end{equation}
To bound $II$, define
\begin{equation}
\label{Lambdadef}
I^\omega_N(m) = \{1 \leq i \leq N\colon\,I_i \in J^\omega(m)\}.
\end{equation}
Then
\begin{equation}
\label{IIest}
II = \sum_{m>M(\omega)} \sum_{i\in I^\omega_N(m)} \log \psi^\omega_{\beta,h}(I_i)
\leq \sum_{m>M(\omega)} |I^\omega_N(m)|\,(m+1). 
\end{equation}
It follows from (\ref{Tdef}--\ref{Adef}) and (\ref{Lambdadef}) that if $I^\omega_N(m)
\neq\emptyset$ and $m>M(\omega)$, then $k_N\geq|I^\omega_N(m)|\,m^4$. Hence
\begin{equation}
\label{ineq3}
\begin{aligned}
k_N 
&\geq \max_{m>M(\omega)} |I^\omega_N(m)|\,m^4\\
&\geq \sum_{m>M(\omega)}|I^\omega_N(m)|\,m^4\,
\left(\frac{m^{-2}}{\sum_{\bar{m}>M(\omega)} \bar{m}^{-2}}\right)\\
&\geq \sum_{m>M(\omega)}  |I^\omega_N(m)|\,Cm^2
\end{aligned}
\end{equation}
with $C=1/\sum_{\bar{m}\in\N} \bar{m}^{-2}>0$. Combining (\ref{IIest}--\ref{ineq3}), we get
\begin{equation} 
\label{ineq4}
-gk_N + II \leq \sum_{m>M(\omega)} |I^\omega_N(m)|\,[-gCm^2+(m+1)]. 
\end{equation}
Since $g>0$, we have $\max_{m\in\N} [-gCm^2+ (m+1)] = C(g) \leq 1+(1/4gC) < \infty$. 
Since $\sum_{m\in\Z} |I^\omega_N(m)|$ $= N$, we can combine (\ref{ineq1}), (\ref{ineq2}) 
and (\ref{ineq4}) to get the claim in (\ref{sumbd}) with $K(\omega,g)=M(\omega)+1+C(g)$.

\medskip\noindent
{\bf 3.} It remains to prove (\ref{sumfin}). Estimate
\begin{equation}
\label{Asum}
\P(A(m)) \leq \sum_{l\in\N} \sum_{j\in\N} 
\P\big(T^\omega_j(l,m) \leq jm^4\big)
\end{equation}
and
\begin{equation}
\label{Test1}
\begin{aligned}
\P\big(T^\omega_j(l,m) \leq jm^4\big)
&\leq \binom{jm^4}{j} \left[\P\big((0,l] \in J^\omega(l,m)\big)\right]^j\\
&\leq \left[em^4\,\P\left(\sum_{k=1}^{l} \omega_k 
\leq -\left[lh+\frac{m}{2\beta}\right]\right)\right]^j.
\end{aligned}
\end{equation}
By our concentration of measure estimate in Lemma~\ref{disordertail}, with $n=l$, 
$A=\frac{m}{2\beta}$ and $B=h$, there exists a $C>0$ (depending on $\beta,h$; see 
(\ref{mnbound}--\ref{Cdef})) such that 
\begin{equation}
\label{conmes}
\P\left(\sum_{k=1}^{l} \omega_k \leq -\left[lh+\frac{m}{2\beta}\right]\right)
\leq e^{-C(l+m)}.
\end{equation}
Combine (\ref{Asum}--\ref{conmes}), and pick $m_0$ large enough so that $em^4e^{-Cm/2} 
\leq 1$ for all $m>m_0$, to get
\begin{equation}
\sum_{m>m_0} \P(A(m)) \leq \sum_{m>m_0} \sum_{l\in\N} \sum_{j\in\N} 
e^{-C(l+m)j/2} < \infty.
\end{equation}
This proves the claim in (\ref{sumfin}).
\end{proof}


\subsection{Condition for finiteness of $\Phi_{\beta,h}(Q)$}
\label{appA.2}

\begin{lemma}
\label{finitephi} 
Fix $\beta,h>0$, $\rho\in\cP(\N)$ and $\nu\in\cP(\R)$. Then, for all $Q\in\cP^\mathrm{inv}
(\widetilde{\R}^\N)$ with $h(\pi_1 Q|q_{\rho,\nu})$ $<\infty$, there are finite constants 
$C>0$, $\gamma>\frac{2\beta}{C}$ and $K=K(\beta,h,\rho,\nu,\gamma)$ such that 
\begin{equation}
\Phi_{\beta,h}(Q)\leq \gamma\; h(\pi_1 Q|q_{\rho,\nu})+K.
\end{equation}
\end{lemma}

\begin{proof}
Abbreviate 
\begin{equation}
f(y) = \frac{d(\pi_1Q)}{dq_{\rho,\nu}}(y), 
\quad u(y) = -2\beta[\tau(y)h+\sigma(y)], 
\qquad y\in\widetilde{\R}=\cup_{n\in\N} \R^n.
\end{equation}
Fix $\gamma>2\beta/C$, with $C>0$ as in \eqref{Cdef}, 
and for $n, m\in\N$  define
\begin{equation} 
\begin{aligned}
A_{m,n} &= \{y\in\R^n\colon\, m-1\leq \g\log f(y)<m\},\\
A_{0,n} &= \{y\in\R^n\colon\, 0\leq f(y)<1\},\\
B_{m,n} &= \{y\in\R^n\colon\, m-1\leq u(y)<m\}.
\end{aligned}
\end{equation} 
Note that  
\begin{equation}
\R^n = A_{0,n} \cup \left[\cup_{m\in\N} A_{m,n}\right], \quad n\in\N,
\end{equation}
and that
\begin{equation} 
B_n = \bigcup_{m\in\N}\,B_{m,n}, \quad n\in\N,
\end{equation}
is the set of points $y\in\R^n$ for which $u(y)\geq0$. This gives rise to the decomposition
\begin{equation}
\begin{aligned}
\Phi_{\beta,h}(Q)
&= \sum_{n\in\N} \int_{\R^n} \log\left(\tfrac12\left[1+e^{u(y)}\right]\right)(\pi_1Q)(dy)\\
&\leq \sum_{n\in\N} \int_{\R^n} \log\left(1 \vee e^{u(y)}\right)(\pi_1Q)(dy)\\ 
&= \sum_{n\in\N}\sum_{m\in\N} \int_{B_{m,n}} u(y)f(y)\,q_{\rho,\nu}(dy)\\
&= I + II + III 
\end{aligned}
\end{equation}
with 
\begin{equation}\label{I-III}
\begin{aligned}
I &= \sum_{n\in\N}\sum_{m\in\N} \int_{ [\cup_{l\in\N_0} B_{m+l,n}] \cap A_{m,n} } 
u(y)f(y)\,q_{\rho,\nu}(dy)\\
II &= \sum_{n\in\N} \sum_{m\in\N} \int_{ A_{m,n} \cap [\cup_{l=1}^{m-1}B_{l,n}] } 
u(y)f(y)\,q_{\rho,\nu}(dy),\\
III &= \sum_{n\in\N} \int_{ A_{0,n} \cap [\cup_{m\in\N}\;B_{m,n}] } 
u(y)\,f(y)\,q_{\rho,\nu}(dy).
\end{aligned}
\end{equation}
The terms $I$ and $II$ deal with the set $B_n\cap \bigcup_{m\in\N}A_{m,n}$, while $III$ 
deals with the set $B_n\cap A_{0,n}$. Note that
\begin{equation}
\label{Iubs}
\begin{aligned}
I &\leq \sum_{n\in\N}\rho(n) \sum_{m\in\N} 
e^{m/\g}\sum_{l\in\N_0} (m+l)\,\Pr( B_{m+l,n}),\\
III &\leq \sum_{n\in\N} \rho(n) \sum_{m\in\N} m\,\Pr( B_{m,n}),
\end{aligned}
\end{equation}
where we recall that $\Pr=\nu^{\otimes\N}$. The upper bound on $I$ uses that $f\leq e^{m\,\g}$ 
on $A_{m,n}$ and $u<m$ on $B_{m,n}$. The upper bound on $III$ uses that $f\leq 1$ on $A_{0,n}$ 
and $u<m$ on $B_{m,n}$. We need to show that each of the three terms is finite. Observe from 
\eqref{Iubs} that $III \leq I$. Hence it suffices to show that $I$ and $II$ are finite.   

\medskip\noindent
$I$: Estimate
\begin{equation}
\begin{split}
I &\leq \sum_{n\in\N} \rho(n) \sum_{m\in\N} e^{m/\g} \sum_{l\in\N_0} (m+l)\,\Pr(B_{m+l,n})\cr
&\leq  \sum_{n\in\N} \rho(n) \sum_{m\in\N} e^{m/\g}\sum_{l\in\N_0} (l+m)
\,\Pr\left(\sum_{k=1}^n\o_k\leq-\left[nh+\frac{l+m-1}{2\beta}\right] \right)\cr
&\leq \sum_{n\in\N} \rho(n)\,e^{-Cn}  \sum_{m\in\N} e^{m/\g} 
\sum_{l\in\N_0} (l+m)\,\exp\left[-\frac{C(l+m-1)}{ 2\beta}\right] <\infty,
\end{split}
\end{equation}
where the third inequality follows from Lemma \ref{disordertail}, with $A=\frac{l+m-1}{2\beta}$,
$B=h$ and $C>0$ (depending on $\beta,h$; see (\ref{mnbound}--\ref{Cdef})).

\medskip\noindent
$II$: Use that $u(y)<m-1\leq \g\log f(y)$ for $y \in A_{m,n} \cap [\cup_{l=1}^{m-1} B_{l,n}]$, 
to estimate
\begin{equation}
\begin{split}
II &\leq \g \sum_{n\in\N} \sum_{m\in\N} \int_{ A_{m,n} \cap [\cup_{l=1}^{m-1} B_{l,n}] } 
f(y)\log f(y)\,q_{\rho,\nu}(dy)\cr
&\leq \g \sum_{n\in\N} \sum_{m\in\N} \int_{A_{m,n}} 
f(y)\log f(y)\,q_{\rho,\nu}(dy)\cr
&= \g \sum_{n\in\N} \int_{\R^n\backslash A_{0,n}}
f(y)\log f(y)\,q_{\rho,\nu}(dy) < \infty.
\end{split}
\end{equation}
The finiteness of the last term stems from the fact that
\begin{equation}
\label{hsum}
h(\pi_1Q|q_{\rho,\nu})= \sum_{n\in\N} \int_{\R^n\backslash A_{0,n}}
f(y)\log f(y)\,q_{\rho,\nu}(dy)+\sum_{n\in\N} \int_{ A_{0,n}}
f(y)\log f(y)\,q_{\rho,\nu}(dy)
\end{equation}
is assumed to be finite, while the second term in the right-hand side of \eqref{hsum}
lies in $[-1/e,0]$.
\end{proof}


\section{Application of Varadhan's lemma}
\label{appB}

This appendix settles \eqref{Sfinalpre} for $\beta,h>0$ and $g> 0$.

\begin{lemma}
\label{varlem}
For all $\beta,h>0$ and $g>0$, 
\begin{equation}
\label{vareq}
\bar S^\mathrm{que}(\beta,h;g)
=\sup_{Q\in\cC^{\rm fin}\cap\cR} \left[\Phi_{\beta,h}(Q)-gm_Q-I^\mathrm{ann}(Q)\right],
\end{equation}
where $\bar S^\mathrm{que}(\beta,h;g)$ is the $\omega$-a.s.\ constant limit defined in 
\eqref{Slim}.
\end{lemma}

\begin{proof}
Throughout the proof, $\beta,h>0$  and $g> 0$ are fixed. Note that, since $h(\pi_1Q
\mid q_{\rho,\nu}) \leq H(Q|q_{\rho,\nu}^{\otimes\N}) = I^\mathrm{ann}(Q)<\infty$, it 
follows from (\ref{Phidef}--\ref{phidef}) and Lemma~\ref{finitephi} that $\Phi_{\beta,h}
(Q)$ is finite on $\cC^{\rm fin}=\{Q\in\cP^\mathrm{inv}(\widetilde{E}^\N)\colon\,
I^\mathrm{ann}(Q)<\infty,\,m_Q<\infty\}$. 

\medskip\noindent
{\bf Lower bound:}
Because $\Phi_{\beta,h}$ is lower semi-continuous on $\cP^\mathrm{inv}(\widetilde{E}^\N)$ 
and finite on $\cC^{\rm fin}$, the set 
\begin{equation}
\label{aepsilon}
\cA_\epsilon = \big\{Q'\in\cP^\mathrm{inv}(\widetilde{E}^\N)\colon\,
\Phi_{\beta,h}(Q') > \Phi_{\beta,h}(Q)-\epsilon\big\}
\end{equation}
is open for every $Q\in\cC^{\rm fin}$ and $\epsilon>0$. Fix $Q\in\cC^{\rm fin}\cap\cR$ and $\epsilon>0$, and use 
(\ref{FNexpr}--\ref{Slim}) to estimate
\begin{equation}
\label{varlb1z=1}
\begin{aligned}
\bar{S}^\mathrm{que}(\beta,h;g)
&= \log \cN(g) + \limsup_{N\rightarrow\infty}\frac{1}{N} 
\log E_g^\ast\Big(e^{N\Phi_{\beta,h}(R_N^\omega)}\Big)\\
&\geq \log\cN(g) + \liminf_{N\rightarrow\infty}\frac{1}{N} 
\log E_g^\ast\Big(e^{N\Phi_{\beta,h}(R_N^\omega)}
\,1_{\cA_\epsilon}(R_N^\omega)\Big)\\
&\geq \log \cN(g) + \inf_{Q'\in \cA_\epsilon}\Phi_{\beta,h}(Q')
+ \liminf_{N\rightarrow\infty} \frac{1}{N}\log P_g^\ast(\cA_\epsilon)\\
&\geq \log \cN(g) + \inf_{Q'\in \cA_\epsilon} \Phi_{\beta,h}(Q')
-\inf_{Q'\in \cA_\epsilon} I_g^\mathrm{que}(Q')\\
&\geq \log \cN(g) + \Phi_{\beta,h}(Q)-I_g^\mathrm{que}(Q)-\epsilon,
\end{aligned}
\end{equation}
where in the third inequality we use the quenched LDP in Theorem~\ref{qLDP}. 
Next, note that $I^\mathrm{que}_g(Q)=I^\mathrm{ann}_g(Q)$ for $Q\in\cR$ by 
Theorem~\ref{qLDP} and $I^\mathrm{ann}_g(Q)=I^\mathrm{ann}(Q) + \log\cN(g) + g m_Q$ 
for $Q\in\cC^{\rm fin}$ by Lemma~\ref{rhozre}. Insert these
identities, take the supremum over $Q\in\cC^{\rm fin}\cap\cR$ and let $\epsilon\downarrow 0$, 
to arrive at the desired lower bound.

\medskip \noindent
{\bf Upper bound:}
The proof of the upper bound uses a truncation argument and comes in 4 steps. 

\medskip\noindent
{\bf 1.}
Abbreviate $\chi(y)=\log\phi_{\beta,h}(y).$ For $M>0$, let (compare with 
(\ref{Phidef}--\ref{phidef}))
\begin{equation}
\begin{split}
\Phi^M_{\beta,h}(Q) &= \int_{\widetilde{E}} (\widetilde\pi_1 Q)(dy)
\left[\chi(y) \wedge M\right],\cr
\bar\Phi^M_{\beta,h}(Q)&= \int_{\widetilde{E}}(\widetilde\pi_1 Q)(dy)\,
\chi(y)\,1_{\{\chi(y)>M\}}.
\end{split}
\end{equation}
Since $\phi_{\beta,h}\geq\tfrac12$, $Q\mapsto \Phi^{M}_{\beta,h}(Q)$ is bounded and 
continuous. Our goal will be to compare $\bar{S}^\mathrm{que}(\beta,h;g)$ with its 
truncated analogue (compare with (\ref{FNexpr}--\ref{Slim}))
\begin{equation}
\label{SMdef}
\bar{S}_M^\mathrm{que}(\beta,h;g) = \log \cN(g) + \limsup_{N\to\infty} \frac{1}{N}
\log E_g^\ast\big(e^{N\Phi^M_{\beta,h}(R_N^\omega)}\big),
\qquad M>0,
\end{equation}
and afterwards let $M\to\infty$. 
 
\medskip\noindent
{\bf 2.} Note that 
\begin{equation}
\label{phim}
\Phi_{\beta,h}(Q)-\bar\Phi^M_{\beta,h}(Q) \leq \Phi^M_{\beta,h}(Q).
\end{equation} 
Therefore, for any $g>0$ and any $-\infty<q<0<p<1$ with $p^{-1}+q^{-1}=1$, the reverse
of H\"older's inequality gives
\begin{equation}
\label{uperbs} 
\begin{split}
E_g^\ast\left(e^{N\Phi^M_{\beta,h}(R_N^\omega)}\right)
&\geq E_g^\ast\left(e^{N\Phi_{\beta,h}(R_N^\omega)}\,
e^{-N\bar\Phi^M_{\beta,h}(R_N^\o)}\right)\cr 
&\geq E_g^\ast\left(e^{pN\Phi_{\beta,h}(R_N^\omega)}\right)^{1/p}\, 
E_g^\ast\left(e^{-qN\bar\Phi^M_{\beta,h}(R_N^\o)}\right)^{1/q}\cr
&= E_g^\ast\left(e^{pN\Phi_{\beta,h}(R_N^\omega)}\right)^{1/p}\, 
E_0^\ast\left(e^{-qN\left[\bar\Phi^M_{\beta,h}(R_N^\o)
-[g/(-q)]\,m_{R_N^\o}\right]}\right)^{1/q}\,\cN(g)^{-N/q},
\end{split}
\end{equation}
where the equality uses \eqref{rhoz}, and
\begin{equation}
\begin{aligned}
&N\bar\Phi^M_{\beta,h}(R_N^\o)
= N\int_{\widetilde{E}}(\widetilde\pi_1R_N^\omega)(dy)\,\chi(y)\,1_{\{\chi(y)>M\}}
=\sum_{i=1}^N \chi(y_i)\,1_{\{\chi(y_i)>M\}},\\
&N m_{R_N^\o} = \sum_{i=1}^N \tau(y_i).
\end{aligned}
\end{equation} 
We next claim that $\omega$-a.s.\ there exists an $M'(\omega)<\infty$, depending on 
$\beta$, $h$, $g$ and $p$, such that
\begin{equation}
\label{estextra1}
E_0^\ast\left(e^{-qN\left[\bar\Phi^M_{\beta,h}(R_N^\o)
-[g/(-q)]\,m_{R_N^\o}\right]}\right) \leq 1 \qquad \forall\,M>M'(\omega).
\end{equation}
Indeed, $\{\chi(y_i)>M\} = \{-2\beta\sum_{k\in I_i} (\omega_k+h)>\log(2e^M-1)\}$,
and so we can repeat the argument in the proof of Lemma~\ref{mainlemma}, restricting
the estimates to $m$-values with $m\geq\log(2e^M-1)$. Clearly, there exists an 
$M_0<\infty$ such that $-[g/(-q)]Cm^2+(m+1)\leq 0$ for $m \geq M_0$. Therefore the claim 
in \eqref{estextra1} follows for any $M'(\omega)$ such that $\log(2e^{M'(\omega)}-1)
> M_0 \vee M(\omega)$ with $M(\omega)$ defined below \eqref{sumfin}. With this choice 
of $M'(\omega)$, the term $I$ in \eqref{ineq1} is absent, and we can estimate 
$\bar\Phi^M_{\beta,h}(R_N^\o)-[g/(-q)]\,m_{R_N^\o}\leq 0$ as in 
(\ref{Lambdadef}--\ref{ineq4}).

\medskip\noindent
{\bf 3.} We next apply Varadhan's lemma to \eqref{SMdef} using Theorem~\ref{qLDP} and the 
fact that $\Phi^M_{\beta,h}$ is bounded and continuous on $\cP^\mathrm{inv}(\widetilde{E}^\N)$. This gives
\begin{equation}
\label{SMubz<1}
\begin{aligned}
\bar{S}_M^\mathrm{que}(\beta,h;g) -\log\cN(g) 
&= \sup_{Q\in\cP^\mathrm{inv}(\widetilde{E}^\N)}
\left[\Phi^M_{\beta,h}(Q)-I_g^\mathrm{que}(Q)\right]\\
&= \sup_{Q\in\cR} \left[\,\Phi^M_{\beta,h}(Q)-
g\,m_Q-I^\mathrm{ann}(Q)\right]\\
&= \sup_{Q\in\cC^{\rm fin}\cap\cR}
 \left[\Phi^M_{\beta,h}(Q)-g\,m_Q-I^\mathrm{ann}(Q)\right]\\
&\leq \sup_{Q\in\cC^{\rm fin}\cap\cR}
 \left[\Phi_{\beta,h}(Q)-g\,m_Q-I^\mathrm{ann}(Q)\right]\\
&= S^\mathrm{que}(\beta,h;g), 
\end{aligned}
\end{equation} 
where the second equality uses \eqref{Iannz} and \eqref{eqgndefinitionIalgz},
and the third equality uses that $\Phi^M_{\beta,h} \leq M < \infty$ in combination 
with the fact that the $Q$'s with $I^\mathrm{ann}(Q)=\infty$ or $m_Q=\infty$ do 
not contribute to the supremum. The inequality uses that $\Phi^M_{\beta,h} \leq 
\Phi_{\beta,h}$. Combining \eqref{SMdef} and (\ref{uperbs}--\ref{SMubz<1}), and 
letting $N\to\infty$ followed by $M\to\infty$, we get 
\begin{equation}
\label{smlim}
\frac{1}{p}\log\cN(g)+\limsup_{N\to\infty}\frac{1}{pN}\log E_g^\ast
\left(e^{pN\Phi_{\beta,h}(R_N^\omega)}\right)\\
\leq S^\mathrm{que}(\beta,h;g).
\end{equation}

\medskip\noindent
{\bf 4.}
It remains to show that the left-hand side of \eqref{smlim} tends to $\bar S^\mathrm{que}
(\beta,h;g)$ as $p \uparrow 1$. Define
\begin{equation}
S^{\beta,h}(p) = \limsup_{N\to\infty}\frac{1}{N}\log E_g^\ast
\left(e^{pN\Phi_{\beta,h}(R_N^\omega)}\right), \qquad p \geq 0.
\end{equation} 
Clearly, $p \mapsto S^{\beta,h}(p)$ is non-decreasing and convex on $(0,\infty)$.
Moreover,
\begin{equation}
S^{\beta,h}(p) \leq \left\{\begin{array}{ll}
pS^{\beta,h}(1)  &\mbox{for } p\in (0,1],\\
S^{p\beta,h}(1)  &\mbox{for } p\in [1,\infty).
\end{array}
\right.
\end{equation}
The first line follows from Jensen's inequality, the second line from the fact that
$p\Phi_{beta,h} \leq \Phi_{p\beta,h}$ for $p\in [1,\infty)$ (recall \eqref{phidef}). 
We know from the remark made at the end of Section~\ref{S6.1} that $S^{\beta,h}(1) 
= \bar{S}^\mathrm{que}(\beta,h;g)-\log\cN(g)<\infty$ because $g>0$. Therefore $p
\mapsto S^{\beta,h}(p)$ is finite on $(0,\infty)$ and, by convexity, is continuous 
on $(0,\infty)$.   
\end{proof}


\section{Continuity at $g=0$}
\label{appC}

In this appendix we prove \eqref{Sin1}. The key is the following proposition
relating the two quenched LDP's in Theorem~\ref{qLDP}. Recall 
(\ref{trunword}--\ref{Rdef}), and abbreviate $\cR^\mathrm{fin} = \{Q\in\cR\colon\,
m_Q<\infty\}$.

\begin{proposition}
\label{prop:connvarprinc}
Suppose that $E$ is finite. Then for every $Q\in\cP^\mathrm{inv}(\widetilde{E}^\N)$ 
there exists a sequence $(Q_n)$ in $\cR^\mathrm{fin}$ such that $\lim_{n\to\infty}
I^\mathrm{ann}(Q_n)=I^\mathrm{que}(Q)$. 
\end{proposition}

\begin{proof}
The proof is \emph{not self-contained}, because it uses the approximation argument in 
Birkner, Greven and den Hollander~\cite{BiGrdHo10}, Sections~3--4 (this argument was 
also exploited in Cheliotis and den Hollander~\cite{ChdHo10}, Appendix~B). For 
simplicity we pretend that the support of $\rho$ is $\N$. The proof is easily 
extended to $\rho$ with infinite support.

\medskip\noindent
{\bf 1.}
We first assume that $Q \in \cP^\mathrm{erg,fin}(\widetilde{E}^\N)$ with
\begin{equation} 
\cP^\mathrm{erg,fin}(\widetilde{E}^\N)
= \big\{Q\in\cP^\mathrm{inv}(\widetilde{E}^\N)\colon\,Q \mbox{ is ergodic},\,
m_Q<\infty\big\}.
\end{equation} 
For $M \in \N$ and $\epsilon_1>0$, choose 
\begin{equation}
\label{ABdef}
\cA = \{z_a\colon\,a=1,\dots,A\} \subset \widetilde{E}^M,
\qquad \cB = \{\zeta^{(b)}\colon\, b=1,\dots,B\} = \kappa(\cA),
\end{equation} 
as in \cite{BiGrdHo10}, Equations (3.5--3.6), satisfying also \cite{BiGrdHo10}, 
Equation (4.2--4.3) for a small neighbourhood of $Q$. Each element of $\cB 
\subset \widetilde{E}$ consists of approximately $L=M m_Q$ letters (for simplicity 
we pretend that each $b \in \cB$ has precisely $L$ letters). Cut $X$ 
into $L$-blocks, and let 
\begin{equation}
G_j = 1_{\{\text{an element of $\cB$ appears in $X|{}_{((j-1)L,jL]}$}\}}.
\end{equation}
Note that $(G_j)$ are i.i.d.\ Bernoulli($p$) random variables with
\begin{equation}
p = p(M,\epsilon_1) = \exp\big[-MH(\Psi_Q \mid \nu^{\otimes\N})\,[1+o(1)]\big],
\qquad M\to\infty,\,\epsilon_1 \downarrow 0.
\end{equation} 
Therefore 
\begin{equation}
\sigma_1=\min\{j\in\N\colon\, G_j = 1\}
\end{equation} 
is geometrically distributed  with success probability $p$. Put $\tilde{Y}_1
=\kappa(X|_{(0,(\sigma_1-1)L]})$ (and make a trivial modification when 
$\sigma_1=1$ to avoid an empty word later on). Given $X|_{((\sigma_1-1)L,
\sigma_1 L]} = \zeta^{(b)} \in \cB$, let 
\begin{equation}
(\tilde{Y}_2,\dots,\tilde{Y}_{M+1})=z_a
\end{equation} 
be a suitably drawn random element of $\cA$ ($a$ is drawn uniformly from 
$\{a'\colon\,\kappa(z_{a'})=\zeta^{(b)}\}$). Repeating this construction, we 
obtain a random sequence $\tilde{Y}=(\tilde{Y}_j)$ in $\widetilde{E}^\N$. 
Denote the law of this random sequence by $\tilde{Q}_{M,\epsilon_1}$. Note 
that, by construction, $\kappa(\tilde{Y})=X$, so that $\tilde{Q}_{M,\epsilon_1}
\in \cR$, and that the consecutive $(M+1)$-blocks 
\begin{equation}
(\tilde{Y}_{(k-1)(M+1)+1},\dots,\tilde{Y}_{k(M+1)})_{k\in\N}
\end{equation} 
form an i.i.d.\ sequence (in particular, $\tilde{Q}_{M,\epsilon_1}$ is 
mixing and has finite mean word lenghts). Furthermore, $\tilde{Y}_1$ and 
$(\tilde{Y}_2,\dots,\tilde{Y}_{M+1})$ are independent. Let $\hat{Q}_{M,\epsilon_1}$ 
be the shift-invariant version of $\tilde{Q}_{M,\epsilon_1}$ obtained by 
randomizing the position of the origin. Then 
\begin{equation}
\hat{Q}_{M,\epsilon_1} \in \cR^\mathrm{fin}.
\end{equation} 
By construction, $\hat{Q}_{M,\epsilon_1} \to Q$ weakly as $M\to\infty$ and 
$\epsilon_1 \downarrow 0$. 

\medskip\noindent
{\bf 2.} It remains to check that 
\begin{equation}
\label{eq:claimerg}
I^\mathrm{ann}(\hat{Q}_{M,\epsilon_1}) \to 
I^{\mathrm{que}}(Q), \qquad M\to\infty,\,\epsilon_1 \downarrow 0.
\end{equation} 
Since $\hat{Q}_{M,\epsilon_1}$ is the shift-invariant mean of $\tilde{Q}_{M,\epsilon_1}$, 
we have 
\begin{equation} 
\label{eq:decompHtildeQ}
H(\hat{Q}_{M,\epsilon_1} \mid q_{\rho,\nu}^{\otimes\N}) 
= H(\tilde{Q}_{M,\epsilon_1} \mid q_{\rho,\nu}^{\otimes\N}) 
= \frac1{M+1} \Big[h\big(\cL(\tilde{Y}_1) \mid q_{\rho,\nu}\big) 
+ h\big(\cL(\tilde{Y}_2,\dots,\tilde{Y}_{M+1}) \mid 
q_{\rho,\nu}^{\otimes M}\big)\Big],
\end{equation}
where the second equality uses the special block structure of $\tilde{Q}_{M,\epsilon_1}$. 
By construction, we have 
\begin{equation} 
h\big(\cL(\tilde{Y}_2,\dots,\tilde{Y}_{M+1}) \mid q_{\rho,\nu}^{\otimes M}\big)
\in MH(Q^* \mid q_{\rho,\nu}^{\otimes \N}) 
+ (-4\epsilon_1 M, 4\epsilon_1 M)
\end{equation}
(see \cite{BiGrdHo10}, Equations (3.6) and (3.8)). Furthermore, 
\begin{equation} 
\label{eq:firstcontr}
h\big(\cL(\tilde{Y}_1) \mid q_{\rho,\nu}\big) 
\in M (\alpha-1) H(\Psi_Q \mid \nu^{\otimes \N}) + (-\delta M, \delta M),
\end{equation}
where $\delta \downarrow 0$ as $\epsilon_1 \downarrow 0$. To see why the latter holds, 
note that 
\begin{equation}
\label{hform1}
\begin{aligned} 
h\big(\cL(\tilde{Y}_1) \mid q_{\rho,\nu}\big)
&= \sum_{t=0}^\infty \sum_{x_1,\dots,x_{tL} \in E \atop 
\text{no $L$-block from $\cB$}} 
\frac{p(1-p)^t \prod_{k=1}^{tL}\nu(x_k)}{(1-p)^t} \log\left[\frac{p(1-p)^t 
\frac{\prod_{k=1}^{tL}\nu(x_k)}{(1-p)^t}}{\rho(tL)\prod_{k=1}^{tL}\nu(x_k)}\right] \\
&= \sum_{t=0}^\infty p \sum_{x_1,\dots,x_{tL} \in E \atop 
\text{no $L$-block from $\cB$}} 
\Big(\prod_{k=1}^{tL}\nu(x_k)\Big)
\log\left[\frac{p}{\rho(tL)}\right]\\
&= \sum_{t=0}^\infty p(1-p)^t\,\log\left[\frac{p}{\rho(tL)}\right] \\
&= \log p 
- \sum_{t=0}^\infty p (1-p)^t \log(tL)\,\frac{\log(\rho(tL))}{\log(tL)}\\ 
&= \log p + \alpha[1+o(1)] \sum_{t=0}^\infty p (1-p)^t \log(tL),
\qquad L\to\infty.
\end{aligned}
\end{equation}
Finally, note that $\log L = \log (Mm_Q) = O(\log M) = o(M)$ as $M\to\infty$ and
\begin{equation}
\label{hform2}
\sum_{t=0}^\infty p (1-p)^t \log t 
= -\log p\,+\, \sum_{t=0}^\infty p (1-p)^t\,\log(tp) 
= -\log p\,+\, \int_0^\infty e^{-y}\,\log y\,dy 
\,+\, o(1), \quad p \downarrow 0,
\end{equation}
where the integral equals minus Euler's constant. Since $-\log p \in M H(\Psi_Q 
\mid \nu^{\otimes\N}) + [-\delta M,\delta M]$, (\ref{hform1}--\ref{hform2}) combine 
to yield (\ref{eq:firstcontr}). Clearly, (\ref{eq:decompHtildeQ}--\ref{eq:firstcontr})
imply (\ref{eq:claimerg}), which completes the proof for $Q \in \cP^\mathrm{erg,fin}
(\widetilde{E}^\N)$. 

\medskip\noindent
{\bf 3.}
If $Q \in \cP^\mathrm{inv}(\widetilde{E}^\N)$ is ergodic with $m_Q=\infty$, then we 
approximate $Q$ by$[Q]_\tr$ (recall \eqref{trunword}), approximate each $[Q]_\tr$ 
from inside $\cR^\mathrm{fin}$ as above, and then diagonalize the approximation 
scheme. This yields the claim because $[Q]_\tr \to Q$ weakly and $I^\mathrm{que}([Q]_\tr) 
\to I^{\mathrm{que}}(Q)$ as $\tr\to\infty$ (recall (\ref{truncapproxcont})). 
Finally, if $Q \in \cP^\mathrm{inv}(\widetilde{E}^\N)$ is not ergodic, then 
we first approximate its ergodic decomposition by a finite sum and afterwards 
approximate each summand as above (similarly as in \cite{Bi08}, proof of 
Proposition~2, and \cite{BiGrdHo10}, proof of Proposition~4.1).
\end{proof}

We are now ready to prove (\ref{Sin1}). 

\medskip\noindent
$\bullet$ \underline{1st inequality}:
In \eqref{Sdef} we defined
\begin{equation}
\label{Scomp}
S^\mathrm{que}(\beta,h;g)
= \sup_{\cC^\mathrm{fin}\cap\cR} [\Phi_{\beta,h}(Q)-gm_Q-I^\mathrm{ann}(Q)].
\end{equation}
As shown in Appendix~\ref{appB} (recall \eqref{Sfinalpre}),
\begin{equation}
\label{Scomp1}
\bar{S}^\mathrm{que}(\beta,h;g) = S^\mathrm{que}(\beta,h;g) \qquad \forall\,g>0.
\end{equation}
Let $(Q_n)$ be any sequence in $\cC^\mathrm{fin}\cap\cR$ such that 
\begin{equation}
\Phi_{\beta,h}(Q_n)-I^\mathrm{ann}(Q_n) \geq 
S^\mathrm{que}(\beta,h;0)-\tfrac{1}{n}.
\end{equation}
By choosing $g=g_n=1/n\max\{m_{Q_1},\dots,m_{Q_n}\}$ in (\ref{Scomp}--\ref{Scomp1}), 
we get
\begin{equation}
\bar{S}^\mathrm{que}(\beta,h;g_n) \geq S^\mathrm{que}(\beta,h;0)-\tfrac2n,
\end{equation}
which yields $\bar{S}^\mathrm{que}(\beta,h;0+) \geq S^\mathrm{que}(\beta,h;0)$
after letting $n\to\infty$.

\medskip\noindent
$\bullet$ \underline{2nd inequality}:
Recall that $E=\R$. For $M\in\N$, let
\begin{equation}
D_M = \big\{-M,-M+1/M,\dots,M-1/M,M\big\}
\end{equation} 
be the grid of spacing $1/M$ in $[-M,M]$, which serves as a finite set of letters
approximating $E$. Let $\widetilde{D}_M = \cup_{n\in\N} D_M^n$ be the set of finite 
words drawn from $D_M$. Let $T_M\colon\,E \to D_M$ be the letter map
\begin{equation}
\label{TMdef}
T_M(x) = \left\{\begin{array}{ll}
M &\mbox{ for } x \in [M,\infty),\\
M \lceil x/M \rceil &\mbox{ for } x \in (-M,M),\\
-M &\mbox{ for } x \in (-\infty,-M],
\end{array}
\right.
\end{equation}
and $\widetilde{T}_M\colon\,\widetilde{E} = \cup_{k\in\N} E^k \to [\widetilde{D}_M]_M 
= \cup_{k=1}^M D_M^k$ the word map
\begin{equation}
\widetilde{T}_M(y) = \widetilde{T}_M(x_1,\dots,x_m)
= (T_Mx_1,\dots,T_Mx_{m \wedge M}), \qquad m\in\N,\,x_1,\dots,x_m\in E.
\end{equation}
For $Q\in\cP^\mathrm{inv}(\widetilde{D}_M^\N)$, define (compare with 
(\ref{Phidef}--\ref{phidef}))
\begin{equation}
\Phi^M_{\beta,h}(Q) = \int_{\widetilde{D}_M} (\widetilde\pi_1 Q)(dy) 
\log \phi^M_{\beta,h}(y),
\end{equation}
where, for $y \in \widetilde{D}_M^\N$,
\begin{equation}
\label{phiMphi}
\phi^M_{\beta,h}(y) = \left\{\begin{array}{ll}
\phi_{\beta,h}(y) &\mbox{ for } y=(x_1,\dots,x_m) \in 
[D_M\setminus \{-M,M\}]^m,\,m=1,\dots,M-1,\\
\tfrac12 &\mbox{ otherwise.}
\end{array}
\right.
\end{equation}

Next, let $I_M^\mathrm{que}\colon\,\cP^\mathrm{inv}(\widetilde{D}_M^\N) \to [0,\infty]$ 
and $I_M^\mathrm{ann}\colon\,\cP^\mathrm{inv}(\widetilde{D}_M^\N) \to [0,\infty]$ be the 
quenched, respectively, annealed rate function when the disorder distribution is $\nu_M$ 
given by $\nu_M = \nu \circ T_M^{-1}$ and the word length distribution is $\rho_M$ given by 
$\rho_M(m) = \rho(m)$ for $m=1,\dots,M-1$ and $\rho(M) = \sum_{m\geq M} \rho(m)$. Define
\begin{equation}
\cC_M^\mathrm{fin} = \big\{Q\in\cP^\mathrm{inv}(\widetilde{D}_M^\N)\colon\,
I_M^\mathrm{que}(Q)<\infty,\,m_Q<\infty\big\}.
\end{equation} 
We know from \eqref{jexcont}, \eqref{FNdef} and \eqref{Slim} that $\bar{S}^\mathrm{que}
(\beta,h;g)$ is non-increasing as a function of the disorder distribution $\nu$. 
Inside the interval $(-M,M)$ the map $T_M$ moves points upwards, while $\phi_{\beta,h}(y) 
\geq \tfrac12$, $y\in\widetilde{E}$. We therefore see from \eqref{phiMphi} that 
$\phi_{\beta,h}(y) \geq \phi^M_{\beta,h}(\widetilde{T}_My)$, $y\in\widetilde{E}$.
Hence $\bar{S}^\mathrm{que}(\beta,h;g)$ is bounded from below by its analogue 
$\bar{S}_M^\mathrm{que}(\beta,h;g)$ with $\nu$ replaced by $\nu_M$, $\rho$ by $\rho_M$ 
and $\phi_{\beta,h}$ by $\phi^M_{\beta,h}$. It therefore follows from \eqref{Scomp1}
and the first inequality that
\begin{equation}
\label{Scomp2}
\begin{aligned}
\bar{S}^\mathrm{que}(\beta,h;0+) &\geq \bar{S}_M^\mathrm{que}(\beta,h;0+)
\geq S_M^\mathrm{que}(\beta,h;0)\\
&= \sup_{Q\in\cC_M^\mathrm{fin} \cap \cR}
\big[\Phi^M_{\beta,h}(Q)-I_M^\mathrm{ann}(Q)\big]
= \sup_{Q\in\cC_M^\mathrm{fin}}
\big[\Phi^M_{\beta,h}(Q)-I_M^\mathrm{que}(Q)\big],
\end{aligned}
\end{equation}
where the last equality in \eqref{Scomp2} uses Proposition~\ref{prop:connvarprinc} in 
combination with the fact that $D_M$ is finite and $\Phi^M_{\beta,h}$ is bounded and 
continuous on $\cC_M^\mathrm{fin}$ (note that $I_M^\mathrm{ann}=I_M^\mathrm{que}$ on
$\cC_M^\mathrm{fin} \cap \cR$ by \eqref{Requiv}). For $Q\in \cP^\mathrm{inv}(\widetilde{E}^\N)$, 
let $[Q]_M = Q \circ (\widetilde{T}_M^\N)^{-1}$. Then the right-hand side of 
\eqref{Scomp2} equals 
\begin{equation}
\label{Scomp3}
\sup_{Q\in\cC^\mathrm{fin}} \big[\Phi^M_{\beta,h}([Q]_M)
-I_M^\mathrm{que}([Q]_M)\big].
\end{equation}

Next, $m_{[Q]_M} \leq m_Q$, $\widetilde{T}_M$ is a projection, and relative 
entropies are non-increasing under the action of a projection. Recalling 
(\ref{eqgndefinitionIalg}--\ref{eqnratefctexplicitalg}), we therefore have 
$I_M^\mathrm{que}([Q]_M) \leq I^\mathrm{que}(Q)$ for all $Q\in\cP^\mathrm{inv}
(\widetilde{E}^\N)$ and $M\in\N$. Hence (\ref{Scomp2}--\ref{Scomp3}) combine to 
give
\begin{equation}
\label{Scomp4}
\bar{S}^\mathrm{que}(\beta,h;0+) \geq \sup_{Q\in\cC^\mathrm{fin}} 
\big[\Phi^M_{\beta,h}([Q]_M)-I^\mathrm{que}(Q)\big].
\end{equation}
Finally, because $\lim_{M\to\infty} [Q]_M=Q$ weakly for all $Q\in\cP^\mathrm{inv}
(\widetilde{E}^\N)$ and $\lim_{M\to\infty}\phi^M_{\beta,h}(y)=\phi_{\beta,h}(y)$ 
for all $y\in\widetilde{E}$, Fatou's lemma tells us that $\lim_{M\to\infty} 
\Phi^M_{\beta,h}([Q]_M) \geq \Phi_{\beta,h}(Q)$. Hence we arrive at (recall 
\eqref{Sdefalt})
\begin{equation}
\bar{S}^\mathrm{que}(\beta,h;0+) \geq \sup_{Q\in\cC^\mathrm{fin}} 
[\Phi_{\beta,h}(Q)-I^\mathrm{que}(Q)] = S_*^\mathrm{que}(\beta,h).
\end{equation}


\section{Concentration of measure estimates for the disorder}
\label{appD}

First we introduce some notation. After that we state and prove the concentration of 
measure estimate for the disorder $\omega$ that was used in the proof of 
Lemmas~\ref{mainlemma} and \ref{finitephi}  (Lemmas~\ref{disordertail}--\ref{lbF} below). 

Recall \eqref{mgffin}. The cumulant generating function $\lambda \mapsto M(\lambda)$ 
is analytic, non-negative and strictly convex on $\R$, with $M(0)=M'(0)=0$ (recall 
(\ref{nuzmuv})). In particular, $G=M'$ and its inverse $H=G^{-1}$ are both analytic 
and strictly increasing on $[0,\infty)$. 

For $W,x>0$, define
\begin{equation}
\label{fdef}
\begin{split}
f_{W,x}(\l)= x\,\left[M(\l)-\l\,\tfrac{W}{x}\right], \qquad \lambda \in \R,
\end{split}
\end{equation}
and note that $\l \mapsto f_{W,x}(\l)$ is strictly convex on $\R$, with $f_{W,x}(0)
= 0$ and $f_{W,x}'(0)=-W<0$. Putting
\begin{equation}
\label{pdef}
\chi = \lim_{\l\to\infty} G(\l) \in (0,\infty]
\end{equation}
(which, by (\ref{mgffin}), equals the supremum of the support of the law of $-\omega_1$),
we have $\lim_{\l\to\infty} f_{W,x}(\l)/\l$ $=\chi x-W$, and so there are two cases:
\begin{enumerate}
\item[(I)] 
If $\frac{W}{x}\leq\chi$, then $f_{W,x}$ has a unique minimizer at some $\l_*\in (0,\infty]$.
Note that $\l^*=\infty$ if and only if  $\frac{W}{x}=l_*$. 
\item[(II)]
If $\frac{W}{x}>\chi$, then $f_{W,x}$ attains it minimum at infinity. In this case 
$f_{W,x}(\infty)=-\infty$, since 
\begin{equation}
-\l\,W\leq f_{W,x}(\l)=-\l\,x\left[\frac{W}{x}-\frac{M(\l)}{\l}\right]
\leq -\l\,x\left[\frac{W}{x}-\chi\right],
\end{equation}
where we use that $0\leq \frac{M(\l)}{\l}\leq\chi$.
\end{enumerate}
In case (I), we have
\begin{equation}
\label{lamstar}
\l_* = \l_*(W,x) = H(\tfrac{W}{x}), \qquad
f_{W,x}(\l_*) = -x\,\left[\l_*\,G(\l_*)-M(\l_*)\right].
\end{equation}
Since $H(y)$ is well defined only for $y\leq\chi$, in what follows we will always assume 
that the arguments of $H$ are at most $\chi$.

Our concentration of measure estimate is the following. Let
\begin{equation}
\label{Fdef}
F(\l) = \l\,G(\l)-M(\l), \qquad \lambda\in [0,\infty).
\end{equation}

\begin{figure}[htbp]
\vspace{3cm}
\begin{center}
\begin{picture}(14,9)(0,-1.5)
\setlength{\unitlength}{0.45cm}
\put(-7,0){\line(7,0){14}}
\put(0,-3){\line(0,7){9}}
{\thicklines
\qbezier(0,0)(3,.3)(6,5)
\qbezier(0,0)(-3,.3)(-6,5)
}
\qbezier[40](0,-2)(3,1.6)(7,5.1)
\qbezier[10](3.3,0) (3.3,0.6) (3.3,1.5)
\put(-.8,.6){$0$}  
\put(7.5,-0.2){$\lambda$}
\put(-1,6.7){$M(\lambda)$}
\put(3,-1.1){$\bar\lambda$}
\put(3.3,1.6){\circle*{.3}}
\put(-2.8,-2.2){$-F(\bar\lambda)$}
\end{picture}
\end{center}
\vspace{.4cm}
\begin{center}
\caption{\small Qualitative picture of $\lambda\mapsto M(\lambda)$. 
The slope at $\bar\lambda$ equals $G(\bar\lambda)$.}
\end{center}
\label{fig-mofl}
\end{figure}
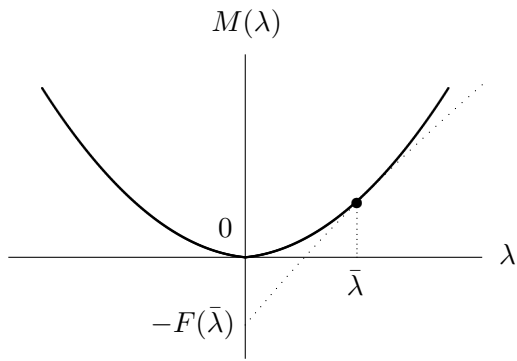
\vspace{-1cm}

\begin{lemma}
\label{disordertail}
For $n\in\N$ and $A,B>0$, 
\begin{equation}
\label{mnbound}
\P\left(\sum_{k=1}^n\o_k\leq -A-nB\right)
\left\{\begin{array}{ll}
\leq \exp\left[-n\,F(H(\tfrac{A}{n}+B))\right] &\mbox{ when } A/n+B\leq \chi,\\ 
= 0 &\mbox{ when } A/n+B > \chi,
\end{array}
\right.
\end{equation}
where 
\begin{equation}\label{mnbound1}
n\,F(H(\tfrac{A}{n}+B)) \geq  C(A+n),\quad \text{ when } A/n+B\leq \chi,
\end{equation} 
with 
\begin{equation}
\label{Cdef}
C=\tfrac12[F(H(B))\wedge F(H(1))]>0.
\end{equation}
\end{lemma}

\begin{proof}
Estimate
\begin{equation}
\begin{split}
\Pr\left(\sum_{k=1}^n\o_k\leq -W\right)
&= \inf_{\l>0} \Pr\left(e^{-\l\sum_{k=1}^n\o_k} \geq e^{\l\,W}\right)\cr
&\leq \inf_{\l>0} e^{-\l\,W}\,\Big[\E\big(e^{-\l\o_1}\big)\Big]^n
= \inf_{\l>0} e^{-\l\,W+n\,M(\l)} = e^{\inf_{\l>0} f_{W,n}(\l)}
\end{split}
\end{equation}
with $\l\mapsto f_{W,n}(\l)$ the function defined in \eqref{fdef}. In Case (I),
(\ref{lamstar}) shows that the minimal value of $f_{W,n}$ is $-nF(\l_*(W,n))
=-nF(H(\frac{W}{n}))$. Together with the lower bound on $nF(H(\frac{W}{n}))$ that 
is derived in Lemma~\ref{lbF} below, this proves the first line of (\ref{mnbound})
with the estimates in (\ref{mnbound1}--\ref{Cdef}). In Case (II), $f_{W,n}$ attains 
its infimum at infinity, with $f_{W,n}(\infty)=-\infty$, which proves the second 
line of (\ref{mnbound}).
\end{proof}

\begin{lemma}
\label{lbF}
For every $A,B>0$ and $x\in[1,\infty)$ with $A/x+B\leq \chi$ there exists a $C>0$ 
(depending on $B$ only) such that
\begin{equation}
\label{lbFm}
x\,F(H(\tfrac{A}{x}+B)) \geq C(A+x), \qquad x\in[1,\infty).
\end{equation}
\end{lemma}

\begin{proof} 
For $x \geq A$, estimate
\begin{equation}
 x\,F(H(\tfrac{A}{x}+B))
\geq x\,F(H(B)) \geq \tfrac12(A+x)\,F(H(B)). 
\end{equation}
For $x \leq A$, on the other hand, estimate  
\begin{equation}
x\,F(H(\tfrac{A}{x}+B))
\geq A\,(\tfrac{A}{x})^{-1}\,F(H(\tfrac{A}{x}))
\geq A\,F(H(1)) \geq \tfrac12(A+x)\,F(H(1)),
\end{equation}
where the second inequality uses that $y\mapsto y^{-1}F(H(y))$ is strictly increasing
on $(0,\chi)$. Combining the two estimates, we get the claim with $C$ given by (\ref{Cdef}).
\end{proof}



\end{document}